\documentclass[11pt, a4paper]{amsart}

\usepackage{amssymb}
\usepackage{amsrefs}
\usepackage{amsmath}
\usepackage{latexsym}
\usepackage{amscd}
\usepackage{amsfonts}
\usepackage{pb-diagram}
\usepackage[all]{xypic}
\usepackage[mathcal,mathscr]{eucal}
\usepackage{amsthm}
\usepackage{graphicx}
\usepackage{url}
\usepackage{pdfsync}
\usepackage{bbm}
\usepackage{xypic}
\usepackage[pdftex,colorlinks]{hyperref}
\hypersetup{linkcolor = blue, citecolor=blue}
\usepackage{verbatim}

\usepackage{pifont}
\usepackage{amsbsy}
\usepackage{epstopdf}
\usepackage{times}
\usepackage{epic,eepic}
\usepackage{esint}
\usepackage{tikz}
\usepackage{pgfplots}

\newtheorem{theorem}{Theorem}[section]
\newtheorem{lemma}{Lemma}[section]
\newtheorem{definition}{Definition}[section]
\newtheorem{corollary}{Corollary}[section]

\theoremstyle{definition}
\newtheorem{remark}{Remark}[section]

\newcommand{\intav}[1]{\mathchoice {\mathop{\vrule width 6pt height 3 pt depth  -2.5pt
\kern -8pt \intop}\nolimits_{\kern -6pt#1}} {\mathop{\vrule width
5pt height 3  pt depth -2.6pt \kern -6pt \intop}\nolimits_{#1}}
{\mathop{\vrule width 5pt height 3 pt depth -2.6pt \kern -6pt
\intop}\nolimits_{#1}} {\mathop{\vrule width 5pt height 3 pt depth
-2.6pt \kern -6pt \intop}\nolimits_{#1}}}

\newcommand{\intavl}[1]{\mathchoice {\mathop{\vrule width 6pt height 3 pt depth  -2.5pt
\kern -8pt \intop}\limits_{\kern -6pt#1}} {\mathop{\vrule width 5pt
height 3  pt depth -2.6pt \kern -6pt \intop}\nolimits_{#1}}
{\mathop{\vrule width 5pt height 3 pt depth -2.6pt \kern -6pt
\intop}\nolimits_{#1}} {\mathop{\vrule width 5pt height 3 pt depth
-2.6pt \kern -6pt \intop}\nolimits_{#1}}}

\title[ Fine behavior at critical points]{Global $C^{1,\alpha}$ regularity and fine behavior at critical points for quasilinear elliptic equations with Orlicz growth and unbounded drift}
\author{Jos\'e Erivamberto Lima Oliveira }
\date{} 
\begin{document}

\address{José Erivamberto Lima Oliveira \newline
Faculdade de Filosofia Dom Aureliano Matos, Universidade Estadual do Cear\'a \newline
Campus de Limoeiro do Norte, Avenida
 Dom Aureliano Matos, 2058, Limoeiro do Norte, Cear\'a, Brazil.}
\email{erivamberto.oliveira@uece.br}

\thanks{Corresponding author: José Erivamberto Lima Oliveira - erivamberto.oliveira@uece.br}


\begin{abstract}
We investigate the global regularity up to the boundary and the fine pointwise behavior along the critical set for bounded weak solutions to general quasilinear elliptic Dirichlet problems in Orlicz--Sobolev spaces. Our structural framework encompasses operators with non-standard Lieberman-type growth alongside lower-order terms featuring natural gradient growth, an unbounded source term in $L^q$, and a unbounded spatial drift weight in $L^s$. To overcome the analytical obstructions posed by the unbounded drift across the anisotropic growth gap, we first establish global Hölder continuity up to the boundary via Moser iteration. We then construct a novel smooth regularization scheme for general homogeneous constant-coefficient operators—merging the approximation frameworks of Lieberman and Cianchi--Maz'ya—to derive global $C^{1,\alpha}$ regularity up to the boundary through a variational perturbation approach. Building upon this global theory, we employ a geometric compactness and blow-up argument to prove that at any interior or boundary critical point where the gradient vanishes, the solution achieves an improved pointwise $C^{1,\alpha}$ regularity. Remarkably, at boundary critical  points, this quantitative improvement decouples from the lower-order terms and completely bypasses both the boundary datum regularity and the domain smoothness. Finally, we discuss the sharpness of the integrability threshold required for the unbounded drift weight in the subcritical regime.
\end{abstract}

\maketitle


\section{Introduction}

In this paper, we investigate both the global regularity up to the boundary and the fine pointwise behavior on the critical set for weak solutions of the general quasilinear elliptic Dirichlet boundary value problem in an open, bounded domain $\Omega \subset \mathbb{R}^n$ ($n \ge 2$):
\begin{equation}\label{1}
\begin{cases}
\operatorname{div} \mathcal{A}(x, u, \nabla u) + \mathcal{B}(x, u, \nabla u) = 0 \quad &\text{in}\quad \Omega,\\
u=\psi \quad &\text{on} \quad \partial\Omega,
\end{cases}
\end{equation}
where $\psi\in C^{1,\alpha_\psi}(\partial\Omega)$ (or $\psi \in C^{\alpha_\psi}(\partial\Omega)$). Here, the principal part $\mathcal{A}: \Omega \times \mathbb{R} \times \mathbb{R}^n \to \mathbb{R}^n$ and the lower-order term $\mathcal{B}: \Omega \times \mathbb{R} \times \mathbb{R}^n \to \mathbb{R}$ are Carath\'eodory functions. We assume that the operator's growth and ellipticity are governed by a density function $g \in C([0, +\infty)) \cap C^1((0, +\infty))$ satisfying the structural condition:
\begin{equation}\label{01}
    0 < \delta \leq\frac{t g'(t)}{g(t)} \leq g_0, \quad \text{for all} \quad t>0.
\end{equation}
Letting $G(t) = \int_0^t g(s)ds$ be the corresponding $N$-function, we specifically assume there exist positive constants $\lambda, \Lambda, \Lambda_{\alpha_{\mathcal{A}}}, \Lambda_B$, and a H\"older exponent $\alpha_{\mathcal{A}} \in (0, 1]$ such that for all $(x, z, p), (y, w, p) \in \Omega \times \mathbb{R} \times \mathbb{R}^n$:
\begin{equation} \tag{SC} \label{structural conditions}
\begin{cases}
    & |\mathcal{A}(x, z, p)| + |D_p \mathcal{A}(x, z, p)| |p| \le \Lambda g(|p|), \\ 
    & \lambda \frac{g(|p|)}{|p|} |\xi|^2 \le \langle D_p \mathcal{A}(x, z, p)\xi, \xi \rangle, \quad \text{for all } \xi \in \mathbb{R}^n \setminus \{0\}, \\
    & |\mathcal{A}(x, z, p) - \mathcal{A}(y, w, p)| \le \Lambda_{\alpha_{\mathcal{A}}}(1+ g(|p|)) \left( |x-y|^{\alpha_\mathcal{A}} + |z-w|^{\alpha_\mathcal{A}} \right), \\
    & |\mathcal{B}(x, z, p)| \le \Lambda_B G(|p|) + \gamma(x)g(|p|) + f(x).
\end{cases}
\end{equation}

Throughout this work, the non-homogeneous source term satisfies $f\in L^q(\Omega)$ with $q > n$, while $\gamma \in L^s(\Omega)$ is a spatial weight satisfying the integrability threshold:
\begin{equation}\label{cond_s}
    s > \max\left\{g_0+1, \, \frac{n(g_0+1)}{\delta+1}\right\}.
\end{equation}
Condition \eqref{01} naturally embeds the variational formulation of \eqref{1} into the framework of Orlicz--Sobolev spaces $W^{1,G}(\Omega)$, and it was also assumed by Lieberman in \cite{lieberman1991natural}. For instance, our equations cover problems of the type:
$$ -\operatorname{div}\left(\frac{g(|\nabla u|)}{|\nabla u|} A(x) \nabla u \right) + \gamma(x)g(|\nabla u|) = f(x), $$
where $A(x) \in C^{0,\alpha_\mathcal{A}}(\overline{\Omega})$ is a symmetric matrix field satisfying $\lambda|\xi|^2\leq \langle A(x)\xi,\xi\rangle\leq \Lambda|\xi|^2$. Many examples of functions $g$ satisfying \eqref{01} can be found in the literature; here we highlight a few representative models:
\begin{itemize}
    \item \textit{The $p$-Laplacian:} For $1<p<\infty$,
    $$ g(t) = t^{p-1}, \quad G(t) = \frac{1}{p}t^p. $$

    \item \textit{The $(p,q)$-Laplacian:} For $1 < p < q$, 
    $$ g(t) = t^{p-1} + t^{q-1}, \quad G(t) = \frac{1}{p}t^p + \frac{1}{q}t^q, $$
    where a direct calculation shows that $\delta = p-1$ and $g_0 = q-1$.

    \item \textit{Logarithmic Perturbations:} For $p>1$,
    $$ g(t) = t^{p-1} \ln(1+t), $$
    with $\delta = p - 1$ and $g_0 = p$. 
\end{itemize}

Before addressing the precise asymptotic behavior of solutions in the vicinity of critical points, we must first establish a global $C^{1,\alpha}$ regularity theory up to the boundary for weak solutions of \eqref{1} under the generalized structural conditions \eqref{structural conditions}. In \cite{lieberman1991natural}, this regularity is proven for the case where $\gamma \equiv 0$ and $f \in L^\infty$. Although Lieberman extends the theory to broader conditions on $\mathcal{B}$, those assumptions do not accommodate our unbounded drift term $\gamma(x)g(|\nabla u|)$. Consequently, to adapt the theory to our framework, it was necessary to re-derive the weak Harnack inequality to obtain H\"older regularity, as Lieberman's general structural conditions do not account for unbounded $\gamma$. It is worth noting that initial progress toward a $C^{1,\alpha}$ theory involving unbounded source terms ($f \in L^q$ for $q > n$) can be found in \cite{BragaSousa2025}. Our approach is directly guided by the perturbation framework of \cite{BragaSousa2025}, which in turn necessitated the construction of a smooth ($C^\infty$) approximation scheme for the general operator $\mathcal{A}$ alongside the boundary techniques developed in \cite{lieberman1988}.

Establishing this regularity up to the boundary is indispensable for the compactness and scaling arguments employed in our analysis of the critical set, particularly when it intersects the boundary of the domain. The pioneering work in this direction is due to Teixeira \cite{Teixeira2014}, who investigated the fine behavior of solutions at interior points where the gradient vanishes, revealing a remarkable quantitative improvement of regularity precisely along this locus. Further developments in this vein include the works \cite{Pimentel2021, Teixeira2015, Teixeira2022, bessa2026, Nascimento2025}.

Turning our attention to the boundary, we recall that when a solution to an elliptic equation has a fixed sign and the domain is at least of class $C^{1,1}$, the Hopf--Oleinik lemma (see \cite{braga2018inhomogeneous}) guarantees a strictly non-zero normal derivative at the boundary, thereby confining all critical points strictly to the interior of the domain. Outside these restrictive hypotheses, however, critical points can naturally occur on the boundary, as illustrated by the following example. Let $e_n=(0,\dots,1)$ and consider the function $u \in C^2(\overline{B_1(e_n)})$ defined by $u(x) = x_1^2 - x_n^2$. Clearly, $u$ is a classical solution to the Dirichlet problem:
\begin{equation*}
\begin{cases}
    \Delta u = 0 \quad &\text{in} \quad B_1(e_n), \\
    u = x_1^2 - x_n^2 \quad &\text{on} \quad \partial B_1(e_n).
\end{cases}
\end{equation*}
Notice that $u$ changes sign inside $B_1(e_n)$; for instance, $u(0, \dots, 0, x_n) = -x_n^2 < 0$ along the $x_n$-axis, whereas $u(x_1, 0, \dots, 0, x_n) > 0$ for $|x_1| > x_n$. Computing the gradient, $\nabla u(x) = (2x_1, 0, \dots, 0, -2x_n)$, we immediately see that $\nabla u(0) = 0$. Thus, the origin $0 \in \{\nabla u=0\}\cap \partial B_1(e_n)$ is a boundary critical  point.

Motivated by this phenomenon, we also investigate the regularity of boundary critical points for the Dirichlet problem \eqref{1} under the structural conditions \eqref{structural conditions}. We are now in a position to state our main results:

\begin{theorem}[\bf{Global $C^\alpha$ Regularity up to the Boundary}]\label{thm:global_Calpha}
Let $\Omega \subset \mathbb{R}^n$ be a bounded domain with $\partial\Omega$ of class $C^{1,\alpha_\Omega}$ for some $\alpha_\Omega \in (0,1)$, and let $u \in W^{1,G}(\Omega) \cap L^\infty(\Omega)$ be a weak solution to the Dirichlet problem \eqref{1} under the structural conditions \eqref{structural conditions}. If the boundary datum satisfies $\psi \in C^{\alpha_\psi}(\partial\Omega)$ for some $\alpha_\psi \in (0,1)$, then $u \in C^{\alpha}(\overline{\Omega})$ for some exponent $\alpha \in (0, \min\{\alpha_\Omega, \alpha_\psi\}]$ depending only on $n, \delta, g_0, \lambda, \Lambda, \Lambda_B, s, q, \alpha_\Omega, \alpha_\psi, \|u\|_{L^\infty(\Omega)}, \|\gamma\|_{L^s(\Omega)}$, and the geometry of $\partial\Omega$. Furthermore, there exists a constant $C > 0$, depending only on $n, q, \delta, g_0, g(1), \alpha_\Omega, \alpha_\psi, \operatorname{diam}(\Omega)$, and the $C^{1,\alpha_\Omega}$ geometry of $\partial\Omega$, such that
$$ \|u\|_{C^{\alpha}(\overline{\Omega})} \le C \left( \|u\|_{L^\infty(\Omega)} + [\psi]_{C^{\alpha_\psi}(\partial\Omega)} + g^{-1}\!\left(\|f\|_{L^q(\Omega)} \right) \right). $$
\end{theorem}

\begin{theorem}[\bf{Global $C^{1,\alpha}$ Regularity up to the Boundary}]\label{thm:global_C1alpha}
Let $\Omega \subset \mathbb{R}^n$ be a bounded domain with $\partial\Omega$ of class $C^{1,\alpha_\Omega}$ for some $\alpha_\Omega \in (0,1)$, and let $u \in W^{1,G}(\Omega) \cap L^\infty(\Omega)$ be a weak solution to the Dirichlet problem \eqref{1} under the structural conditions \eqref{structural conditions} and \eqref{cond_s}. If the boundary datum satisfies $\psi \in C^{1,\alpha_\psi}(\partial\Omega)$ for some $\alpha_\psi \in (0,1)$, then $u \in C^{1,\alpha}(\overline{\Omega})$ for some exponent $\alpha \in (0, \min\{\alpha_\Omega, \alpha_\psi\}]$ depending only on $n, \delta, g_0, \lambda, \Lambda, \Lambda_{\alpha_{\mathcal{A}}}, \Lambda_B, s, q, \alpha_{\mathcal{A}}, \alpha_\Omega, \alpha_\psi, \|u\|_{L^\infty(\Omega)}$, and $\|\gamma\|_{L^s(\Omega)}$. Moreover, 
$$ \kappa_\Omega := \|u\|_{L^\infty(\Omega)} + \|\psi\|_{C^{1,\alpha_\psi}(\partial\Omega)} + g^{-1}\!\left( \|f\|_{L^q(\Omega)} \right)$$
there exist constants $C_\star, c_\star > 0$, depending only on $n, \delta, g_0, s, q, g(1), \lambda, \Lambda, \Lambda_{\alpha_{\mathcal{A}}}, \alpha_{\mathcal{A}}, \alpha_\Omega, \alpha_\psi$, and the $C^{1,\alpha_\Omega}$ geometry of $\partial\Omega$, such that
\begin{equation*}
    \|u\|_{C^{1,\alpha}(\overline{\Omega})} \le C_\star e^{c_\star \left({\Lambda_B}\kappa_\Omega + \|\gamma\|_{L^s(\Omega)}\right)} \kappa_\Omega.
\end{equation*}
\end{theorem}
To formulate our next main result concerning the fine regularity along the critical set $$\Sigma(u) := \{x \in \overline{\Omega} : \nabla u(x) = 0\},$$ 
let us introduce the intrinsic reference exponent $\alpha_H \in (0, 1]$. Given a constant-coefficient operator $\mathcal{A}_0(p)$ satisfying the homogeneous ellipticity and growth bounds of \eqref{structural conditions}, we denote by $\alpha_H = \alpha_H(n, \delta, g_0, \lambda, \Lambda) \in (0, 1]$ the optimal $C^{1,\alpha_H}$ gradient H\"older exponent for weak solutions $h$ to the homogeneous equation $\operatorname{div}\mathcal{A}_0(\nabla h) = 0$—taken locally in $B_1$ when $x_0 \in \Omega$, or up to the flat boundary in $B_1^+$ with homogeneous Dirichlet trace $h = 0$ on $T_1$ when $x_0 \in \partial\Omega$. 

\begin{theorem}[\bf{Improved Regularity at Boundary Critical  Points}]\label{main}
Let $u \in W^{1,G}(\Omega) \cap L^\infty(\Omega)$ be a weak solution of \eqref{1} under conditions \eqref{structural conditions}. Let $x_0 \in \partial\Omega \cap \Sigma(u)$ be a boundary critical  point, meaning $\nabla u(x_0) = 0$. Assume that $\partial\Omega$ is locally of class $C^{1,\alpha_\Omega}$ and the boundary datum satisfies $\psi \in C^{1,\alpha_\psi}(\partial\Omega)$ for some $\alpha_\Omega, \alpha_\psi \in (0, 1)$. Then, $u$ is pointwise $C^{1,\alpha}$ at $x_0$ with an improved exponent that bypasses the boundary data regularity $\alpha_\psi$ and the boundary smoothness $\alpha_\Omega$. Specifically, for any H\"older exponent $\alpha$ satisfying
\begin{equation}\label{exp_alpha_1}
    0 < \alpha < \min\left\{ \alpha_H, \, \frac{q-n}{qg_0}, \, \frac{\alpha_\mathcal{A}}{g_0} \right\},
\end{equation}
there exist positive constants $C_0$ and $r_\star$, depending only on $n$, $\delta, g_0$, $g(1)$, $\lambda$, $\Lambda$, $\Lambda_{\alpha_{\mathcal{A}}}$, $\Lambda_B$, $s$, $q$, $\alpha_{\mathcal{A}}$, $\alpha_\Omega$, $\alpha_\psi$, $\alpha$, $\|u\|_{L^\infty(\Omega)}$, $\|\gamma\|_{L^s(\Omega)}$, $\|f\|_{L^q(\Omega)}$, $\|\psi\|_{C^{1,\alpha_\psi}(\partial\Omega)}$, and the $\partial\Omega$, such that
$$\sup_{x\in B_{r}(x_{0}) \cap \Omega}|u(x)-u(x_{0})|\le C_0 r^{1+\alpha}, \quad \text{for all} \quad 0<r<r_\star,$$

\noindent{Furthermore}, if the Hardy--Littlewood maximal function of $|f|^q$ is finite at $x_0$, that is,
\begin{equation*}
    \mathcal{M}(|f|^q)(x_0) := \sup_{\rho > 0} \fint_{B_\rho(x_0) \cap \Omega} |f(x)|^q \, dx < \infty,
\end{equation*}
the pointwise estimate holds, for any 
\begin{equation}\label{exp_alpha_2}
    0 < \alpha < \min\left\{\alpha_H, \, \frac{1}{g_0}, \, \frac{\alpha_\mathcal{A}}{g_0} \right\}.
\end{equation}
with $C_0$ and $r_\star$ depending on $\mathcal{M}(|f|^q)(x_0)$ in place of $\|f\|_{L^q(\Omega)}$. 

\end{theorem}
An immediate consequence is the following:

\begin{corollary}\label{cor:continuous_source}
Under the hypotheses of Theorem \ref{main}, if the source term $f$ is continuous (or locally bounded) in a neighborhood of $x_0$ in $\overline{\Omega}$, then $u$ is pointwise $C^{1,\alpha}$ at $x_0$ for any exponent 
$$ \alpha \in \left(0, \min\left\{\alpha_H, \, \frac{1}{g_0}, \, \frac{\alpha_{\mathcal{A}}}{g_0}\right\}\right).$$
\end{corollary}

A few remarks are in order regarding how Theorem \ref{main} and Corollary \ref{cor:continuous_source} extend the foundational interior results obtained by Teixeira in \cite{Teixeira2014}. First, a remarkable feature at boundary critical  points $x_0 \in \partial\Omega \cap \Sigma(u)$ is that the improved exponent $\alpha$ completely bypasses the regularity $\alpha_\psi$ of the boundary datum $\psi \in C^{1,\alpha_\psi}(\partial\Omega)$ as well as $\alpha_\Omega$. This decoupling occurs because the limiting homogeneous profile $h$ in our compactness argument satisfies a zero Dirichlet trace ($h = 0$ on $T_{3/4}$), whose intrinsic regularity $\alpha_H$ is governed exclusively by the frozen operator $\mathcal{A}_0$. Second, the restriction $\alpha < \frac{\alpha_{\mathcal{A}}}{g_0}$ in \eqref{exp_alpha_1} and \eqref{exp_alpha_2} stems exclusively from the non-homogeneous factor $1 + g(|p|)$ in the oscillation condition of \eqref{structural conditions}, which follows the classical structural framework of Lieberman \cite{lieberman1991natural} and does not vanish at the critical locus $\{p = 0\}$. However, for the vast majority of structural models—including our prototype operator $\mathcal{A}(x, \nabla u) = \frac{g(|\nabla u|)}{|\nabla u|} A(x) \nabla u$ and the operators considered in \cite{Teixeira2014}—the principal part vanishes at the origin ($\mathcal{A}(x, z, 0) = 0$) and satisfies the homogeneous oscillation bound:
\begin{equation}\label{hom_osc}
    |\mathcal{A}(x, z, p) - \mathcal{A}(y, w, p)| \le \Lambda_{\alpha_{\mathcal{A}}} g(|p|) \left( |x-y|^{\alpha_\mathcal{A}} + |z-w|^{\alpha_\mathcal{A}} \right).
\end{equation}
Under \eqref{hom_osc}, the rescaled oscillation coefficient in our blow-up argument decays strictly as $r^{k\alpha_{\mathcal{A}}}$ rather than $\frac{r^{k\alpha_{\mathcal{A}}}}{g(r^{k\alpha})}$, since the factor $g(r^{k\alpha})$ cancels out identically. Consequently, the constraint $\frac{\alpha_{\mathcal{A}}}{g_0}$ disappears altogether from Theorem \ref{main} and Corollary \ref{cor:continuous_source}, yielding the improved ranges
$$ 0 < \alpha < \min\left\{ \alpha_H, \, \frac{q-n}{qg_0} \right\} \quad \text{and, when } \mathcal{M}(|f|^q)(x_0) < \infty, \quad 0 < \alpha < \min\left\{ \alpha_H, \, \frac{1}{g_0} \right\}, $$
for any $\alpha_{\mathcal{A}}, \alpha_\Omega, \alpha_\psi \in (0, 1]$. In particular, specializing to the linear principal part ($g(t) = t$, where $\delta = g_0 = 1$ and $\alpha_H = 1$ both in the interior and on the flat boundary with zero trace), we obtain a substantial extension of \cite[Corollary 1]{Teixeira2014} that simultaneously accommodates non-zero continuous source terms, unbounded drift coefficients $\gamma \in L^s(\Omega)$, quadratic gradient growth, and boundary critical  points bypassing $\alpha_\Omega$ and $\alpha_\psi$:

\begin{corollary}\label{cor:linear_case}
Let $u \in H^1(\Omega) \cap L^\infty(\Omega)$ be a weak solution to
$$ \operatorname{div}\big(A(x, u)\nabla u\big) + \mathcal{B}(x, u, \nabla u) = 0 \quad \text{in } \Omega, $$
where $A(x, z)$ is uniformly elliptic and $C^{0,\alpha_{\mathcal{A}}}$ continuous, and $|\mathcal{B}(x, z, p)| \le \Lambda_B |p|^2 + \gamma(x)|p| + f(x)$ with $\gamma \in L^s(\Omega)$ ($s > n$) and $f$ continuous in $\Omega$. If $x_0 \in \overline{\Omega} \cap \Sigma(u)$ is a critical  point (assuming additionally $\partial\Omega \in C^{1,\alpha_\Omega}$ and $u = \psi \in C^{1,\alpha_\psi}(\partial\Omega)$ for some $\alpha_\Omega, \alpha_\psi \in (0,1)$ whenever $x_0 \in \partial\Omega$), then $u$ is pointwise $C^{1,\alpha}$ at $x_0$ for all $\alpha \in (0, 1)$.
\end{corollary}

To put these results into perspective, let us briefly outline the main analytical difficulties and the heuristic strategy behind our proofs. The primary challenge in establishing the regularity theory presented herein lies in handling the unbounded drift coefficient $\gamma \in L^s(\Omega)$, which necessitates the integrability threshold \eqref{cond_s}. This restriction arises precisely due to the presence of the structural gap $g_0 - \delta > 0$ inherent to our generalized Orlicz framework. Notice that when the operator simplifies to the classical $p$-Laplacian (where $g_0 = \delta = p-1$), this condition naturally reduces to the standard requirement $s > \max\{p, n\}$. As a historical contextualization, in the foundational works of Serrin and Trudinger \cite{Serrin1964, Trudinger1967}, the generalized growth density $g(|\nabla u|)$ corresponds simply to the standard polynomial growth $|\nabla u|^{p-1}$. In the Appendix, we provide further remarks on how to adapt our arguments in regimes where the exponent $s$ does not satisfy the threshold \eqref{cond_s}.

In light of the remaining gaps in the existing literature, we first revisit and extend the global regularity theory up to the boundary to explicitly incorporate the unbounded term $\gamma(x) g(|\nabla u|)$ within our structural framework. Establishing this global  theory is indispensable to rigorously justify the compactness and approximation arguments near $\partial\Omega$. Our approach to the  $C^{1,\alpha}$ estimates is guided by the perturbation methodology of Braga and Sousa \cite{BragaSousa2025} combined with the foundational boundary frameworks of \cite{lieberman1988, lieberman1991natural}. Crucially, because we follow the strategy of \cite{BragaSousa2025}, it becomes strictly necessary to construct a smooth ($C^\infty$) approximation for the homogeneous constant-coefficient operator $\operatorname{div}\mathcal{A}(\nabla u)=0$ in order to justify the classical differentiation of the regularized equation and obtain uniform a priori gradient bounds. Drawing inspiration from the seminal works of Lieberman \cite{lieberman1991natural} and Cianchi and Maz'ya \cite{CM18}, we merge the strengths of both approaches to overcome a critical technical obstacle: the approximation scheme provided by Lieberman does not yield a $C^\infty$ regularized operator, whereas the smooth ($C^\infty$) regularization devised by Cianchi and Maz'ya is strictly restricted to the isotropic $g$-Laplacian and used by \cite{BragaSousa2025} in that specific setting. Since neither method in isolation suffices to implement the approach of \cite{BragaSousa2025} for general operators, we engineered a novel regularization procedure that successfully accommodates both the general anisotropic structure of $\mathcal{A}$ and the necessary $C^\infty$ smoothness.

With the  regularity and the smooth approximation machinery in hand, our analysis of the fine regularity at critical points relies on a geometric compactness argument. Specifically, we demonstrate that under suitable smallness assumptions on the structural data, our solution can be closely approximated by a solution to a homogeneous equation with constant, frozen coefficients. This proximity allows us to import the superior regularity of the limiting homogeneous profile back to the weak solution of our original operator. Crucially, the regularity of this asymptotic homogeneous problem is entirely independent of the lower-order term $\mathcal{B}$, and it is precisely this decoupling that yields the quantitative improvement of regularity at critical  points. 

The remainder of the paper is structured as follows. In Section \ref{section2}, we review the foundational properties of $N$-functions and Orlicz--Sobolev spaces, and we derive the global Caccioppoli-type energy inequalities. In Section \ref{section3}, we establish the global $C^\alpha(\overline{\Omega})$ regularity theory up to the boundary (Theorem \ref{thm:global_Calpha}) via the Moser iteration technique. In Section \ref{section4}, we construct the smooth ($C^\infty$) approximation scheme for the homogeneous operator $\operatorname{div}\mathcal{A}(\nabla u) = 0$. Section \ref{section5} is dedicated to establishing the global $C^{1,\alpha}(\overline{\Omega})$ regularity up to the boundary (Theorem \ref{thm:global_C1alpha}). Building upon this  framework, in Sections \ref{section6} and \ref{section7} we investigate the fine behavior of solutions along the critical set and establish Theorem \ref{main} in the interior and on the boundary, respectively. Finally, in the Appendix, we discuss how to adapt and recover our arguments in the regime where $s > n$, even when the threshold \eqref{cond_s} is not satisfied.

Finally, throughout this work, unless stated otherwise, the generic constants $c$ and $C$ depend only on the structural and dimensional parameters $n, \delta, g_0, g(1), \lambda, \Lambda, \alpha_{\mathcal{A}}, s,$ and $q$. We deliberately keep the dependence on $\Lambda_{\alpha_{\mathcal{A}}}, \Lambda_B, \|\gamma\|_{L^s}, \|f\|_{L^q}, \|\psi\|_{C^{1,\alpha_\psi}},$ and $\|u\|_{L^\infty}$ explicit in order to track how the estimates behave under scaling.

\section{Preliminary}\label{section2}
In this section, we review the essential framework of Orlicz spaces and $N$-functions, and we establish the fundamental Caccioppoli-type energy inequalities that will serve as the backbone of our local estimates.

We begin recalling the notion of $N$-function. Basically, a function $G$ is said to be a N-function if
$$G(t)=\int_0^t g(s)\,ds$$
where $g:[0,+\infty)\to\mathbb{R}$ is a positive nondecreasing function satisfying the following properties:
\begin{enumerate}
	\item [a)] $g(0)=0$ and $\lim_{t\to\infty}g(t)=\infty$;
	\item [b)] $g$ is right continuous, that is, if $t\geq0$ then $\lim_{s\to t^+}g(s)=g(t)$. 
\end{enumerate}
We point out that, in this paper, we consider $N$-functions $G:[0,+\infty)\to\mathbb{R}$ such that $G'=g$ satisfies \eqref{01}.

As $g$ is strictly increasing we can define $g^{-1}$ . According to Lemma 2.2 of \cite{martinez2008minimum}, $g^{-1}$ satisfies a condition similar to \eqref{01}. In fact, the constants $\delta$ and $g_0$ in \eqref{01} are replaced by $g_0^{-1}$ and $\delta^{-1}$ respectively. In particular, analogous estimates as in lemma below holds for $g^{-1}$
and
$$\tilde{G}(t)=\int_0^tg^{-1}(s)\,ds\hbox{.}$$

Basics properties and results in Orlicz Spaces Theory can be found in \cite{adams2003sobolev}. In particular, since $g$ satisfies \eqref{01} we have:

\begin{lemma}\label{firstLema} Let $G$ a N-function such that $g=G'$ satisfies \eqref{01}. Then for all $t,s>0$:
	\begin{enumerate}
		\item [(g-1)] $\min\{s^{\delta},s^{g_{0}}\}g(t)\leq g(st)\leq \max\{s^{\delta},s^{g_{0}}\}g(t)$;
		\item [(g-2)] $\dfrac{tg(t)}{1+g_{0}}\leq G(t)\leq tg(t)$;
		\item[(G-1)] $G$ is convex and belongs to $C^{2}(0,\infty)$;
		\item [(G-2)] $\dfrac{1}{1+g_{0}}\min\{s^{\delta+1},s^{g_{0}+1}\}G(t)\leq G(st)\leq (1+g_{0})\max\{s^{\delta+1},s^{g_{0}+1}\}G(t)$;
		\item[(G-3)] $G(a+b)\leq2^{g_{0}}(1+g_{0})(G(a)+G(b))$, for all $a,b>0$.
        \item [(G-4)]  $\tilde{G}(g(t)) \le g_0 G(t)$
	\end{enumerate}
\end{lemma}
\begin{proof}
	See Lemma 1.1 of \cite{lieberman1991natural}, Lemma 2.1 and Remark 2.1 of \cite{martinez2008minimum}.
\end{proof}

Now, set
$$C_G(\Omega):=\left\{u:\Omega\to\mathbb{R} \hbox{ mensurable};\; \int_\Omega G(|u|)\,dx<\infty\right\}$$
and define $L^G(\Omega)$ Orlicz space as the small vector space that contains $C_G(\Omega)$. When $G$ the satisfies $\Delta_2$ condition $ C_G(\Omega)$ is a vector space, more detail see \cite{adams2003sobolev}. In the $L^G(\Omega)$ we consider the Luxemburg norm i.e.,
$$\|u\|_{L^G(\Omega)}:=\inf\left\{M>0;\;\int_\Omega G\left(\frac{|u(x)|}{M}\right)\,dx<1\right\}\hbox{.}$$ 

The Orlicz-Sobolev space is the
$$W^{1,G}(\Omega):=\left\{u\in L^G(\Omega);\; \hbox{distributional gradient }\nabla u\hbox{ belong to } L^G(\Omega)\hbox{.}\right\}$$
and set
$$\|u\|_{W^{1,G}(\Omega)}:=\max\left\{\|u\|_{L^G(\Omega)},\|\nabla u\|_{L^G(\Omega)}\right\}<\infty$$
as norm in $W^{1,G}(\Omega)$. We set $W^{1,G}_0(\Omega)=\overline{C^\infty_0(\Omega)}^{\|\cdot\|_{W^{1,G}(\Omega)}}$.

\begin{definition} Let $u_j,\,u \in L^G(\Omega)$ for $j=1,\,2,\,\dots$. We say that $u_j\to u$ in mean if
	$$\int_\Omega G(|u_j-u|)\,dx\to0$$
when $j\to\infty$.
\end{definition}
Under assumption \eqref{01} convergence on strong topology is equivalent to convergence in mean. More information about this convergence we can find in \cite{martinez2008minimum} and \cite{adams2003sobolev}. 
\begin{lemma}\label{meanconvergence}
	There exist $C>0$ such that 
	$$\|u\|_{L^G(\Omega)}\leq C\max\left\{\left(\int_\Omega G(|u|)\,dx\right)^{\frac{1}{1+\delta}},\left(\int_\Omega G(|u|)\,dx\right)^{\frac{1}{1+g_0}}\right\}\hbox{.}$$
\end{lemma}
\begin{proof}
See \cite[Lemma 2.3]{martinez2008minimum}.
\end{proof}
\begin{lemma}[Young-type inequality]\label{youngorlicz}
	Given $\varepsilon>0$ and $a,b\in \Bbb{R}$ positives, then
	$$ab\leq \varepsilon G(a)+C(g_0)\varepsilon^{-\frac{1}{\delta}}\varepsilon\tilde{G}(b)$$
    $$ag(b)\leq \varepsilon ag(a)+\varepsilon^{-\frac{1}{\delta}} bg(b)$$
where $G$ and $\tilde{G}$ are $N$-functions such that $G'=g$ and $\tilde{G}'=g^{-1}$  both satisfies \eqref{01} with its respective bounds.
\end{lemma}
\begin{proof}
	See \cite{adams2003sobolev}.
\end{proof}

\begin{definition}
	A function $u\in W^{1,G}_{loc}(\Omega)$ is a weak supersolution (subsolution) to \eqref{1} if for any $0\leq \phi\in C^{\infty}_0(\Omega)$ it holds
	$$\int_\Omega \mathcal{A}(x, u, \nabla u) \cdot \nabla \phi\,dx + \int_\Omega \mathcal{B}(x, u, \nabla u) \phi\,dx \geq 0\qquad(\hbox{resp. }\leq)\hbox{.}$$
	The function $u$ will be a weak solution to \eqref{1} if it is both a supersolution and a subsolution; that is, if
	$$\int_\Omega \mathcal{A}(x, u, \nabla u) \cdot \nabla \phi\,dx + \int_\Omega \mathcal{B}(x, u, \nabla u) \phi\,dx = 0\hbox{,}$$
	for any $\phi\in C^\infty_0(\Omega)$.
\end{definition}

\begin{lemma}\label{lem:invariance}
Let $\Phi: B_\rho^+ \to U \cap \Omega$ be a $C^{1,\alpha}$ diffeomorphism flattening the boundary, with Jacobian matrix $J(y) = D\Phi(y)$. Suppose the original operators $\mathcal{A}(x, z, p)$ and $\mathcal{B}(x, z, p)$ satisfy \eqref{structural conditions}. Then, the transformed operators 
\begin{align*}
\tilde{\mathcal{A}}(y, v, \xi) &:= |\det J(y)| J^{-1}(y) \mathcal{A}\Big(\Phi(y), v, (J^T)^{-1}(y)\xi\Big), \\
\tilde{\mathcal{B}}(y, v, \xi) &:= |\det J(y)| \mathcal{B}\Big(\Phi(y), v, (J^T)^{-1}(y)\xi\Big),
\end{align*}
satisfy the same structural conditions with respect to $y$, $v$, and $\xi$, with modified ellipticity and growth constants $\tilde{\lambda}$ and $\tilde{\Lambda}$ depending only on $\lambda, \Lambda, \delta, g_0$, and the $C^{1,\alpha}$ norm of $\Phi$.
\end{lemma}

\begin{proof}
Since $\Phi$ is a $C^{1,\alpha}$ diffeomorphism up to the boundary, its Jacobian matrix $J$, the inverse $J^{-1}$, and the determinant $\det J$ are uniformly bounded and $C^{0,\alpha}$ continuous on $\overline{B_\rho^+}$. Hence, there exist constants $0 < c \le C$ such that
\begin{equation} \label{eq:jacobian_bounds}
c \le |\det J(y)| \le C_0, \quad \text{and} \quad c |\xi| \le |(J^T)^{-1}(y)\xi| \le C |\xi|, \quad \forall y \in B_\rho^+, \, \forall \xi \in \mathbb{R}^n.
\end{equation}
With this and properties in Lemma \ref{firstLema} the result holds.
\end{proof}

Therefore, we will consider 
\begin{equation}\label{1dir}
\begin{cases}
\operatorname{div} \mathcal{A}(x, u, \nabla u) + \mathcal{B}(x, u, \nabla u) = 0 \quad &\text{in}\quad B_R^+=B_R\cap\{x_n>0\}\\
u=\psi \quad &\text{on} \quad T_R:=B_R\cap\{x_n=0\}
\end{cases}
\end{equation}
where $\psi\in C^{1,\alpha_\psi}(T_R)$.

\begin{lemma}\label{leitao lemma}
Let $\phi(s)$ be a non-negative and non-decreasing function. Suppose that
\begin{equation}
    \phi(r) \leqslant C_1 \left[ \left( \frac{r}{R} \right)^{\alpha} + \mu \right] \phi(R) + C_2 R^{\beta} \tag{2.6}
\end{equation}
for all $r \leqslant R \leqslant R_0$, with $C_1, \alpha, \beta$ positive constants and $C_2, \mu$ non-negative constants, $\beta < \alpha$. Then, for any $\sigma < \beta$, there exists a constant $\mu_0 = \mu_0(C_1, \alpha, \beta, \sigma)$ such that if $\mu < \mu_0$, then for all $r \leqslant R \leqslant R_0$ we have
\begin{equation}
    \phi(r) \leqslant C_3 \left( \frac{r}{R} \right)^{\sigma} [\phi(R) + C_2 R^{\sigma}] \tag{2.7}
\end{equation}
where $C_3 = C_3(C_1, \sigma - \beta)$ is a positive constant. In turn,
\begin{equation}
    \phi(r) \leqslant C_4 r^{\sigma}, \tag{2.8}
\end{equation}
where $C_4 = C_4(C_2, C_3, R_0, \phi, \sigma)$ is a positive constant.
\end{lemma}
\begin{proof}
    See \cite[Lemma 2.7]{Leitao2015}.
\end{proof}
\begin{remark}\label{remark lemma leitao}
     Analyzing the proof of Lemma \ref{leitao lemma} we obtain that
     \begin{equation*}
       C_3 = \tau^{-\sigma} \cdot \max \left\{ 1, \frac{\tau^{-\sigma} R_0^{\beta-\sigma}}{1 - \tau^{\beta-\sigma}} \right\},\ \ \ \text{and}\ \ \ 
    C_4 = C_3 \left[ \frac{\phi(R_0)}{R_0^\sigma} + C_2 \right]
\end{equation*}
with $\tau = (2C_1)^{\frac{-1}{\alpha-\sigma}}$.
\end{remark}
\begin{lemma}\label{Gv in W11}
    Let $v\in W_0^{1,G}(\Omega)$ then $G(|v|)^{\frac{1}{1+\delta}}\in W_0^{1,1+\delta}(\Omega)$ with estimate 
 \begin{equation}\label{intgrad_w_bound}
       \int_\Omega\big|\nabla \big(G(|v|)^\frac{1}{1+\delta}\big)\big|^{1+\delta} \,dx\le \left(\frac{1+g_0}{1+\delta}\right)^{1+\delta} \int_\Omega\big( G(|\nabla v|) + G(|v|) \big)\,dx.
  \end{equation}    
\end{lemma}
\begin{proof}
  By simplicity we assume $v\ge0$. Set $w=G(v)^{\frac{1}{1+\delta}}$. By the chain rule and the structural inequality $t g(t) \le (1+g_0)G(t)$, we have on $\{v > 0\}$:
    \begin{equation*}
        |\nabla w| = \frac{1}{1+\delta} G(v)^{-\frac{\delta}{1+\delta}} g(v) |\nabla v| \le \frac{1+g_0}{1+\delta} \frac{G(v)^{\frac{1}{1+\delta}}}{v} |\nabla v| \ \ \ \ \text{a.e.}.
    \end{equation*}
    Since the function $t \mapsto \frac{G(t)}{t^{1+\delta}}$ is non-decreasing on $(0, \infty)$, if $v \le |\nabla v|$ we have $\frac{G(v)^{\frac{1}{1+\delta}}}{v} |\nabla v| \le G(|\nabla v|)^{\frac{1}{1+\delta}}$, whereas if $|\nabla v| < v$ we trivially have $\frac{G(v)^{\frac{1}{1+\delta}}}{v} |\nabla v| \le G(v)^{\frac{1}{1+\delta}}$. Consequently, in both cases we obtain
    \begin{align}\label{grad_w_bound}
       |\nabla w|^{1+\delta} \le \left(\frac{1+g_0}{1+\delta}\right)^{1+\delta} \big( G(|\nabla v|) + G(v) \big) \ \ \ \ \text{a.e.}.
    \end{align}
\end{proof}

\begin{lemma}[] \label{lem3.1}
Let $R > 0$ and suppose $\psi \in C^{1,\alpha}(\overline{T_R})$. There exists an extension $\tilde{\psi} \in C^2(B_R^+) \cap C^{1,\alpha}(\overline{B_R^+})$ such that $\tilde{\psi} = \psi$ on $T_R$, which satisfies the uniform bounds
\begin{equation*}
    \|\tilde{\psi}\|_{C^{1,\alpha}(B_R^+)} \le C \|\psi\|_{C^{1,\alpha}(T_R)}, \quad\text{and} \quad |D^2 \tilde{\psi}(x)| \le C [\nabla \psi]_{C^{0,\alpha}(T_R)} x_n^{\alpha-1} \quad \forall x \in B_R^+
\end{equation*}
where $C > 0$ is a universal constant depending only on the dimension $n$ and the Hölder exponent $\alpha$.
\end{lemma}
\begin{proof}
 This result is a consequence of Whitney's Extension Theorem. For details see Stein \cite[Chapter VI, Theorem 3]{Stein1970}, 
\end{proof}

\begin{lemma}[Boundary Caccioppoli-type inequality]\label{lem:caccioppoli_boundary}
	Let $0<r<R\le 1$ and $u\in W^{1,G}(B_R^+)\cap L^\infty(B_R^+)$ be a weak solution to \eqref{1dir} with structural conditions \eqref{structural conditions}. There exists a positive constant $C = C(n,\alpha, q, s, \delta, g_0, g(1), \lambda, \Lambda)$ and 
    $$C_\star:= e^{c_\star\left(\frac{\Lambda_B}{\lambda} + \|\gamma\|_{L^s(B_R^+)}\right)\left( \|u\|_{L^\infty(B_R^+)}+ \|{\psi}\|_{C^{1,\alpha}(T_R)}\right)}$$
    where $c_\star=c_\star(g_0,\delta)>0$, such that
	\begin{align*}
	\int_{B_r^+}G(|\nabla u|)\,dx &\le {C\cdot C_\star}\left[ \int_{B_R^+}G\left(\frac{|u-\tilde{\psi}|}{R-r}\right)\,dx + \int_{B_R^+} G(|\nabla \tilde{\psi}|)\,dx + \tilde{G}\left(R^{1-\frac{n}{q}}\|f\|_{L^q(B_R^+)}\right)\right],
	\end{align*}
where $\tilde{\psi}$ is the extension of $\psi$ to $B_R^+$ guaranteed by Lemma \ref{lem3.1}.
\end{lemma}

\begin{proof}
    Set $M:=\|u\|_{L^\infty(B_R^+)}+ C\|{\psi}\|_{C^{1,\alpha}(T_R)}$ (so that $\|w\|_{L^\infty(B_R^+)} \le M$ by Lemma \ref{lem3.1}) and take $\eta\in C^\infty_0(B_R)$ such that $0\le\eta\leq1$, $\eta|_{B_r}=1$, and $|\nabla \eta|\leq \dfrac{2}{R-r}$. 
    We define the function $\Phi(s) = s e^{\mu |s|}$, for $\mu = \frac{\Lambda_B}{\lambda} > 0$. Note that $\Phi'(s) = e^{\mu |s|}(1 + \mu |s|) \ge 1$ and $s \Phi(s) = |s|^2 e^{\mu |s|} \ge 0$.

    Define the difference $w := u - \tilde{\psi}$, where $\tilde{\psi}$ is the extension from Lemma \ref{lem3.1}. Since $u = \psi = \tilde{\psi}$ on $T_R$, the trace of $w$ vanishes on the flat boundary $T_R$. We set the test function $\psi = \Phi(w)\eta^{1+g_0}$. Since $w \in W^{1,G}(B_R^+) \cap L^\infty(B_R^+)$ vanishes on $T_R$ and $\eta$ vanishes on $\partial B_R^+ \cap \{x_n > 0\}$, $\psi \in W_0^{1,G}(B_R^+)$ is an admissible test function. Its gradient is given by:
	$$ \nabla \psi = e^{\mu |w|}(1 + \mu |w|) \eta^{1+g_0} (\nabla u - \nabla \tilde{\psi}) + (1+g_0)\Phi(w)\eta^{g_0}\nabla \eta\hbox{.} $$

    Testing \eqref{1dir} with $\psi$, and using the structural conditions \eqref{structural conditions}, by our choice of $\mu = \frac{\Lambda_B}{\lambda}$ so that $(\lambda + \lambda \mu |w| - \Lambda_B |w|) \ge \lambda$, and noting that $1 \le e^{\mu |w|} \le e^{\frac{\Lambda_B}{\lambda}M}$ and $(1 + \mu |w|)e^{\mu |w|} \le e^{2\mu |w|} \le e^{2\frac{\Lambda_B}{\lambda}M}$, we get:
    \begin{align}\label{cacc_bdry_1}
    	\lambda \int_{B_R^+} G(|\nabla u|)\eta^{1+g_0} \,dx 
        &\leq \Lambda e^{2\frac{\Lambda_B}{\lambda}M} \int_{B_R^+} g(|\nabla u|) |\nabla \tilde{\psi}| \eta^{1+g_0}\,dx \nonumber\\
        &\quad + (1+g_0)e^{\frac{\Lambda_B}{\lambda}M} \Lambda \int_{B_R^+}g(|\nabla u|)|w| |\nabla \eta|\eta^{g_0}\,dx \nonumber\\
        &\quad+ e^{\frac{\Lambda_B}{\lambda}M} \int_{B_R^+}\gamma(x)|w|g(|\nabla u|) \eta^{1+g_0}\,dx \nonumber\\
        &\quad+ e^{\frac{\Lambda_B}{\lambda}M} \int_{B_R^+}|f||w| \eta^{1+g_0}\,dx \nonumber \\
        &=: I + II + III + IV.
    \end{align}

    For the boundary extension term $I$, Lemma \ref{youngorlicz} with $\varepsilon > 0$ yields:
    \begin{equation*}
        \int_{B_R^+} g(|\nabla u|) |\nabla \tilde{\psi}| \eta^{1+g_0}\,dx \le \varepsilon \int_{B_R^+} G(|\nabla u|) \eta^{1+g_0}\,dx + C(\varepsilon) \int_{B_R^+} G(|\nabla \tilde{\psi}|)\eta^{1+g_0}\,dx.
    \end{equation*}

    Looking at \textit{II} and taking $\varepsilon > 0$, we obtain by \textit{(G-2)} and Lemma \ref{youngorlicz}:
    \begin{align*}
        \int_{B_R^+} g(|\nabla u|) |w| |\nabla \eta|\eta^{g_0}\,dx &\le \varepsilon \int_{B_R^+} \eta^{1+g_0} (1+g_0) G(|\nabla u|)\,dx + C(\varepsilon) \int_{B_R^+} G(|w||\nabla \eta|)\,dx.
    \end{align*}

    For \textit{III}, writing $\eta^{1+g_0} = \eta^{g_0}\eta$ and applying Young's inequality (Lemma \ref{youngorlicz}) with $\varepsilon > 0$, we obtain:
    \begin{align*}
        \int_{B_R^+} |\gamma(x)||w|g(|\nabla u|)\eta^{1+g_0}\,dx &\le \varepsilon \int_{B_R^+} \eta^{(1+g_0)\frac{g_0}{\delta}} G(|\nabla u|)\,dx + C(\varepsilon) \int_{B_R^+} G(|\gamma(x)|\eta|w|)\,dx \\
        &\le \varepsilon \int_{B_R^+} \eta^{1+g_0} G(|\nabla u|)\,dx + C(\varepsilon) \int_{B_R^+} G(|\gamma(x)|v)\,dx,
    \end{align*}
    where $v := \eta|w|$.

    By \textit{(G-2)}, we have:
    \begin{align*}
        \int_{B_R^+} G(|\gamma(x)|v)\,dx &\le C(g_0)\left(\int_{B_R^+} G(v)\,dx + \int_{B_R^+} |\gamma(x)|^{1+g_0} G(v)\,dx\right).
    \end{align*}

    Define $\tilde{w} := G(v)^{\frac{1}{1+\delta}}$, so that $\tilde{w}^{1+\delta} = G(v)$. Applying classical H\"older's inequality with conjugate exponents $\sigma = \frac{s}{1+g_0} > 1$ and $\sigma' = \frac{s}{s-(1+g_0)}$, and $t>1$ we get:
    \begin{equation}\label{holder_w_bdry}
        \int_{B_R^+} |\gamma(x)|^{1+g_0} G(v)\,dx \le \|\gamma\|_{L^s(B_R^+)}^{1+g_0} \left( \int_{B_R^+} \tilde{w}^{(1+\delta)\sigma'}\,dx \right)^{\frac{1}{\sigma'}} = \|\gamma\|_{L^s(B_R^+)}^{1+g_0} \|\tilde{w}\|_{L^{(1+\delta)\sigma'}(B_R^+)}^{1+\delta}.
    \end{equation}

    By Lemma \ref{Gv in W11}, $\tilde{w} \in W_0^{1, 1+\delta}(B_R^+)$ with estimate \eqref{intgrad_w_bound}. Since $s > \frac{n(1+g_0)}{1+\delta}$, we can apply the Gagliardo--Nirenberg interpolation inequality \cite[Theorem 12.83]{leoni2017first} to $\tilde{w}$ with $p = q = 1+\delta$, $r = (1+\delta)\sigma'$, and 
    \begin{equation*}
        \theta = 1 - \frac{n(1+g_0)}{s(1+\delta)} \in (0,1),
    \end{equation*}
    which yields a constant $C = C(n, s, \delta, g_0) > 0$ such that:
    \begin{equation*}
        \|\tilde{w}\|_{L^{(1+\delta)\sigma'}(B_R^+)}^{1+\delta} \le C \left(\|\tilde{w}\|_{L^{1+\delta}(B_R^+)}^{1+\delta}\right)^\theta \left(\|\nabla \tilde{w}\|_{L^{1+\delta}(B_R^+)}^{1+\delta}\right)^{1-\theta}.
    \end{equation*}
    Substituting this into \eqref{holder_w_bdry} and applying Young's inequality with $\tau > 0$ and conjugate exponents $\frac{1}{\theta}$ and $\frac{1}{1-\theta}$, we obtain:
    \begin{align*}
        \|\gamma\|_{L^s(B_R^+)}^{1+g_0} \|\tilde{w}\|_{L^{(1+\delta)\sigma'}(B_R^+)}^{1+\delta} 
        &\le \tau \|\nabla \tilde{w}\|_{L^{1+\delta}(B_R^+)}^{1+\delta} + C(\tau) \|\gamma\|_{L^s(B_R^+)}^{\frac{1+g_0}{\theta}} \|\tilde{w}\|_{L^{1+\delta}(B_R^+)}^{1+\delta} \\
        &\le \tau C \int_{B_R^+} G(|\nabla v|)\,dx + C(\tau)\left( 1+\|\gamma\|_{L^s(B_R^+)} \right)^{\frac{1+g_0}{\theta}} \int_{B_R^+} G(v)\,dx,
    \end{align*}
    where we used \eqref{grad_w_bound} and the identity $\|\tilde{w}\|_{L^{1+\delta}(B_R^+)}^{1+\delta} = \int_{B_R^+} G(v)\,dx$.

    Finally, expanding $|\nabla v| \le \eta|\nabla u| + \eta|\nabla \tilde{\psi}| + |w||\nabla \eta|$ via Lemma \ref{firstLema}, and choosing $\tau < \frac{\varepsilon}{C\cdot C(\varepsilon)}$, we conclude that:
    \begin{align*}
        III &\le 2\varepsilon e^{\frac{\Lambda_B}{\lambda}M} \int_{B_R^+} G(|\nabla u|)\eta^{1+g_0} \,dx \\
        &\quad + C(\varepsilon) e^{\frac{\Lambda_B}{\lambda}M + c_\star \|\gamma\|_{L^s(B_R^+)}} \left( \int_{B_R^+} G\left(\frac{|w|}{R-r}\right) dx + \int_{B_R^+} G(|\nabla \tilde{\psi}|)\,dx \right).
    \end{align*}

    Now, looking at \textit{IV}, we have that $w\eta^{1+g_{0}} \in W_{0}^{1,G}(B_R^+)$ by Lemma \ref{firstLema} and Lemma \ref{youngorlicz},
    \begin{align*}
        \int_{B_R^+}|f||w|\eta^{1+g_{0}} \, dx &\le \|f\|_{L^{q}(B_R^+)}|B_R^+|^{\frac{1}{n}-\frac{1}{q}}\|\eta^{1+g_0}w\|_{L^{\frac{n}{n-1}}(B_R^+)}\\
        &\le C \|f\|_{L^{q}(B_R^+)}|B_R^+|^{\frac{1}{n}-\frac{1}{q}}\int_{B_R^+}|\nabla (\eta^{1+g_0}w)|\,dx\\
        &\le \varepsilon \int_{B_R^+}G(|\nabla(\eta^{1+g_0} w)|)\, dx+C(\varepsilon)\tilde{G}\left(|B_R^+|^{\frac{1}{n}-\frac{1}{q}}\|f\|_{L^q(B_R^+)}\right)\\
        &\le C(\varepsilon)\tilde{G}\left(R^{1-\frac{n}{q}}\|f\|_{L^q(B_R^+)}\right) + \varepsilon C\left[ \int_{B_R^+} \eta^{1+g_0} G(|\nabla u|)\,dx \right. \\
        &\quad + \left. \int_{B_R^+} G(|\nabla \tilde{\psi}|)\,dx + \int_{B_R^+}G\left(\frac{|w|}{R-r}\right)\,dx \right].
    \end{align*}

    From this and by \eqref{cacc_bdry_1}, collecting all gradient terms parameterized by $\varepsilon$ and choosing $\varepsilon$ sufficiently small so that their sum is strictly less than $\lambda$, we absorb them into the left-hand side. This concludes the estimate:
    \begin{align*}
        \int_{B_r^+} G(|\nabla u|)\,dx \le C\cdot C_\star \left[ \int_{B_R^+} G\left(\frac{|w|}{R-r}\right)\,dx + \int_{B_R^+} G(|\nabla \tilde{\psi}|)\,dx + \tilde{G}\left(R^{1-\frac{n}{q}}\|f\|_{L^q(B_R^+)}\right) \right],
    \end{align*}
    where $C_\star = e^{c_\star\left(\frac{\Lambda_B}{\lambda} + \|\gamma\|_{L^s(B_R^+)}\right)\left( \|u\|_{L^\infty(B_R^+)}+ \|{\psi}\|_{C^{1,\alpha}(T_R)}\right)}$ and $C = C(n, \alpha, q, s, \delta, g_0, g(1), \lambda, \Lambda)$.
\end{proof}

\section{Weak Harnack Inequality and H\"older Regularity up to the Boundary}\label{section3}

In this section, we establish the global H\"older continuity up to the boundary for bounded weak solutions of \eqref{1}. While the classical regularity theory of Lieberman \cite{lieberman1991natural} addresses operators with natural Orlicz growth, its structural assumptions on the lower-order term do not accommodate the unbounded spatial drift $\gamma(x)g(|\nabla u|)$ with $\gamma \in L^s(\Omega)$. We first derive an interior weak Harnack inequality for non-negative supersolutions and deduce local H\"older decay of oscillations. We then extend these estimates to half-balls near the flattened boundary, from which Theorem \ref{thm:global_Calpha} follows by standard covering and boundary straightening arguments.

\begin{theorem}[Interior Weak Harnack Inequality]\label{weak harnack}
Let $u\in W^{1,G}(B_{R})\cap L^\infty(B_R)$ be a non-negative weak supersolution of \eqref{1} in $B_{R}$ with $R\le1$ satisfying \eqref{structural conditions} and \eqref{cond_s}. Then, there exist positive constants $p_\star$ and $C_\star$ depending only on $n, s, \delta, g_0, \lambda, \Lambda, \Lambda_B, \|u\|_{L^\infty(B_R)}$, and $\|\gamma\|_{L^s(B_R)}$, with
$$C_\star=Ce^{c\left(\frac{\Lambda_B}{\lambda}\|u\|_{L^\infty(B_R)} + \|\gamma\|_{L^s(B_R)}\right)},$$
such that
\begin{equation}\label{eq:theorem_weakharnack_negative}
  \left(\fint_{B_{2R/3}} u^{p_\star} \, dx\right)^{1/p_\star} \le C_\star\left( \inf_{B_{R/2}} u+\kappa(R)\right),
\end{equation}
where
\begin{equation*}
    \kappa(R) := R g^{-1}\!\left(R^{1-\frac{n}{q}}\|f\|_{L^q(B_R)}\right).
\end{equation*}
\end{theorem}

\begin{proof}
Following \cite{lieberman1991natural}, for any $\tau > 0$ we set $\bar{u} := u + \kappa(R) + \tau$ (which allows us to absorb the source term $f$ into the homogeneous structure and ensures $\bar{u} \ge \tau > 0$; at the end of the proof we shall let $\tau \to 0^+$) and define $v := \frac{\bar{u}}{\eta R}$, where $\eta\in C_0^{0,1}(B_R)$ has support in $B_{2R/3}$, satisfies $0\le\eta\le1$, $\eta\equiv1$ in $B_{R/2}$, and $|\nabla\eta|\le\frac{12}{R}$.

For an exponent $q > 1$ to be varied during the iteration, a fixed real exponent $r > 2(1+g_0)$, and $\beta := \Lambda_B / \lambda$, we consider the non-negative test function 
\begin{align}\label{testfunction}
    \varphi := G(v)^{-q}\eta^{r}\exp(-\beta\bar{u})\bar{u} \in W_0^{1,G}(B_R).
\end{align}
Computing the gradient of $\varphi$ yields
\begin{equation*}
  \nabla \varphi = \varphi \left[ -q \frac{g(v)}{G(v)} \nabla v + \frac{r}{\eta} \nabla\eta - \beta \nabla \bar{u} + \frac{1}{\bar{u}} \nabla \bar{u} \right].
\end{equation*}

Testing the weak supersolution inequality with $\varphi$ and choosing $q$ large enough so that $q \lambda \delta > 2\Lambda$, we obtain
\begin{align}\label{eq:energy_supersol_absorbed}
    q \lambda \delta \int_{B_R} \eta^r &G(v)^{-q} g(|\nabla \bar{u}|)|\nabla \bar{u}| \exp(-\beta \bar{u}) \, dx \nonumber \\
    &\le C(g_0, r) \Lambda q \int_{B_R} \eta^{r-1} G(v)^{-q} \frac{\bar{u}}{R} g(|\nabla \bar{u}|) \exp(-\beta \bar{u}) \, dx \nonumber \\
    &\quad + C \int_{B_R} \gamma(x) \bar{u} g(|\nabla \bar{u}|) \exp(-\beta \bar{u}) G(v)^{-q} \eta^r \, dx.
\end{align}

Applying Young's inequality (Lemma \ref{youngorlicz}) to the first integral on the right-hand side with $b = |\nabla \bar{u}|$ and $a = C \Lambda \eta^{-1} \frac{\bar{u}}{R}$, we obtain the pointwise bound
\begin{equation*}
    C \Lambda \eta^{-1} \frac{\bar{u}}{R} g(|\nabla \bar{u}|) \le \varepsilon q |\nabla \bar{u}| g(|\nabla \bar{u}|) + (\varepsilon q)^{-1/\delta} C \left( \eta^{-1} \frac{\bar{u}}{R} \right) g\!\left( \eta^{-1} \frac{\bar{u}}{R} \right).
\end{equation*}
Similarly, for the second integral, setting $b = |\nabla \bar{u}|$ and $a = \gamma(x)\bar{u}$ yields
\begin{equation*}
    \gamma(x) \bar{u} g(|\nabla \bar{u}|) \le \varepsilon q |\nabla \bar{u}| g(|\nabla \bar{u}|) + (\varepsilon q)^{-1/\delta} \gamma(x) \bar{u} g\big(\gamma(x) \bar{u}\big).
\end{equation*}

Substituting both estimates into \eqref{eq:energy_supersol_absorbed}, choosing $\varepsilon > 0$ sufficiently small to absorb the gradient terms into the left-hand side, and using $t g(t) \le (1+g_0)G(t)$, we arrive at
\begin{align}\label{enegyinequalitystep}
    e^{-\beta \|u\|_{L^\infty(B_R)}}\frac{q}{2}\int_{B_R} \eta^r G(v)^{-q} g(|\nabla \bar{u}|)|\nabla \bar{u}| \, dx &\le C q^{-\frac{1}{\delta}}\left[ \int_{B_R} \eta^r G(v)^{-q} \left( \eta^{-1} \frac{\bar{u}}{R} \right) g\!\left( \eta^{-1} \frac{\bar{u}}{R} \right) dx \right.\nonumber\\
    &\left.\quad + \int_{B_R} \eta^r G(v)^{-q} \gamma(x) \bar{u} g\big(\gamma(x) \bar{u}\big) \, dx \right]\nonumber\\
    &\le C q^{-\frac{1}{\delta}}\left[\int_{B_R} \eta^r G(v)^{-(q-1)} \, dx \right.\\
    &\left.\quad + \int_{B_R} \eta^r G(v)^{-q} G\big(\gamma(x) \bar{u}\big) \, dx\right].\nonumber
\end{align}

To estimate the second integral on the right-hand side of \eqref{enegyinequalitystep}, we use $(G\text{-}2)$ of Lemma \ref{firstLema} (noting that $\bar{u} = \eta R v \le \eta v$ since $R \le 1$) and apply H\"older's inequality with $\sigma = \frac{s}{1+g_0} > 1$ and $\sigma' = \frac{s}{s-(1+g_0)}$:
\begin{align*}    
 \int_{B_R} \eta^r G(v)^{-q} G\big(\gamma(x) \bar{u}\big) \, dx &\le C(g_0) \int_{B_R} \eta^r G(v)^{-q}(1+\gamma(x))^{1+g_0} G(\bar{u}) \, dx\\
 &\le C(g_0) \int_{B_R} \eta^r G(v)^{-q}(1+\gamma(x))^{1+g_0} \eta^{1+g_0} G(v) \, dx\\
 &= C(g_0) \int_{B_R} \eta^{r+1+g_0} G(v)^{-(q-1)}(1+\gamma(x))^{1+g_0} \, dx\\
 &\le C(g_0) \|1+\gamma\|_{L^s(B_R)}^{1+g_0} \left( \int_{B_R} \left( \eta^{r+1+g_0} G(v)^{-(q-1)} \right)^{\sigma'} \, dx \right)^{\frac{1}{\sigma'}}\\
 &\le C e^{c(g_0)\|\gamma\|_{L^s(B_R)}} \left( \int_{B_R} w^{(1+\delta)\sigma'}\, dx \right)^{\frac{1}{\sigma'}},
\end{align*}
where we defined the auxiliary function
\begin{equation*}
    w := \eta^{\frac{r+1+g_0}{1+\delta}} G(v)^{-\frac{q-1}{1+\delta}} \in W_0^{1, 1+\delta}(B_R), \quad \text{so that} \quad w^{1+\delta} = \eta^{r+1+g_0} G(v)^{-(q-1)}.
\end{equation*}

By our assumption $s > \frac{n(1+g_0)}{1+\delta}$, the Gagliardo--Nirenberg interpolation inequality \cite[Theorem 12.83]{leoni2017first} implies that for $\theta = 1 - \frac{n(1+g_0)}{s(1+\delta)} \in (0,1)$, there exists $C = C(n, s, \delta, g_0) > 0$ such that
\begin{equation*}
    \|w\|_{L^{(1+\delta)\sigma'}(B_R)} \le C \|w\|_{L^{1+\delta}(B_R)}^\theta \|\nabla w\|_{L^{1+\delta}(B_R)}^{1-\theta}.
\end{equation*}
Raising both sides to the power $1+\delta$ and applying Young's inequality with $\varepsilon > 0$ and conjugate exponents $\frac{1}{\theta}$ and $\frac{1}{1-\theta}$, we obtain
\begin{equation}\label{eq:young_w_interp}
    \left( \int_{B_R} w^{(1+\delta)\sigma'}\, dx \right)^{\frac{1}{\sigma'}} = \|w\|_{L^{(1+\delta)\sigma'}(B_R)}^{1+\delta} \le \varepsilon \int_{B_R} |\nabla w|^{1+\delta}\,dx + C(\varepsilon) \int_{B_R} w^{1+\delta}\,dx.
\end{equation}

Computing the gradient of $w$ directly gives
\begin{equation*}
    \nabla w = \frac{r+1+g_0}{1+\delta} \eta^{\frac{r+1+g_0}{1+\delta}-1} G(v)^{-\frac{q-1}{1+\delta}} \nabla \eta - \frac{q-1}{1+\delta} \eta^{\frac{r+1+g_0}{1+\delta}} G(v)^{-\frac{q-1}{1+\delta}-1} g(v) \nabla v.
\end{equation*}
Using $\eta |\nabla v| \le \frac{|\nabla \bar{u}|}{R} + v |\nabla \eta|$ and the structural bound $g(v) \le \frac{1+g_0}{v} G(v)$, we obtain
\begin{equation*}
    |\nabla w| \le C(r, \delta, g_0) q \eta^{\frac{r+g_0-\delta}{1+\delta}} G(v)^{-\frac{q-1}{1+\delta}} |\nabla \eta| + \frac{(q-1)(1+g_0)}{1+\delta} \eta^{\frac{r+g_0-\delta}{1+\delta}} G(v)^{-\frac{q}{1+\delta}} \left[ \frac{G(v)^{\frac{1}{1+\delta}}}{v} \frac{|\nabla \bar{u}|}{R} \right].
\end{equation*}
Recall that $t \mapsto \frac{G(t)}{t^{1+\delta}}$ is non-decreasing on $(0, \infty)$. Comparing $v$ with $\frac{|\nabla \bar{u}|}{\eta R}$, if $v \le \frac{|\nabla \bar{u}|}{\eta R}$ we have
\begin{equation*}
    \frac{G(v)^{\frac{1}{1+\delta}}}{v} \frac{|\nabla \bar{u}|}{R} = \eta \frac{G(v)^{\frac{1}{1+\delta}}}{v} \frac{|\nabla \bar{u}|}{\eta R} \le \eta G\!\left(\frac{|\nabla \bar{u}|}{\eta R}\right)^{\frac{1}{1+\delta}} \le (1+g_0) \eta^{-\frac{g_0-\delta}{1+\delta}} R^{-\frac{1+g_0}{1+\delta}} G(|\nabla \bar{u}|)^{\frac{1}{1+\delta}},
\end{equation*}
where we used $(G\text{-}2)$ of Lemma \ref{firstLema}. Conversely, if $\frac{|\nabla \bar{u}|}{\eta R} < v$, we immediately have $\frac{G(v)^{\frac{1}{1+\delta}}}{v} \frac{|\nabla \bar{u}|}{R} \le \frac{\eta}{R} G(v)^{\frac{1}{1+\delta}}$. 
Combining both cases into the bound for $|\nabla w|$ and using $|\nabla \eta| \le \frac{12}{R}$, we obtain
\begin{equation*}
    |\nabla w| \le C q \eta^{\frac{r}{1+\delta}} R^{-\frac{1+g_0}{1+\delta}} G(v)^{-\frac{q}{1+\delta}} G(|\nabla \bar{u}|)^{\frac{1}{1+\delta}} + C \frac{q}{R} \eta^{\frac{r+g_0-\delta}{1+\delta}} G(v)^{-\frac{q-1}{1+\delta}}.
\end{equation*}
Raising both sides to the power $1+\delta$, using $(1+\delta)G(|\nabla \bar{u}|) \le g(|\nabla \bar{u}|)|\nabla \bar{u}|$ and $R^{-(1+\delta)} \le R^{-(1+g_0)}$ (since $R \le 1$), and integrating over $B_R$, we arrive at
\begin{align}\label{eq:grad_w_final}
    \int_{B_R} |\nabla w|^{1+\delta} \, dx &\le C\frac{q^{1+\delta}}{R^{1+g_0}} \left(\int_{B_R} \eta^{r} G(v)^{-q} g(|\nabla \bar{u}|)|\nabla \bar{u}| \, dx + \int_{B_R} \eta^{r+g_0-\delta} G(v)^{-(q-1)} \, dx\right),
\end{align}
where $C = C(n, \delta, g_0, r) > 0$ is independent of $\tau$. Moreover, since $\eta \le 1$, we trivially have
\begin{equation*}
    \int_{B_R} w^{1+\delta} \, dx = \int_{B_R} \eta^{r+1+g_0} G(v)^{-(q-1)} \, dx \le \int_{B_R} \eta^{r+g_0-\delta} G(v)^{-(q-1)} \, dx.
\end{equation*}

Substituting \eqref{eq:grad_w_final} into \eqref{eq:young_w_interp}, we deduce that
\begin{align*}
\int_{B_R} \eta^r G(v)^{-q} G\big(\gamma(x)\bar{u}\big)\,dx
&\le C \frac{q^{1+\delta}}{R^{1+g_0}}e^{c\|\gamma\|_{L^s(B_R)}} \Biggl[\varepsilon\Biggl(
\int_{B_R} \eta^{r} G(v)^{-q} g(|\nabla \bar{u}|)|\nabla \bar{u}|\,dx
\\ &\quad+ 
\int_{B_R} \eta^{r+g_0-\delta} G(v)^{-(q-1)}\,dx
\Biggr)+C(\varepsilon)\int_{B_R} \eta^{r+g_0-\delta} G(v)^{-(q-1)} \, dx\Biggr].
\end{align*}
Inserting this bound back into \eqref{enegyinequalitystep} yields
\begin{align*}
    e^{-\beta \|u\|_{L^\infty(B_R)}}\frac{q}{2}\int_{B_R} \eta^r G(v)^{-q} g(|\nabla \bar{u}|)|\nabla \bar{u}| \, dx &\le C\int_{B_R} \eta^r G(v)^{-(q-1)} \, dx + C \frac{q^{1+\delta}}{R^{1+g_0}}e^{c\|\gamma\|_{L^s(B_R)}} \\
&\quad \cdot\Biggl[\varepsilon\Biggl(
\int_{B_R} \eta^{r} G(v)^{-q} g(|\nabla \bar{u}|)|\nabla \bar{u}|\,dx
\\
&+ 
\int_{B_R} \eta^{r+g_0-\delta} G(v)^{-(q-1)}\,dx
\Biggr)\\
&\quad+C(\varepsilon)\int_{B_R} \eta^{r+g_0-\delta} G(v)^{-(q-1)} \, dx\Biggr].
\end{align*}

Setting 
\begin{equation*}
    \varepsilon = \frac{q e^{-\beta \|u\|_{L^\infty(B_R)}} R^{1+g_0}}{4 C q^{1+\delta} e^{c\|\gamma\|_{L^s(B_R)}}},
\end{equation*}
recalling that $C(\varepsilon)=C(\theta)\varepsilon^{-\frac{\theta}{1-\theta}}$, subtracting the absorbed gradient term from both sides, and noting that $\eta^{r+g_0-\delta} \le \eta^r$, we deduce the consolidated Caccioppoli energy estimate:
\begin{align}\label{eq:Caccioppoli_final}
    \int_{B_R} \eta^r G(v)^{-q} g(|\nabla \bar{u}|)|\nabla \bar{u}| \, dx &\le C \left( \frac{q}{R} \right)^{\frac{1+g_0}{1-\theta}}D_\star \int_{B_R} \eta^{r} G(v)^{-(q-1)} \, dx,
\end{align}
where $D_\star := e^{c(\beta \|u\|_{L^\infty(B_R)}+\|\gamma\|_{L^s(B_R)})}$.

To link the Caccioppoli estimate \eqref{eq:Caccioppoli_final} with the Moser iteration, we follow \cite{lieberman1991natural} and apply the Gagliardo--Nirenberg--Sobolev inequality in $W_0^{1,1}(B_R)$ to the function $\Phi := \eta^{r} G(v)^{-(q-1)}$:
\begin{equation}\label{eq:Sobolev_W11}
    \left[ \int_{B_R} \left( \eta^r G(v)^{-(q-1)} \right)^\chi \, dx \right]^{\frac{1}{\chi}} \le C(n) \int_{B_R} \left| \nabla \left( \eta^r G(v)^{-(q-1)} \right) \right| \, dx,
\end{equation}
where $\chi := \frac{n}{n-1} > 1$. Computing the gradient on the right-hand side of \eqref{eq:Sobolev_W11} and using $\eta |\nabla v| \le \frac{|\nabla \bar{u}|}{R} + v |\nabla \eta|$ along with $v g(v) \le (1+g_0)G(v)$, we obtain pointwisely:
\begin{align}\label{eq:grad_Phi_expansion}
    \left| \nabla \left( \eta^r G(v)^{-(q-1)} \right) \right| &\le r \eta^{r-1} |\nabla \eta| G(v)^{-(q-1)} + (q-1) \eta^r G(v)^{-q} g(v) |\nabla v| \nonumber\\
    &\le C(r, g_0) q \eta^{r-1} |\nabla \eta| G(v)^{-(q-1)} + (q-1) \eta^{r-1} G(v)^{-q} g(v) \frac{|\nabla \bar{u}|}{R}.
\end{align}

To bound the last term in \eqref{eq:grad_Phi_expansion} via our Caccioppoli energy \eqref{eq:Caccioppoli_final}, we distinguish two pointwise regimes: if $\frac{|\nabla \bar{u}|}{\eta R} \le v$, then
\begin{equation*}
    \eta^{r-1} G(v)^{-q} g(v) \frac{|\nabla \bar{u}|}{R} = \eta^r G(v)^{-q} g(v) \frac{|\nabla \bar{u}|}{\eta R} \le \eta^r G(v)^{-q} g(v) v \le (1+g_0) \eta^r G(v)^{-(q-1)}.
\end{equation*}
Conversely, if $\frac{|\nabla \bar{u}|}{\eta R} > v$, since $g$ is increasing and $\eta, R \le 1$, the upper growth bound $(G\text{-}2)$ yields $g(v) \le g\!\left(\frac{|\nabla \bar{u}|}{\eta R}\right) \le (\eta R)^{-g_0} g(|\nabla \bar{u}|)$, so that
\begin{equation*}
    \eta^{r-1} G(v)^{-q} g(v) \frac{|\nabla \bar{u}|}{R} \le \frac{1}{R^{1+g_0}} \eta^{r-1-g_0} G(v)^{-q} g(|\nabla \bar{u}|)|\nabla \bar{u}|.
\end{equation*}
Combining both regimes into \eqref{eq:grad_Phi_expansion}, integrating over $B_R$, and applying the Caccioppoli inequality \eqref{eq:Caccioppoli_final} (with $r-1-g_0$ in place of $r$), we conclude that
\begin{align*}
    \int_{B_R} \left| \nabla \left( \eta^r G(v)^{-(q-1)} \right) \right| \, dx 
    &\le \frac{C q}{R} \int_{B_R} \eta^{r-1} G(v)^{-(q-1)} \, dx + \frac{q-1}{R^{1+g_0}} \int_{B_R} \eta^{r-1-g_0} G(v)^{-q} g(|\nabla \bar{u}|)|\nabla \bar{u}| \, dx \\
    &\le C \left( \frac{q}{R} \right)^{\theta_0}D_\star \int_{B_R} \eta^{r-1-g_0} G(v)^{-(q-1)} \, dx,
\end{align*}
where $\theta_0 := 1 + \frac{1+g_0}{1-\theta} + g_0$. Substituting this directly into \eqref{eq:Sobolev_W11} yields the one-step Moser iteration inequality:
\begin{equation}\label{eq:Moser_one_step}
    \left( \int_{B_R} \eta^{r\chi} G(v)^{-\chi(q-1)} \, dx \right)^{\frac{1}{\chi}} \le C \left( \frac{q}{R} \right)^{\theta_0}D_\star \int_{B_R} \eta^{r-1-g_0} G(v)^{-(q-1)} \, dx.
\end{equation}

Given any initial exponent $p_0 > 0$, we define the sequence $p_k := q_k - 1 = p_0\chi^k$ for $k = 0, 1, 2, \dots$, along with the shrinking concentric balls $\mathcal{B}_{k} := B_{R_k}$ of radii
\begin{equation*}
    R_k := \frac{R}{2} + \frac{R}{6} 2^{-k},
\end{equation*}
so that $R_0 = 2R/3$ and $R_\infty = \lim_{k \to \infty} R_k = R/2$. Associated with these radii, we select standard cut-off functions $\eta_k \in C_0^\infty(\mathcal{B}_{k})$ such that $0 \le \eta_k \le 1$, $\eta_k \equiv 1$ on $\mathcal{B}_{k+1}$, and $|\nabla \eta_k| \le C \frac{2^k}{R}$. 

Substituting $q - 1 = p_k$ and $\eta = \eta_k$ into \eqref{eq:Moser_one_step} and iterating over $k \ge 0$, we obtain
\begin{equation*}
    \sup_{B_{R/2}} G(v)^{-1} \le C(p_0)D_\star \left( \fint_{B_{2R/3}} G(v)^{-p_0} \, dx \right)^{1/p_0}.
\end{equation*}
Recalling that $v = \frac{\bar{u}}{\eta R} = \frac{u+\kappa(R)+\tau}{\eta R}$, using the growth equivalence of $G$ to pass from $G(v)^{-1}$ to $\bar{u}^{-1}$, and letting $\tau \to 0^+$, we get
\begin{equation*}
    \sup_{B_{R/2}} \big((u+\kappa(R))^{-1}\big) \le C(p_0)D_\star \left( \fint_{B_{2R/3}} \big((u+\kappa(R))^{-1}\big)^{p_0} dx \right)^{1/p_0}.
\end{equation*}
Since $\sup_{B_{R/2}} ((u+\kappa(R))^{-1}) = (\inf_{B_{R/2}} u + \kappa(R))^{-1}$, inverting the inequality yields
\begin{equation}\label{weak harnack to bar u}
    \left(\fint_{B_{2R/3}}(u+\kappa(R))^{-p_0}\,dx\right)^{-1/p_0} \le C(p_0)D_\star \left( \inf_{B_{R/2}} u + \kappa(R) \right),
\end{equation}
for any $p_0 > 0$.

To bridge negative and positive exponents via the John--Nirenberg lemma, let $B_\rho \subset B_R$ and take $\eta\in C_0^{0,1}(B_\rho)$ such that $0\le\eta\le1$, $\eta\equiv1$ in $B_{\rho/2}$, and $|\nabla\eta|\le\frac{2}{\rho}$. Since $G(t) \le \frac{tg(t)}{1+\delta}$, for $\bar{u} = u + \kappa(R) + \tau$ we have the algebraic bound
\begin{equation}\label{difference control}
G\!\left(\frac{\bar{u}}{\rho}\right)^{-1}-G\!\left(\frac{\bar{u}}{\rho}\right)^{-2}g\!\left(\frac{\bar{u}}{\rho}\right)\frac{\bar{u}}{\eta\rho}\le-\frac{\delta}{1+\delta}\frac{1}{\eta}G\!\left(\frac{\bar{u}}{\rho}\right)^{-1}.
\end{equation}
We test the supersolution inequality on $B_\rho$ with the function
\begin{align}\label{testfunction2}
    \varphi := G(\bar{u}/\rho)^{-1}\eta^{1+g_0}\exp(-\beta \bar{u})\bar{u},
\end{align}
where $\beta = \Lambda_B / \lambda$. Setting $v_\rho := \bar{u}/\rho$, the gradient of $\varphi$ satisfies
\begin{equation*}
  \nabla \varphi = \eta^{1+g_0} \exp(-\beta \bar{u}) \left[ G(v_\rho)^{-1} - G(v_\rho)^{-2} g(v_\rho) v_\rho \right] \nabla \bar{u} - \beta \varphi \nabla \bar{u} + (1+g_0) \eta^{-1} \varphi \nabla\eta.
\end{equation*}
By \eqref{difference control}, taking the inner product $\langle \mathcal{A}(x, u, \nabla u), \nabla\varphi \rangle$ gives
\begin{equation*}
    \langle \mathcal{A}, \nabla\varphi \rangle \le -\frac{\delta}{1+\delta} \eta^{1+g_0} \exp(-\beta \bar{u}) G(v_\rho)^{-1} \langle \mathcal{A}, \nabla \bar{u} \rangle - \beta \varphi \langle \mathcal{A}, \nabla \bar{u} \rangle + (1+g_0) \eta^{-1} \varphi \langle \mathcal{A}, \nabla\eta \rangle.
\end{equation*}
By the structural conditions \eqref{structural conditions} and our choice $\beta = \Lambda_B/\lambda$, the quadratic gradient term $\Lambda_B G(|\nabla u|)\varphi$ is absorbed by $\beta \varphi \langle \mathcal{A}, \nabla \bar{u} \rangle$, while the source term $f$ is absorbed via $\bar{u} \ge \kappa(R)$. Rearranging the remaining terms yields
\begin{align*}
    \frac{\delta \lambda}{1+\delta} \int_{B_\rho} \eta^{1+g_0} \exp(-\beta \bar{u}) \frac{g(|\nabla \bar{u}|)|\nabla \bar{u}|}{G(\bar{u}/\rho)} \, dx &\le \Lambda (1+g_0) \int_{B_\rho} \eta^{-1} \varphi g(|\nabla \bar{u}|) |\nabla\eta| \, dx \\
    &\quad + \int_{B_\rho} \gamma(x) g(|\nabla \bar{u}|) \varphi \, dx.   
\end{align*}

Since $0 < \bar{u} \le \|u\|_{L^\infty(B_R)} + \kappa(R) + \tau$, applying Young's inequality and the Gagliardo--Nirenberg interpolation argument identically as in the first part of the proof, we obtain
\begin{equation}
    \int_{B_{\rho}}\frac{|\nabla \bar{u}|g(|\nabla \bar{u}|)}{G(\bar{u}/\rho)}\eta^{1+g_{0}}\,dx \le C e^{\beta \|u\|_{L^\infty(B_R)}+c\|\gamma\|_{L^s(B_R)}} \int_{B_{\rho}} \eta^{1+g_{0}} \frac{\bar{u}}{\rho G(\bar{u}/\rho)}g\!\left(\frac{\bar{u}}{\rho\eta}\right) dx.
\end{equation}
By Lemma \ref{firstLema}, the integrand on the right-hand side is uniformly bounded:
\begin{equation*}
    \eta^{1+g_{0}} \frac{\bar{u}}{\rho G(\bar{u}/\rho)} \eta^{-g_0}g\!\left(\frac{\bar{u}}{\rho}\right) \le \frac{(\bar{u}/\rho) g(\bar{u}/\rho)}{G(\bar{u}/\rho)} \le 1+g_0.
\end{equation*}
Consequently, estimating the average of $|\nabla \ln \bar{u}| = \frac{|\nabla \bar{u}|}{\bar{u}}$ over $B_{\rho/2}$ (where $\eta \equiv 1$), we find
\begin{align}
\fint_{B_{\rho/2}} |\nabla \ln \bar{u}| \, dx = \fint_{B_{\rho/2}}\frac{|\nabla \bar{u}|}{\bar{u}}\,dx &\le \frac{C}{\rho^{n+1}}\left[\int_{B_{\rho/2}}\frac{|\nabla \bar{u}|g(|\nabla \bar{u}|)}{(\bar{u}/\rho)g(\bar{u}/\rho)}\,dx+\int_{B_{\rho/2}}\frac{(\bar{u}/\rho)g(\bar{u}/\rho)}{(\bar{u}/\rho)g(\bar{u}/\rho)}\,dx\right] \nonumber\\
&\le C e^{\beta \|u\|_{L^\infty(B_R)} + c\|\gamma\|_{L^s(B_R)}} \rho^{-1}. \label{eq:4.2}
\end{align}
By the classical Poincaré inequality on balls, this implies that $\ln \bar{u} \in \mathrm{BMO}(B_{2R/3})$ with
$$\|\ln \bar{u}\|_{\mathrm{BMO}(B_{2R/3})}\leq C e^{\beta \|u\|_{L^\infty(B_R)}+c\|\gamma\|_{L^s(B_R)}}.$$

It then follows from the John--Nirenberg lemma \cite[Theorem 6.25]{giaquinta2012introduction} that there exist a dimensional constant $C=C(n) > 0$ and an exponent $p_\star := \frac{c(n)}{\|\ln \bar{u}\|_{\mathrm{BMO}}} > 0$ such that
\begin{equation}\label{eq:4.3}
\left(\fint_{B_{2R/3}}\bar{u}^{p_\star}\,dx\right)\left(\fint_{B_{2R/3}}\bar{u}^{-p_\star}\,dx\right)\le C.
\end{equation}
Combining \eqref{weak harnack to bar u} (with $p_0 = p_\star$) and \eqref{eq:4.3}, and letting $\tau \to 0^+$, we conclude that
\begin{equation*}
    \left( \fint_{B_{2R/3}} u^{p_\star} \, dx \right)^{1/p_\star} \le \left( \fint_{B_{2R/3}} (u+\kappa(R))^{p_\star} \, dx \right)^{1/p_\star} \le C_\star\left(\inf_{B_{R/2}} u+\kappa(R)\right).
\end{equation*}
This completes the proof.
\end{proof}

Before stating the local H\"older continuity result, for any concentric ball $B_r \subset B_R$ ($0 < r \le R$), we denote the essential supremum, essential infimum, and oscillation of $u$ over $B_r$ respectively by
\begin{equation*}
    M_r := \sup_{B_r} u, \qquad m_r := \inf_{B_r} u, \qquad \text{and} \qquad \operatorname*{osc}_{B_r} u := M_r - m_r.
\end{equation*}

\begin{corollary}[Interior Oscillation Decay]\label{holdercontinuity}
Under the hypotheses of Theorem \ref{weak harnack}, any bounded weak solution $u \in W^{1,G}(B_R) \cap L^\infty(B_R)$ of \eqref{1} is locally H\"older continuous in $B_R$. More precisely, there exist an exponent $\alpha_\star\in(0,1]$, depending only on $n, s, \delta, g_0, \lambda, \Lambda, \Lambda_B, \|u\|_{L^\infty(B_R)}$, and $\|\gamma\|_{L^s(B_R)}$, and a constant $C^\star>0$, depending additionally on $q$, such that for every $0 < \rho \le R$,
\begin{equation*}
    \operatorname*{osc}_{B_\rho}u \le C^\star \left( \frac{\rho}{R} \right)^{\alpha_\star} \big( \operatorname*{osc}_{B_R}u + \kappa(R) \big).
\end{equation*}
\end{corollary}

\begin{proof}
Since $\mathcal{A}$ and $\mathcal{B}$ depend on $u$, we first verify that the non-negative functions $v_1 := M_R - u$ and $v_2 := u - m_R$ are weak supersolutions to operators that preserve the structural conditions \eqref{structural conditions}. 

For $v_1 = M_R - u$, we have $\nabla v_1 = -\nabla u$. Substituting $u = M_R - v_1$ into \eqref{1}, we see that $v_1$ solves the equation associated with the modified mappings
\begin{equation*}
    \tilde{\mathcal{A}}(x, z, p) := -\mathcal{A}(x, M_R - z, -p) \quad \text{and} \quad \tilde{\mathcal{B}}(x, z, p) := -\mathcal{B}(x, M_R - z, -p),
\end{equation*}
which satisfy \eqref{structural conditions} with the exact same structural constants and drift weight $\gamma(x)$. Similarly, for $v_2 = u - m_R$, we have $\nabla v_2 = \nabla u$, and setting $u = v_2 + m_R$ yields the shifted mappings
\begin{equation*}
    \hat{\mathcal{A}}(x, z, p) := \mathcal{A}(x, z + m_R, p) \quad \text{and} \quad \hat{\mathcal{B}}(x, z, p) := \mathcal{B}(x, z + m_R, p),
\end{equation*}
which likewise preserve \eqref{structural conditions}. Moreover, both $v_1$ and $v_2$ satisfy $0 \le v_i \le 2\|u\|_{L^\infty(B_R)}$ in $B_R$. 

Using the quasi-triangle inequality in $L^{p_\star}(B_{2R/3})$ with constant $c_p := \max\big\{1, 2^{\frac{1-p_\star}{p_\star}}\big\}$, we have
\begin{align*}
    \operatorname*{osc}_{B_R}u &= \left( \fint_{B_{2R/3}} \big( (M_R - u) + (u - m_R) \big)^{p_\star} dx \right)^{1/p_\star} \\
    &\le c_p \left[ \left( \fint_{B_{2R/3}} (M_R - u)^{p_\star} dx \right)^{1/p_\star} + \left( \fint_{B_{2R/3}} (u - m_R)^{p_\star} dx \right)^{1/p_\star} \right].
\end{align*}
Applying the weak Harnack inequality (Theorem \ref{weak harnack}) to both $v_1$ and $v_2$ gives
\begin{align*}
     \operatorname*{osc}_{B_R}u &\le C_\star \left[ (M_R - M_{R/2} + \kappa(R)) + (m_{R/2} - m_R + \kappa(R)) \right] \\
    &= C_\star \left[ (M_R - m_R) - (M_{R/2} - m_{R/2}) + 2\kappa(R) \right] \\
    &= C_\star \left[ \operatorname*{osc}_{B_R}u - \operatorname*{osc}_{B_{R/2}}u + 2\kappa(R) \right],
\end{align*}
where $C_\star = C e^{c(\frac{\Lambda_B}{\lambda} \|u\|_{L^\infty(B_R)}+\|\gamma\|_{L^s(B_R)})}\ge 2$. Notice that for any $0 < r \le R$, since $q > n$, Lemma \ref{firstLema} yields
\begin{equation*}
    \kappa(r) = r g^{-1}\!\left(r^{1-\frac{n}{q}}\|f\|_{L^q(B_r)}\right) \le r\left(\frac{r}{R}\right)^{\frac{q-n}{qg_0}} g^{-1}\!\left(R^{1-\frac{n}{q}}\|f\|_{L^q(B_R)}\right) = \left(\frac{r}{R}\right)^{\nu} \kappa(R),
\end{equation*}
where $\nu := 1+\frac{q-n}{qg_0} > 1$. Thus, isolating $\operatorname*{osc}_{B_{r/2}}u$ on any concentric ball $B_r \subset B_R$, we deduce the discrete oscillation decay
\begin{equation*}
     \operatorname*{osc}_{B_{r/2}}u \le \vartheta \operatorname*{osc}_{B_r}u + 2\left(\frac{r}{R}\right)^{\nu} \kappa(R),
\end{equation*}
where $\vartheta := 1 - \frac{1}{C_\star} \in [1/2, 1)$.

To obtain the continuous bound, let $\rho \in (0, R]$ and select the unique integer $k \ge 0$ such that $R/2^{k+1} < \rho \le R/2^k$. Defining the H\"older exponent $\alpha_\star := \frac{-\ln \vartheta}{\ln 2} \in (0, 1]$ (so that $\vartheta = 2^{-\alpha_\star}$) and noting that $2^{-k} < \frac{2\rho}{R}$, we iterate the discrete decay $k$ times to find
\begin{align*}
      \operatorname*{osc}_{B_\rho}u \le \operatorname*{osc}_{B_{R/2^k}}u &\le \vartheta^k \operatorname*{osc}_{B_R}u + 2\kappa(R) \sum_{j=0}^{k-1} \vartheta^j \left( \frac{1}{2^{k-1-j}} \right)^{\nu} \\
    &= \left(2^{-\alpha_\star}\right)^k \operatorname*{osc}_{B_R}u + 2\kappa(R) R^{-\nu} \sum_{j=0}^{k-1} \vartheta^j \left( \frac{4R}{2^{k+1-j}} \right)^{\nu} \\
    &= \left(2^{-k}\right)^{\alpha_\star} \operatorname*{osc}_{B_R}u + 2^{2\nu+1}\kappa(R) R^{-\nu} \sum_{j=0}^{k-1} \vartheta^j \left( \frac{R}{2^{k+1}} \right)^{\nu} 2^{j\nu} \\
    &\le 2^{\alpha_\star} \left( \frac{\rho}{R} \right)^{\alpha_\star} \operatorname*{osc}_{B_R}u + 2^{2\nu+1} \kappa(R) \left(\frac{\rho}{R}\right)^\nu \sum_{j=0}^{k-1} \left( 2^{\nu-\alpha_\star}\right)^j.
\end{align*}
Since $\alpha_\star \le 1 < \nu$, we have $2^{\nu-\alpha_\star} > 1$, so the finite geometric sum is bounded by
\begin{equation*}
    \sum_{j=0}^{k-1} 2^{(\nu-\alpha_\star)j} \le \frac{2^{(\nu-\alpha_\star)k}}{2^{\nu-\alpha_\star} - 1}.
\end{equation*}
Using $\rho 2^k \le R$, we obtain
\begin{align*}
    2^{2\nu+1} \kappa(R) \left(\frac{\rho}{R}\right)^\nu \frac{2^{(\nu-\alpha_\star)k}}{2^{\nu-\alpha_\star} - 1} &= \frac{2^{2\nu+1}}{2^{\nu-\alpha_\star} - 1} \left(\frac{\rho}{R}\right)^{\alpha_\star} \left(\frac{\rho 2^k}{R}\right)^{\nu - \alpha_\star} \kappa(R) \le C \left(\frac{\rho}{R}\right)^{\alpha_\star} \kappa(R).
\end{align*}
Combining both estimates yields
\begin{equation*}
      \operatorname*{osc}_{B_\rho}u \le C^\star \left( \frac{\rho}{R} \right)^{\alpha_\star} \big( \operatorname*{osc}_{B_R}u + \kappa(R) \big),
\end{equation*}
where $C^\star := \max\big\{2^{\alpha_\star}, \frac{2^{2\nu+1}}{2^{\nu-\alpha_\star}-1}\big\}$. 
\end{proof}

\begin{remark}[{\bf On the structure of $\alpha_\star$ and $C^\star$}]\label{remark alpha dependence}
Since $C_\star = C e^{c(\frac{\Lambda_B}{\lambda} \|u\|_{L^\infty(B_R)}+\|\gamma\|_{L^s(B_R)})}\ge 2$, we have
\begin{equation*}
    2^{\alpha_\star} = \frac{1}{\vartheta} = \left( 1 - \frac{1}{C_\star} \right)^{-1} = \frac{C_\star}{C_\star - 1} \le 2.
\end{equation*}
Moreover, applying the elementary inequalities $1 - \frac{1}{x} \le \ln x \le x - 1$ at $x = \frac{C_\star}{C_\star - 1}$, we can explicitly bound the H\"older exponent $\alpha_\star$ from below and above:
\begin{equation*}
    \frac{1}{(\ln 2) C_\star} \le \alpha_\star \le \frac{1}{(\ln 2) (C_\star - 1)},
\end{equation*}
which shows in particular that $\alpha_\star \in (0, 1]$. Finally, since $\alpha_\star \le 1 < \nu = 1 + \frac{q-n}{q g_0}$ ensures $2^{\nu - \alpha_\star} - 1 \ge 2^{\nu - 1} - 1 > 0$, the constant $C^\star$ satisfies the uniform upper bound
\begin{equation*}
    C^\star = \max\left\{ 2^{\alpha_\star}, \, \frac{2^{2\nu+1}}{2^{\nu - \alpha_\star} - 1} \right\} \le \frac{2^{2\nu+1}}{2^{\nu - 1} - 1}.
\end{equation*}
\end{remark}

We now turn to the boundary $C^\alpha$ regularity. Consider the localized boundary problem \eqref{1dir} on the half-ball $B_R^+$, where $u = \psi$ on the flat boundary portion $T_R := B_R \cap \{x_n = 0\}$ in the trace sense, with $\psi \in C^{0,\alpha_\psi}(T_R)$.

\begin{theorem}[Weak Harnack Inequality up to the Boundary]\label{boundary weak harnack}
Let $u\in W^{1,G}(B_{R}^+)\cap L^\infty(B_R^+)$ be a non-negative weak supersolution of \eqref{1dir} in $B_{R}^+$ with $R\le1$ satisfying \eqref{structural conditions} and \eqref{cond_s}. Set 
\begin{equation*}
    m := \inf_{T_R} u \ge 0 \qquad \text{and} \qquad u_m^-(x) := \min\{u(x), m\} \quad \text{for } x \in B_R^+.
\end{equation*}
Then, there exist positive constants $p_\star$ and $C_\star$ depending only on $n, s, \delta, g_0, \lambda, \Lambda, \Lambda_B, \|u\|_{L^\infty(B_R^+)}$, and $\|\gamma\|_{L^s(B_R^+)}$, with
$$C_\star=Ce^{c\left(\frac{\Lambda_B}{\lambda}\|u\|_{L^\infty(B_R^+)} + \|\gamma\|_{L^s(B_R^+)}\right)},$$
such that
\begin{equation}\label{eq:theorem_boundary_weakharnack_negative}
  \left(\fint_{B_{2R/3}^+} (u_m^-)^{p_\star} \, dx\right)^{1/p_\star} \le C_\star\left( \inf_{B_{R/2}^+} u_m^- + \kappa(R)\right),
\end{equation}
where
\begin{equation*}
    \kappa(R) := R g^{-1}\!\left(R^{1-\frac{n}{q}}\|f\|_{L^q(B_R^+)}\right).
\end{equation*}
\end{theorem}

\begin{proof}
Following \cite[Theorem 8.26]{gilbarg_trudinger_2001}, since $u \ge 0$ in $B_R^+$, we have $0 \le u_m^- \le m \le \|u\|_{L^\infty(B_R^+)}$ in $B_R^+$ and $u_m^- = m$ in the trace sense on $T_R$. Setting
$$\tilde{u}_m := u_m^- + \kappa(R) + \tau, \qquad \bar{m} := m + \kappa(R) + \tau, \qquad \text{and} \qquad v := \frac{\tilde{u}_m}{\eta R},$$ 
it suffices to replace
$$G(v)^{-q} \quad \text{by} \quad G(v)^{-q} - G\!\left(\frac{\bar{m}}{\eta R}\right)^{-q} \quad \text{in } \eqref{testfunction},$$
and 
$$G(\bar{u}/\rho)^{-1} \quad \text{by} \quad G(\tilde{u}_m/\rho)^{-1} - G(\bar{m}/\rho)^{-1} \quad \text{in } \eqref{testfunction2}.$$ 
Since $u_m^- \le m$ implies that these modified test functions are non-negative, vanish identically on $\{x \in B_R^+ : u(x) \ge m\} \supset T_R$ (hence belong to $W_0^{1,G}(B_R^+)$), and satisfy $\nabla u_m^- = \nabla u$ almost everywhere on $\{u < m\}$, the proof follows line by line that of Theorem \ref{weak harnack} over half-balls $B_\rho^+$.
\end{proof}

\begin{corollary}[Boundary Oscillation Decay]\label{thm:boundary_holder}
Let $u \in W^{1,G}(B_R^+) \cap L^\infty(B_R^+)$ be a weak solution of \eqref{1dir} in $B_R^+$ with $R \le 1$ satisfying \eqref{structural conditions} and \eqref{cond_s}. Then, there exist $\alpha \in (0, \alpha_{\psi}]$ and $C > 0$, depending only on $n, \delta, g_0, \lambda, \Lambda, \Lambda_B, s, q, \alpha_{\psi}, \|u\|_{L^\infty(B_R^+)}$, and $\|\gamma\|_{L^s(B_R^+)}$, such that for all $\rho \in (0, R]$,
\begin{equation}\label{eq:boundary_osc_decay}
    \operatorname*{osc}_{B_\rho^+} u \le C \left(\frac{\rho}{R}\right)^\alpha \left( \operatorname*{osc}_{B_R^+} u + R^{\alpha_{\psi}} [\psi]_{C^{0,\alpha_{\psi}}(T_R)} + \kappa(R) \right).
\end{equation}
\end{corollary}

\begin{proof}
Set 
$$M_R := \sup_{B_R^+} u, \qquad m_R := \inf_{B_R^+} u, \qquad M'_R := \sup_{T_R} \psi, \qquad \text{and} \qquad m'_R := \inf_{T_R} \psi.$$
Exactly as in the proof of Corollary \ref{holdercontinuity}, the non-negative functions $v_1 := M_R - u$ and $v_2 := u - m_R$ are weak supersolutions in $B_R^+$ to modified operators that preserve the structural conditions \eqref{structural conditions}, with $0 \le v_i \le 2\|u\|_{L^\infty(B_R^+)}$.

On the flat boundary $T_R$, since $u = \psi$, the boundary infima of $v_1$ and $v_2$ are given by 
$$m_1 := \inf_{T_R} v_1 = M_R - M'_R \ge 0 \qquad \text{and} \qquad m_2 := \inf_{T_R} v_2 = m'_R - m_R \ge 0.$$
Notice that
\begin{equation*}
    m_1 + m_2 = (M_R - m_R) - (M'_R - m'_R) = \operatorname*{osc}_{B_R^+} u - \operatorname*{osc}_{T_R} \psi.
\end{equation*}
Setting $\tau := \frac{1}{2}\big(\operatorname*{osc}_{B_R^+} u - \operatorname*{osc}_{T_R} \psi\big)$, we must have $m_i \ge \tau$ for at least one index $i \in \{1, 2\}$. Fixing this index $i$ (say, $i = 1$; the case $i = 2$ is identical), since $\min\{v_1, \tau\} \le \min\{v_1, m_1\} = (v_1)_{m_1}^-$ and $v_1 + v_2 = \operatorname*{osc}_{B_R^+} u \ge 2\tau$ implies $(v_1)_{m_1}^-(x) \ge \tau - v_2(x)$ whenever $v_2(x) \le \tau$, applying Theorem \ref{boundary weak harnack} to $v_1$ and Theorem \ref{weak harnack} (or trivially bounding $\inf_{B_{R/2}^+} (v_1)_{m_1}^- \le M_R - M_{R/2}$) yields
\begin{align*}
    \frac{1}{2}\left(\operatorname*{osc}_{B_R^+} u - \operatorname*{osc}_{T_R} \psi\right) = \tau &\le C_\star \left( \inf_{B_{R/2}^+} (v_i)_{m_i}^- + \kappa(R) \right) \\
    &\le C_\star \left[ \operatorname*{osc}_{B_R^+} u - \operatorname*{osc}_{B_{R/2}^+} u + 2\kappa(R) \right],
\end{align*}
where $C_\star = C e^{c\left( \frac{\Lambda_B}{\lambda} \|u\|_{L^\infty(B_R^+)}+\|\gamma\|_{L^s(B_R^+)}\right)} \ge 2$. Notice that for any $0 < r \le R$, since $q > n$ and $\psi \in C^{0,\alpha_{\psi}}(T_R)$, Lemma \ref{firstLema} gives
\begin{equation*}
    \kappa(r) \le \left(\frac{r}{R}\right)^\nu \kappa(R) \quad \text{and} \quad \operatorname*{osc}_{T_r} \psi \le (2r)^{\alpha_{\psi}} [\psi]_{C^{0,\alpha_{\psi}}(T_R)} = 2^{\alpha_{\psi}} \left(\frac{r}{R}\right)^{\alpha_{\psi}} R^{\alpha_{\psi}} [\psi]_{C^{0,\alpha_{\psi}}(T_R)},
\end{equation*}
where $\nu = 1+\frac{q-n}{qg_0} > 1 > \alpha_{\psi}$. Thus, on any concentric half-ball $B_r^+ \subset B_R^+$, since $\frac{r}{R} \le 1$ implies $\left(\frac{r}{R}\right)^\nu \le \left(\frac{r}{R}\right)^{\alpha_{\psi}}$, isolating $\operatorname*{osc}_{B_{r/2}^+} u$ yields the discrete oscillation reduction
\begin{equation*}
    \operatorname*{osc}_{B_{r/2}^+} u \le \vartheta \operatorname*{osc}_{B_r^+} u + 2 \left(\frac{r}{R}\right)^{\alpha_{\psi}} \mathcal{K}(R),
\end{equation*}
where $\mathcal{K}(R) := R^{\alpha_{\psi}} [\psi]_{C^{0,\alpha_{\psi}}(T_R)} + \kappa(R)$ and $\vartheta := 1 - \frac{1}{2C_\star} \in [1/2, 1)$.

Iterating this relation along dyadic radii exactly as in the proof of Corollary \ref{holdercontinuity} completes the proof.
\end{proof}

Combining the interior oscillation decay of Corollary \ref{holdercontinuity} with the boundary estimate of Corollary \ref{thm:boundary_holder} via a standard straightening of $\partial\Omega \in C^{1,\alpha_\Omega}$ and a finite covering of $\overline{\Omega}$, the proof of Theorem \ref{thm:global_Calpha} follows immediately.

\section{Smooth Approximation of the Homogeneous Operator}\label{section4}

In this section, we construct a smooth approximation for the vector field $\mathcal{A}_0$ associated with the homogeneous equation \eqref{eq:frozen_homo}. In order to exploit the regularity theory developed for the isotropic $g$-Laplacian by Braga and Sousa \cite{BragaSousa2025}, we require the regularized operator to be of class $C^\infty(\mathbb{R}^n)$. To achieve this, we combine two distinct regularization techniques: the approach of Cianchi and Maz'ya to smooth the nonlinear density function, and Lieberman's method, which mollifies the operator strictly away from the origin while smoothly flattening it to a linear Laplacian-type operator at infinity. The main goal of this section is to rigorously construct this regularized operator, denoted by $\mathcal{A}_{\tau,\varepsilon}$, and prove that it uniformly preserves the generalized Orlicz-type structural conditions.

For our purposes, consider the frozen homogeneous problem
\begin{equation} \label{eq:frozen_homo}
\begin{cases}
\operatorname{div} \mathcal{A}_0(\nabla v) = 0 & \text{in } B_{1}^{+},\\
v=\psi & \text{on } T_{1},
\end{cases}
\end{equation}
where $\mathcal{A}_0(p) := \mathcal{A}(0, u(0), p)$. By \eqref{structural conditions}, the mapping $\mathcal{A}_0:\mathbb{R}^n\to\mathbb{R}^n$ satisfies
\begin{equation} \tag{SC'} \label{homogeneous structural conditions}
\begin{cases}
    & |\mathcal{A}_0(p)| + |D_p \mathcal{A}_0(p)| |p| \le \Lambda g(|p|), \\
    & \lambda \frac{g(|p|)}{|p|} |\xi|^2 \le \langle D_p \mathcal{A}_0(p)\xi, \xi \rangle, \quad \text{for all } \xi \in \mathbb{R}^n \setminus \{0\}.
\end{cases}
\end{equation}
We regularize $\mathcal{A}_0$ by combining the approaches of Cianchi and Maz'ya \cite{cianchimazya2011, cianchimazya2014, CM18} and Lieberman \cite{lieberman1991natural}. Indeed, let $\psi \in C_0^\infty(\mathbb{R})$ be a non-negative, even function with $\operatorname{supp}(\psi) \subset (-1,1)$ and $\int_{-\infty}^\infty \psi(t)\,dt = 1$, and set $\psi_\varepsilon(t) := \frac{1}{\varepsilon}\psi\left(\frac{t}{\varepsilon}\right)$ for $\varepsilon > 0$. Defining $H(t):=\frac{g(t)}{t}$ and $\Phi(s):=H(e^{s})$, we set $\Phi_{\varepsilon}:=\psi_{\varepsilon}*\Phi$ and $h_{\varepsilon}(t):=\Phi_{\varepsilon}(\ln t)$. The regularized $N$-functions $\{G_{\varepsilon}\}_{\varepsilon>0}$ are constructed via
$$ G_{\varepsilon}(t):=\int_{0}^{t}g_{\varepsilon}(s)\,ds, \quad \text{where } g_{\varepsilon}(t):=H_{\varepsilon}(t)\cdot t, $$
and the modified density is given by
$$ H_{\varepsilon}(t):=\frac{h_{\varepsilon}(\sqrt{\varepsilon+t^{2}})+\varepsilon}{1+\varepsilon\cdot h_{\varepsilon}(\sqrt{\varepsilon+t^{2}})}. $$
By combining Lemma 3.3 in \cite{cianchimazya2011} and Lemma 4.5 in \cite{cianchimazya2014} (see also \cite{CM18}), we obtain
\begin{equation}
H_\varepsilon\in C^\infty([0,\infty)) \quad\text{and}\quad \varepsilon\leq H_\varepsilon(t)\leq \varepsilon^{-1}\quad\text{for all } t \geq 0.
\end{equation}
Furthermore, we have
\begin{equation}\label{quocientconditiontog_e}
\min\{\delta,1\}\leq Q_{g_\varepsilon}(t):=\frac{ g'_{\varepsilon}(t)t}{g_\varepsilon(t)}\leq \max\{g_0,1\} \quad \text{for all } t > 0,
\end{equation}
and
\begin{equation*}
\lim_{\varepsilon\to 0} H_\varepsilon(|\xi|)\xi = H(|\xi|)\xi \quad \text{uniformly in } \overline{B_M(0)}, \text{ for every } M>0.
\end{equation*}

Next, we adapt the Lieberman regularization approach from \cite{lieberman1991natural}. First, fix $0 < \tau < \frac{e^{-2/K}}{2}$, where the constant $K > 0$ will be specified precisely in \eqref{K assumption} below. Let $\eta_\tau(t)$ be the standard piecewise logarithmic cutoff:
\begin{equation} \label{eq:lieberman_shifted_cutoff}
\eta_\tau(t) = 
\begin{cases}
1, & 0 \le t < \tau, \\
1 - K \ln\left(\frac{t}{\tau}\right), & \tau \le t < \tau e^{1/K}, \\
0, & \tau e^{1/K} \le t < \frac{1}{2\tau e^{1/K}}, \\
K \ln\left(2\tau e^{1/K} t\right), & \frac{1}{2\tau e^{1/K}} \le t < \frac{1}{2\tau}, \\
1, & t \ge \frac{1}{2\tau}.
\end{cases}
\end{equation}
Defining $w_\tau(x) := \eta_\tau(e^x)$ for $x \in \mathbb{R}$, its weak derivative satisfies $|w_\tau'(x)| \le K$ almost everywhere. Using the mollifier $\psi_\varepsilon$ introduced above, we define $\tilde{w}_{\varepsilon,\tau} := w_\tau * \psi_\varepsilon$ and set the smooth cutoff $\eta_{\varepsilon,\tau}(t) := \tilde{w}_{\varepsilon,\tau}(\ln t)$ for $t > 0$ (extended by continuity as $\eta_{\varepsilon,\tau}(0) = 1$). Due to the compact support of $\psi_\varepsilon$, the transition boundaries of $\eta_\tau$ are expanded multiplicatively by $e^{\pm\varepsilon}$, yielding explicitly:
\begin{equation} \label{eq:explicit_eta_star}
\eta_{\varepsilon,\tau}(t) = 
\begin{cases}
1, & 0 \le t \le \tau e^{-\varepsilon}, \\
\tilde{w}_{\varepsilon,\tau}(\ln t), & \tau e^{-\varepsilon} < t < \tau e^{\frac{1}{K} + \varepsilon}, \\
0, & \tau e^{1/K + \varepsilon} \le t \le \frac{e^{-\varepsilon}}{2\tau e^{1/K}}, \\
\tilde{w}_{\varepsilon,\tau}(\ln t), & \frac{e^{-\varepsilon}}{2\tau e^{1/K}} < t < \frac{e^{\varepsilon}}{2\tau}, \\
1, & t \ge \frac{e^{\varepsilon}}{2\tau}.
\end{cases}
\end{equation}
This construction ensures that $\eta_{\varepsilon,\tau} \in C^\infty([0,\infty))$ and preserves the exact derivative bound:
\begin{equation} \label{eq:smooth_eta_bound}
|t (\eta_{\varepsilon,\tau})'(t)| = |\tilde{w}_{\varepsilon,\tau}'(\ln t)| \le \int_{-\varepsilon}^\varepsilon |w_\tau'(\ln t - y)|\psi_\varepsilon(y)\,dy \le K \int_{-\varepsilon}^\varepsilon \psi_\varepsilon(y)\,dy = K.
\end{equation}

Simultaneously, we mollify the vector field $\mathcal{A}_0(p)$ strictly away from the origin. Setting $\rho(x) := c\psi(|x|)$ for $x \in \mathbb{R}^n$, where $c := \left(\int_{\mathbb{R}^n}\psi(|x|)\,dx\right)^{-1}$, and $\rho_\varepsilon(x) := \frac{1}{\varepsilon^n}\rho\left(\frac{x}{\varepsilon}\right)$, we define $\mathcal{A}_{0,\varepsilon}(p) := (\mathcal{A}_0 * \rho_\varepsilon)(p)$. 

For our final ingredient, set $S_\tau(p) := \sqrt{|p|^2 + \tau^2}$ and define the smooth truncation function
\begin{equation} \label{eq:explicit_theta}
    \Theta_\tau(t) := \int_{-\tau/2}^{\tau/2} \min\left\{t-s, \, \frac{2}{\tau} \right\} \psi_{\tau/2}(s) \, ds.
\end{equation}
By construction, $\Theta_\tau \in C^\infty([0,\infty))$ is a non-negative, non-decreasing concave function with $\Theta_\tau(0) = 0$. Direct calculations yield
\begin{equation}\label{theta regularity}
    0 \le t \Theta_\tau'(t) \le \Theta_\tau(t) \quad \text{for all } t\in[0,\infty),
\end{equation}  
and
\begin{equation}
    \Theta_\tau(t)=
    \begin{cases}
    t, &\text{if } t\leq\frac{2}{\tau}-\frac{\tau}{2},\\[4pt]
    \frac{2}{\tau}, &\text{if } t\geq \frac{2}{\tau}+\frac{\tau}{2}.
    \end{cases}
\end{equation}

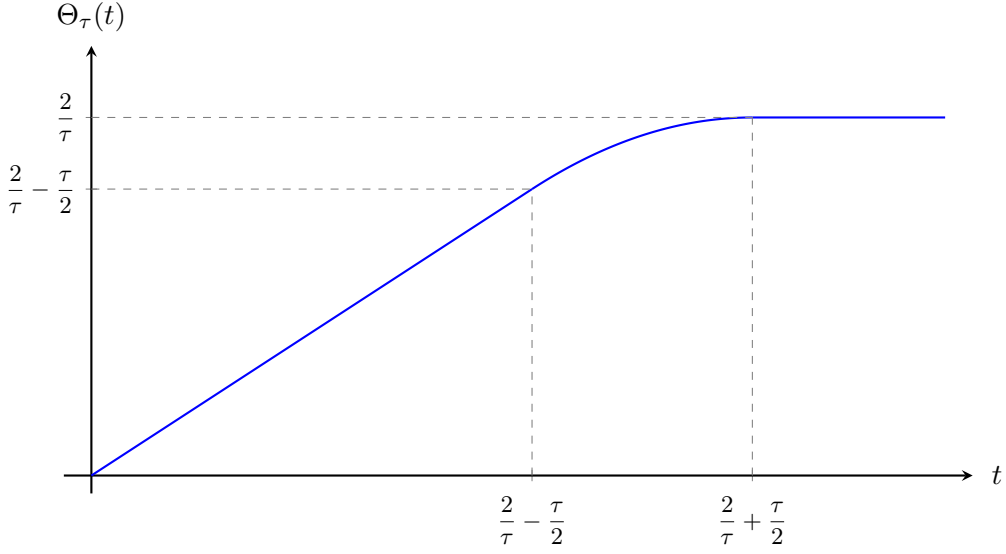
\begin{figure}[htbp]
    \centering
    \begin{tikzpicture}
    \begin{axis}[
        axis lines = center,
        xlabel = {$t$},
        ylabel = {$\Theta_\tau(t)$},
        ymin = -0.5, ymax = 12,
        xmin = -0.5, xmax = 16,
        xtick = {8, 12},
        xticklabels = {
            $\displaystyle\frac{2}{\tau}-\frac{\tau}{2}$,
            $\displaystyle\frac{2}{\tau}+\frac{\tau}{2}$
        },
        ytick = {8, 10},
        yticklabels = {
            $\displaystyle\frac{2}{\tau}-\frac{\tau}{2}$,
            $\displaystyle\frac{2}{\tau}$
        },
        every axis x label/.style={at={(ticklabel* cs:1.01)},anchor=west},
        every axis y label/.style={at={(ticklabel* cs:1.01)},anchor=south},
        height=7.5cm,
        width=0.85\linewidth,
        axis line style={thick},
        tick label style={font=\footnotesize},
        x tick label style={yshift=-1mm, inner sep=1mm}
    ]
    
    \addplot [thick, blue, domain=0:8, samples=2] {x};
    \addplot [thick, blue, domain=8:12, samples=50] {10 - 1/8*(x-12)^2};
    \addplot [thick, blue, domain=12:15.5, samples=2] {10};
    
    \draw[dashed, gray] (axis cs:8,0) -- (axis cs:8,8);
    \draw[dashed, gray] (axis cs:0,8) -- (axis cs:8,8);
    \draw[dashed, gray] (axis cs:12,0) -- (axis cs:12,10);
    \draw[dashed, gray] (axis cs:0,10) -- (axis cs:12,10);
    
    \end{axis}
    \end{tikzpicture}
    \caption{Profile of the truncation function $\Theta_\tau$.}
    \label{fig:theta_tau_function}
\end{figure}

We can now define the $C^\infty$ regularized operator:
\begin{equation} \label{eq:shifted_A_eps_tau_def}
\mathcal{A}_{\tau,\varepsilon}(p) := \big(1 - \eta_{\varepsilon,\tau}(|p|)\big) \mathcal{A}_{0,\varepsilon}(p) + \eta_{\varepsilon,\tau}(|p|) H_{\tau,\varepsilon}(p) p,
\end{equation}
where 
\begin{equation}\label{K assumption}
    H_{\tau,\varepsilon}(p) := h_{\tau,\varepsilon}\big(\Theta_\tau(S_\tau(p))\big) \quad \text{with} \quad h_{\tau,\varepsilon}(t) := \frac{g_\varepsilon(t)}{t}, \qquad \text{and} \qquad K := \frac{\tilde{\lambda}^\ast}{4(2\Lambda+\mu^2)}.
\end{equation}
Here, the doubling and ellipticity constants $\mu$ and $\tilde{\lambda}^\ast$ are defined by
$$ \mu := \max\left\{2^{\max\{g_0, 1\} - 1}, \, 2^{1 - \min\{\delta, 1\}}\right\} \quad \text{and} \quad \tilde{\lambda}^\ast := \min\left\{\frac{\lambda}{2\mu}, \, \frac{\min\{\delta, 1\}}{\mu^2}\right\}. $$
By \eqref{quocientconditiontog_e} and $(g\text{-}1)$ in Lemma \ref{firstLema}, for any $t > 0$ and $s \in [t, 2t]$ we have
\begin{equation} \label{eq:exact_doubling_constant}
\mu^{-1} \frac{g_\varepsilon(t)}{t} \le \frac{g_\varepsilon(s)}{s} \le \mu \frac{g_\varepsilon(t)}{t}.
\end{equation}

With $\tau$ fixed, notice that the active transition zone $Z_\tau := \big\{p \in \mathbb{R}^n : \tau e^{-\varepsilon_\star} \le |p| \le \tau e^{1/K + \varepsilon_\star} \text{ or } \frac{e^{-\varepsilon_\star}}{2\tau e^{1/K}} \le |p| \le \frac{e^{\varepsilon_\star}}{2\tau}\big\}$ (where $\nabla \eta_{\varepsilon,\tau} \neq 0$) is strictly contained in the compact annulus $\{ p \in \mathbb{R}^n : \tau/2 \le |p| \le 1/\tau \}$ for $\varepsilon_\star \le \ln 2$. Since the mapping $s \mapsto g(s)/s$ is continuous and strictly positive for $s > 0$, we can define the fixed positive constant
\begin{equation} \label{eq:m_tau_def}
    m_\tau := \min_{s \in [\tau/2, \, 1/\tau]} \frac{g(s)}{s} > 0.
\end{equation}
We then choose $\varepsilon_\star = \varepsilon_\star(\tau) \in (0, \min\{\tau/4, \ln 2\}]$ sufficiently small such that for all $0 < \varepsilon < \varepsilon_\star$, the following two properties hold:
\begin{enumerate}
\item Uniform convergence of the Cianchi--Maz'ya approximation: $\frac{1}{2} g \le g_\varepsilon \le 2 g$ on $\left[\frac{\tau}{4}, \frac{2}{\tau}\right]$. Consequently, $\mathcal{A}_0$ satisfies the structural conditions \eqref{homogeneous structural conditions} with $g_\varepsilon$ in this annulus.\label{C1}

\item The mollification error of the original field on $Z_\tau$ satisfies the uniform bound\label{C2}
\begin{equation} \label{eq:mollification_error_bound}
    \sup_{p \in Z_\tau} \left( \frac{K}{|p|} |\mathcal{A}_0(p) - \mathcal{A}_{0,\varepsilon}(p)| \right) \le \frac{\tilde{\lambda}^\ast}{8} m_\tau.
\end{equation}
\end{enumerate}
Since $|p| \in [\tau/2, 1/\tau]$ for all $p \in Z_\tau$, the definition of $m_\tau$ and condition \eqref{C1} naturally ensure that $\frac{1}{8} m_\tau \le \frac{1}{8} \frac{g(|p|)}{|p|} \le \frac{1}{4} \frac{g_\varepsilon(|p|)}{|p|}$. Thus, \eqref{eq:mollification_error_bound} immediately guarantees that $\frac{K}{|p|} |\mathcal{A}_0(p) - \mathcal{A}_{0,\varepsilon}(p)| \le \frac{\tilde{\lambda}^\ast}{4} \frac{g_\varepsilon(|p|)}{|p|}$ pointwisely for all $p \in Z_\tau$.

\begin{lemma}\label{lem:shifted_unified_operator}
For each fixed $\tau \in \left(0, \frac{e^{-2/K}}{2}\right)$, there exists $\varepsilon_\star(\tau)>0$ such that for all $0 < \varepsilon < \varepsilon_\star(\tau)$, the $C^\infty$ operator $\mathcal{A}_{\tau,\varepsilon}$ defined in \eqref{eq:shifted_A_eps_tau_def} satisfies the following structural conditions for all $\xi \in \mathbb{R}^n$: for $|p| < \frac{2}{\tau} + \frac{\tau}{2}$,
\begin{equation} \label{eq:SC_prime_global_shifted}
\begin{cases}
    |\mathcal{A}_{\tau,\varepsilon}(p)| + |D_p \mathcal{A}_{\tau,\varepsilon}(p)| |p| \le {\Lambda_\star} g_\varepsilon\big(S_\tau(p)\big), \\
    \langle D_p \mathcal{A}_{\tau,\varepsilon}(p)\xi, \xi \rangle \ge \lambda_\star \frac{g_\varepsilon\big(S_\tau(p)\big)}{S_\tau(p)} |\xi|^2,
\end{cases}
\end{equation}
where ${\lambda_\star}, {\Lambda_\star} > 0$ depend only on $\lambda, \Lambda, \mu, g_0,$ and $\delta$ (and are independent of $\tau$ and $\varepsilon$). Moreover, for $|p| \ge \frac{2}{\tau} + \frac{\tau}{2}$, the operator reduces to a constant multiple of the Laplace operator.
\end{lemma}

\begin{proof} 
First, we partition $\mathbb{R}^n$ into the three regions
\begin{align*}
    \mathcal{Z}_{1} &:= \left\{ p \in \mathbb{R}^n : |p| < \tau e^{-\varepsilon} \right\} \cup \left\{ p \in \mathbb{R}^n : \frac{e^\varepsilon}{2\tau} < |p| < \frac{2}{\tau} + \frac{\tau}{2} \right\}, \\
    \mathcal{Z}_{2} &:= \left\{ p \in \mathbb{R}^n : \tau e^{-\varepsilon} \le |p| \le \frac{e^\varepsilon}{2\tau} \right\}, \\
    \mathcal{Z}_{3} &:= \left\{ p \in \mathbb{R}^n : |p| \ge \frac{2}{\tau} + \frac{\tau}{2} \right\}.
\end{align*}

\begin{figure}[htbp]
    \centering
    \begin{tikzpicture}
    \begin{axis}[
        axis lines = center,
        xlabel = {$|p|$},
        ylabel = {$1 - \eta_{\varepsilon,\tau}(|p|)$},
        ymin = -0.2, ymax = 1.3,
        xmin = -0.2, xmax = 16.5, 
        xtick = {1, 3, 5.5, 7.5, 9.2, 12, 14.2}, 
        xticklabels = {
            $\tau e^{-\varepsilon}$, 
            $\tau e^{\frac{1}{K} + \varepsilon}$, 
            $\displaystyle\frac{e^{-\varepsilon}}{2\tau e^{1/K}}$, 
            $\displaystyle\frac{e^\varepsilon}{2\tau}$,
            $\displaystyle\frac{1}{\tau}$,
            $\displaystyle\frac{2}{\tau}-\frac{\tau}{2}$,
            $\displaystyle\frac{2}{\tau}+\frac{\tau}{2}$
        },
        ytick = {1},
        yticklabels = {$1$},
        every axis x label/.style={at={(ticklabel* cs:1.01)},anchor=west},
        every axis y label/.style={at={(ticklabel* cs:1.01)},anchor=south},
        height=6.5cm,
        width=0.95\linewidth,
        axis line style={thick},
        tick label style={font=\footnotesize},
        x tick label style={rotate=45, anchor=north east, yshift=-1mm, inner sep=1mm}
    ]
    
    \addplot [thick, red, smooth, tension=0.6] coordinates {
        (0,0) (0.8,0) (1.2, 0.1) (1.5, 0.32) (2.0, 0.63) (2.5, 0.85) (2.8, 0.97) (3.2, 1)
        (5.3, 1) (5.7, 0.97) (6.0, 0.85) (6.5, 0.63) (7.0, 0.32) (7.3, 0.1) (7.7, 0) 
        (8.5, 0) (12, 0) (16.5, 0) 
    };
    
    \draw[dashed, gray] (axis cs:3,0) -- (axis cs:3,1);
    \draw[dashed, gray] (axis cs:5.5,0) -- (axis cs:5.5,1);
    
    \draw[dashed, gray, thick] (axis cs:1,0) -- (axis cs:1,1.2);
    \draw[dashed, gray, thick] (axis cs:7.5,0) -- (axis cs:7.5,1.2);
    \draw[dashed, gray, thick] (axis cs:14.2,0) -- (axis cs:14.2,1.2);
    
    \node[font=\footnotesize\bfseries, blue, align=center] at (axis cs:0.5, 0.6) {$\mathcal{Z}_1$};
    \node[font=\footnotesize\bfseries, red, align=center] at (axis cs:4.25, 0.5) {$\mathcal{Z}_2$};
    \node[font=\footnotesize\bfseries, blue, align=center] at (axis cs:10.85, 0.5) {$\mathcal{Z}_1$};
    \node[font=\footnotesize\bfseries, black, align=center] at (axis cs:15.35, 0.5) {$\mathcal{Z}_3$};

    \end{axis}
    \end{tikzpicture}
    \caption{Profile of the cutoff function $1-\eta_{\varepsilon,\tau}$ and the partition into regions $\mathcal{Z}_i$ ($i=1,2,3$).}
    \label{fig:active_zone_smooth}
\end{figure}
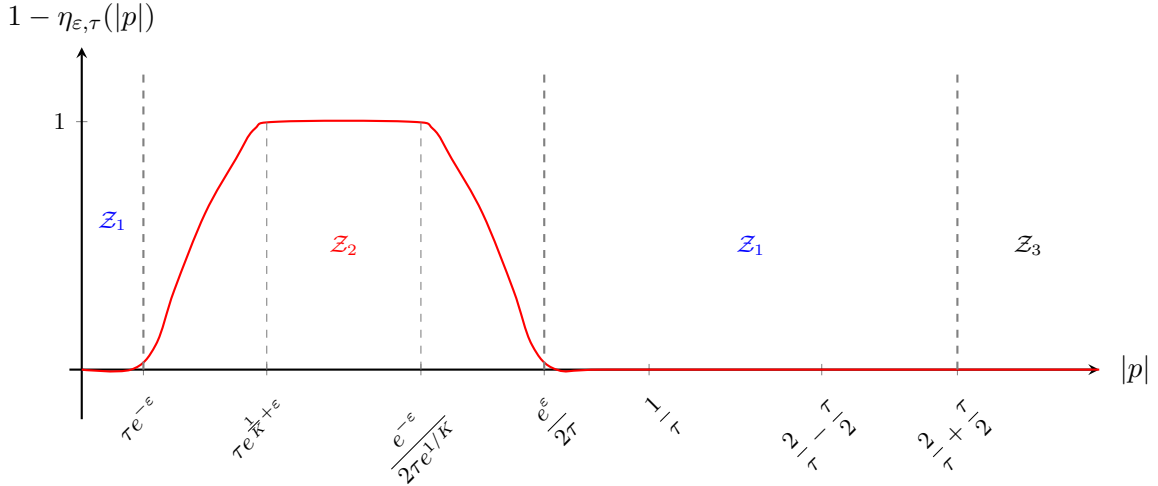

We divide the proof into three steps, verifying the structural conditions in each zone defined above.

\vspace{0.3cm}
\noindent\textit{Step \#1:} Structural conditions in the isotropic regime $\mathcal{Z}_{1}$.
\vspace{0.2cm}

This zone consists of two disjoint components. In $\{|p| < \tau e^{-\varepsilon}\}$, we have $S_\tau(p) < \frac{2}{\tau} - \frac{\tau}{2}$, so 
\begin{equation}\label{near origin}
    \Theta_\tau(S_\tau) = S_\tau.
\end{equation}
In $\left\{\frac{e^\varepsilon}{2\tau} < |p| < \frac{2}{\tau} + \frac{\tau}{2}\right\}$, whenever $S_\tau(p)$ enters the transition range of $\Theta_\tau$, we have
\begin{equation}\label{far away origin}
    \frac{2}{\tau} - \frac{\tau}{2} < S_\tau < \frac{2}{\tau} + \frac{\tau}{2}\implies \Theta_\tau(S_\tau) \in \left(\frac{2}{\tau} - \frac{\tau}{2}, \, \frac{2}{\tau}\right].
\end{equation}
Combining \eqref{near origin} and \eqref{far away origin}, we obtain across all of $\mathcal{Z}_1$:
\begin{equation}\label{theta s relation}
  \frac{3}{5} < \frac{4-\tau^2}{4+\tau^2} \le \frac{\Theta_\tau(S_\tau)}{S_\tau} \le 1 < \frac{4}{3}.
\end{equation}

For $p \in \mathcal{Z}_{1}$, since $\eta_{\varepsilon,\tau}(|p|) = 1$, we have
\begin{equation*}
    \mathcal{A}_{\tau,\varepsilon}(p) = \frac{g_\varepsilon\big(\Theta_\tau(S_\tau)\big)}{\Theta_\tau(S_\tau)} p = h_{\tau,\varepsilon}\big(\Theta_\tau(S_\tau)\big) p = H_{\tau,\varepsilon}(p) p.
\end{equation*}
Its Jacobian matrix is given by
\begin{equation*}
    D_p \mathcal{A}_{\tau,\varepsilon}(p) = H_{\tau,\varepsilon}(p) I_d + h_{\tau,\varepsilon}'\big(\Theta_\tau(S_\tau)\big) \Theta_\tau'(S_\tau) \frac{p \otimes p}{S_\tau}.
\end{equation*}
For any $\xi \in \mathbb{R}^n$, using that $0 \le \frac{\langle p, \xi \rangle^2}{S_\tau} \le S_\tau |\xi|^2$, if $h_{\tau,\varepsilon}'\big(\Theta_\tau(S_\tau)\big) < 0$ we obtain
\begin{align*}
    \langle D_p \mathcal{A}_{\tau,\varepsilon}(p)\xi, \xi \rangle &\ge \left[ H_{\tau,\varepsilon}(p) + h_{\tau,\varepsilon}'\big(\Theta_\tau(S_\tau)\big) \Theta_\tau'(S_\tau) S_\tau \right] |\xi|^2.
\end{align*}
Expanding $h_{\tau,\varepsilon}'(t) = \frac{g_\varepsilon'(t) - h_{\tau,\varepsilon}(t)}{t}$ explicitly yields a convex combination:
\begin{align*}
    H_{\tau,\varepsilon}(p) + h_{\tau,\varepsilon}'\big(\Theta_\tau(S_\tau)\big) \Theta_\tau'(S_\tau) S_\tau 
    &= H_{\tau,\varepsilon}(p) + \left[ g_\varepsilon'\big(\Theta_\tau(S_\tau)\big) - \frac{g_\varepsilon\big(\Theta_\tau(S_\tau)\big)}{\Theta_\tau(S_\tau)} \right] \frac{\Theta_\tau'(S_\tau) S_\tau}{\Theta_\tau(S_\tau)} \\
    &= H_{\tau,\varepsilon}(p) + \Big[ g_\varepsilon'\big(\Theta_\tau(S_\tau)\big) - H_{\tau,\varepsilon}(p) \Big] \alpha \\
    &= (1-\alpha) H_{\tau,\varepsilon}(p) + \alpha g_\varepsilon'\big(\Theta_\tau(S_\tau)\big),
\end{align*}
where $\alpha := \Theta_\tau'(S_\tau) \frac{S_\tau}{\Theta_\tau(S_\tau)} \in [0,1]$ by \eqref{theta regularity} (whereas if $h_{\tau,\varepsilon}'\big(\Theta_\tau(S_\tau)\big) \ge 0$, the quadratic form is immediately bounded below by $H_{\tau,\varepsilon}(p)|\xi|^2$). Thus, by \eqref{quocientconditiontog_e}, 
\begin{equation*}
    \langle D_p \mathcal{A}_{\tau,\varepsilon}(p)\xi, \xi \rangle \ge \min\{\delta, 1\} H_{\tau,\varepsilon}(p) |\xi|^2.
\end{equation*}

Similarly, using the exact same derivative identity and \eqref{quocientconditiontog_e}, the upper bound for the Jacobian simplifies to
\begin{align*}
    |D_p \mathcal{A}_{\tau,\varepsilon}(p)| &\le H_{\tau,\varepsilon}(p) + \big|h_{\tau,\varepsilon}'\big(\Theta_\tau(S_\tau)\big)\big| \Theta_\tau'(S_\tau) \frac{|p|^2}{S_\tau} \\
    &\le H_{\tau,\varepsilon}(p) + \big|h_{\tau,\varepsilon}'\big(\Theta_\tau(S_\tau)\big)\big| \Theta_\tau'(S_\tau) S_\tau \\
    &= H_{\tau,\varepsilon}(p) + \alpha \left| g_\varepsilon'\big(\Theta_\tau(S_\tau)\big) - H_{\tau,\varepsilon}(p) \right| \\
    &\le 2 (1 + \max\{g_0, 1\}) H_{\tau,\varepsilon}(p).
\end{align*}
Multiplying by $|p|$ and using \eqref{theta s relation}, we have $\frac{|p|}{\Theta_\tau(S_\tau)} \le \frac{S_\tau}{\Theta_\tau(S_\tau)} \le \frac{5}{3}$ inside $\mathcal{Z}_{1}$. The upper bound evaluates to
\begin{align*}
    |D_p \mathcal{A}_{\tau,\varepsilon}(p)| |p| &\le 2 (1 + \max\{g_0, 1\}) g_\varepsilon\big(\Theta_\tau(S_\tau)\big) \frac{|p|}{\Theta_\tau(S_\tau)} \\
    &\le \frac{10}{3} (1 + \max\{g_0, 1\}) g_\varepsilon\big(\Theta_\tau(S_\tau)\big).
\end{align*}
Furthermore, noting that $|\mathcal{A}_{\tau,\varepsilon}(p)| = g_\varepsilon\big(\Theta_\tau(S_\tau)\big) \frac{|p|}{\Theta_\tau(S_\tau)} \le \frac{5}{3} g_\varepsilon\big(\Theta_\tau(S_\tau)\big)$, we conclude that
\begin{equation*}
    |\mathcal{A}_{\tau,\varepsilon}(p)| + |D_p \mathcal{A}_{\tau,\varepsilon}(p)| |p| \le 5 (1 + \max\{g_0, 1\}) g_\varepsilon\big(\Theta_\tau(S_\tau)\big).
\end{equation*}
We thus recover \eqref{eq:SC_prime_global_shifted} in $\mathcal{Z}_1$ by combining \eqref{theta s relation} with $(g\text{-}1)$ of Lemma \ref{firstLema}.

\vspace{0.3cm}
\noindent\textit{Step \#2:} Structural conditions in terms of $g_\varepsilon(|p|)$ in the active annulus $\mathcal{Z}_{2}$.
\vspace{0.2cm}

For $\varepsilon \in (0, \varepsilon_\star)$, since $|p| \le \frac{e^\varepsilon}{2\tau}$ in $\mathcal{Z}_2$, we have $S_\tau(p) < \frac{2}{\tau} - \frac{\tau}{2}$. In this regime, $\Theta_\tau$ acts strictly as the identity:
\begin{equation} \label{eq:theta_identity}
    \Theta_\tau(S_\tau) = S_\tau.
\end{equation}
In the active annulus $\mathcal{Z}_{2}$, we first establish the structural conditions directly in terms of $g_\varepsilon(|p|)$. Since $|p| \ge \tau e^{-\varepsilon} > \tau/2$, the regularized magnitude $S_\tau(p)$ is comparable to the Euclidean modulus $|p|$:
\begin{equation} \label{eq:equivalence_S_tau_p}
    |p| \le S_\tau(p) = \sqrt{|p|^2+\tau^2} \le \sqrt{|p|^2 + 4|p|^2} = \sqrt{5}|p| < 3|p|.
\end{equation}
As shown in \eqref{eq:theta_identity}, $\Theta_\tau(S_\tau) = S_\tau$ (hence $\alpha = 1$). Applying the doubling property \eqref{eq:exact_doubling_constant} twice along \eqref{eq:equivalence_S_tau_p}, the shifted density satisfies
\begin{equation} \label{eq:H_equivalence_no_S}
    \mu^{-2} \frac{g_\varepsilon(|p|)}{|p|} \le H_{\tau,\varepsilon}(p) = \frac{g_\varepsilon(S_\tau)}{S_\tau} \le \mu^2 \frac{g_\varepsilon(|p|)}{|p|}.
\end{equation}

Next, for any $y \in B_\varepsilon(0)$, we have $|y| < \varepsilon_\star \le \tau/4 \le |p|/2$. By the triangle inequality, 
\begin{equation} \label{eq:triangle_annulus}
    \frac{1}{2}|p| \le |p-y| \le \frac{3}{2}|p| \quad \implies \quad |p-y| \in \left[\frac{\tau}{4}, \frac{2}{\tau}\right].
\end{equation}
By \eqref{eq:triangle_annulus}, property \eqref{C1}, and \eqref{eq:exact_doubling_constant}, the mollified field $\mathcal{A}_{0,\varepsilon}$ satisfies
\begin{align*}
    \langle D_p \mathcal{A}_{0,\varepsilon}(p)\xi, \xi \rangle &= \int_{B_\varepsilon(0)} \langle D_p \mathcal{A}_0(p-y)\xi, \xi \rangle \rho_\varepsilon(y) \, dy \ge \frac{\lambda}{2\mu} \frac{g_\varepsilon(|p|)}{|p|} |\xi|^2,
\end{align*}
and
\begin{align*}
    |\mathcal{A}_{0,\varepsilon}(p)| + |D_p \mathcal{A}_{0,\varepsilon}(p)| |p| &\le 8\Lambda\mu g_\varepsilon(|p|).
\end{align*}

In the core of $\mathcal{Z}_{2}$ where $\eta_{\varepsilon,\tau} = 0$, the operator reduces to $\mathcal{A}_{\tau,\varepsilon} = \mathcal{A}_{0,\varepsilon}$, and the bounds in terms of $|p|$ hold immediately. 
In the transition zones where $0 < \eta_{\varepsilon,\tau} < 1$, we compute the Jacobian of $\mathcal{A}_{\tau,\varepsilon}$:
\begin{align} \label{eq:shifted_jacobian}
D_p \mathcal{A}_{\tau,\varepsilon}(p) &= (1-\eta_{\varepsilon,\tau}) D_p \mathcal{A}_{0,\varepsilon}(p) + \eta_{\varepsilon,\tau} D_p \big(H_{\tau,\varepsilon}(p) p\big) \nonumber \\
&\quad + \big(H_{\tau,\varepsilon}(p) p - \mathcal{A}_{0,\varepsilon}(p)\big) \otimes \nabla_p \eta_{\varepsilon,\tau}(|p|) \nonumber \\
&=: M_1 + M_2.
\end{align}
Using \eqref{eq:H_equivalence_no_S} and the ellipticity bound for $\mathcal{A}_{0,\varepsilon}$, we obtain
\begin{align*}
    \langle M_1\xi,\xi\rangle &\ge \left[ (1-\eta_{\varepsilon,\tau}) \frac{\lambda}{2\mu} + \eta_{\varepsilon,\tau} \min\{\delta, 1\} \mu^{-2} \right] \frac{g_\varepsilon(|p|)}{|p|} |\xi|^2 \ge \tilde{\lambda}^\ast \frac{g_\varepsilon(|p|)}{|p|} |\xi|^2,
\end{align*}
where $\tilde{\lambda}^\ast = \min\big\{\frac{\lambda}{2\mu}, \min\{\delta, 1\} \mu^{-2}\big\}$ was defined in \eqref{K assumption}.

Next, we bound the cross term $M_2$. By \eqref{eq:smooth_eta_bound}, we have $|\nabla_p \eta_{\varepsilon,\tau}(|p|)| \le \frac{K}{|p|}$. Adding and subtracting $\mathcal{A}_0(p)$ yields
\begin{align*}
   |\langle M_2\xi,\xi\rangle| &\le \frac{K}{|p|} \Big( \big| H_{\tau,\varepsilon}(p) p - \mathcal{A}_0(p) \big| + \big|\mathcal{A}_0(p) - \mathcal{A}_{0,\varepsilon}(p)\big| \Big) |\xi|^2 \\
    &\le K(2\Lambda + \mu^2) \frac{g_\varepsilon(|p|)}{|p|} |\xi|^2 + \frac{K}{|p|} \big|\mathcal{A}_0(p) - \mathcal{A}_{0,\varepsilon}(p)\big| |\xi|^2.
\end{align*}
By our choice of $K$ in \eqref{K assumption} and condition \eqref{C2}, each of the two terms on the right-hand side is bounded by $\frac{\tilde{\lambda}^\ast}{4} \frac{g_\varepsilon(|p|)}{|p|} |\xi|^2$, so that $|\langle M_2\xi,\xi\rangle| \le \frac{\tilde{\lambda}^\ast}{2} \frac{g_\varepsilon(|p|)}{|p|} |\xi|^2$. Therefore,
\begin{align*}
    \langle D_p \mathcal{A}_{\tau,\varepsilon}(p)\xi, \xi \rangle &= \langle M_1\xi, \xi \rangle + \langle M_2\xi, \xi \rangle \\
    &\ge \tilde{\lambda}^\ast \frac{g_\varepsilon(|p|)}{|p|} |\xi|^2 - |\langle M_2\xi, \xi \rangle| \\
    &\ge \frac{\tilde{\lambda}^\ast}{2} \frac{g_\varepsilon(|p|)}{|p|} |\xi|^2.
\end{align*}

To establish the upper growth bound in $\mathcal{Z}_2$, we use \eqref{eq:H_equivalence_no_S} and the growth bound for $\mathcal{A}_{0,\varepsilon}(p)$:
\begin{align*}
    |\mathcal{A}_{\tau,\varepsilon}(p)| &\le (1-\eta_{\varepsilon,\tau}) |\mathcal{A}_{0,\varepsilon}(p)| + \eta_{\varepsilon,\tau} H_{\tau,\varepsilon}(p) |p| \\
    &\le (8\Lambda\mu + \mu^2) g_\varepsilon(|p|).
\end{align*}
Moreover, from \eqref{eq:shifted_jacobian} we have $|D_p\mathcal{A}_{\tau,\varepsilon}(p)||p| \leq |M_1||p| + |M_2||p|$. From \textit{Step \#1}, $|D_p(H_{\tau,\varepsilon}(p)p)| \le 2(1+\max\{g_0, 1\}) H_{\tau,\varepsilon}(p)$, which gives
\begin{align*}
    |M_1| |p| &\le (1-\eta_{\varepsilon,\tau}) |D_p \mathcal{A}_{0,\varepsilon}(p)| |p| + \eta_{\varepsilon,\tau} \big|D_p \big(H_{\tau,\varepsilon}(p) p\big)\big| |p| \\
    &\le 8\Lambda\mu g_\varepsilon(|p|) + 2(1+\max\{g_0, 1\}) H_{\tau,\varepsilon}(p) |p| \\
    &\le \Big[ 8\Lambda\mu + 2(1+\max\{g_0, 1\})\mu^2 \Big] g_\varepsilon(|p|).
\end{align*}
For $M_2$, using $|\nabla_p \eta_{\varepsilon,\tau}(|p|)| \le \frac{K}{|p|}$ directly, we get
\begin{align*}
    |M_2| |p| &\le |\nabla_p \eta_{\varepsilon,\tau}(|p|)| |p| \, \big| H_{\tau,\varepsilon}(p) p - \mathcal{A}_{0,\varepsilon}(p) \big| \\
    &\le K \big( H_{\tau,\varepsilon}(p) |p| + |\mathcal{A}_{0,\varepsilon}(p)| \big) \\
    &\le K(\mu^2 + 8\Lambda\mu) g_\varepsilon(|p|).
\end{align*}
Summing these three estimates yields
\begin{align*}
    |\mathcal{A}_{\tau,\varepsilon}(p)| + |D_p \mathcal{A}_{\tau,\varepsilon}(p)| |p| &\le |\mathcal{A}_{\tau,\varepsilon}(p)| + |M_1| |p| + |M_2| |p| \\
    &\le \tilde{\Lambda}^\ast g_\varepsilon(|p|),
\end{align*}
where $\tilde{\Lambda}^\ast := (1+K)(8\Lambda\mu + \mu^2) + 8\Lambda\mu + 2(1+\max\{g_0, 1\})\mu^2$. Finally, we recover \eqref{eq:SC_prime_global_shifted} in terms of $S_\tau(p)$ by applying \eqref{eq:equivalence_S_tau_p} and \eqref{eq:H_equivalence_no_S}.

\vspace{0.3cm}
\noindent\textit{Step \#3:} Behavior of the operator in the outer region $\mathcal{Z}_{3}$.
\vspace{0.2cm}

In $\mathcal{Z}_3$, we have $|p| \ge \frac{2}{\tau} + \frac{\tau}{2}$, which implies $S_\tau(p) \ge \frac{2}{\tau} + \frac{\tau}{2}$. Thus, the truncation function $\Theta_\tau$ attains its constant upper plateau:
\begin{equation} \label{eq:theta_plateau}
    \Theta_\tau(S_\tau) = \frac{2}{\tau}.
\end{equation}
Since $\eta_{\varepsilon,\tau}(|p|) = 1$ in this region, the regularized operator becomes purely linear:
\begin{equation*}
    \mathcal{A}_{\tau,\varepsilon}(p) = \frac{g_\varepsilon(2/\tau)}{2/\tau} p \quad \text{and} \quad D_p \mathcal{A}_{\tau,\varepsilon}(p) = \frac{g_\varepsilon(2/\tau)}{2/\tau} I_d.
\end{equation*}
Consequently, for any solution $v$ in the region where $|\nabla v| \ge \frac{2}{\tau} + \frac{\tau}{2}$, the homogeneous equation reduces to the classical Laplace equation:
\begin{equation*}
    \operatorname{div} \mathcal{A}_{\tau,\varepsilon}(\nabla v) = \frac{g_\varepsilon(2/\tau)}{2/\tau} \operatorname{div}(\nabla v) = \frac{g_\varepsilon(2/\tau)}{2/\tau} \Delta v = 0.
\end{equation*}
This completes the proof.
\end{proof}

\begin{remark} \label{rem:SC_active_region}
Under the assumptions of Lemma \ref{lem:shifted_unified_operator}, if $p \in \mathcal{Z}_{2}$, the operator $\mathcal{A}_{\tau,\varepsilon}$ strictly satisfies the structural conditions directly with respect to the Euclidean modulus $|p|$:
\begin{equation} \label{eq:SC_prime_particular}
\begin{cases}
    |\mathcal{A}_{\tau,\varepsilon}(p)| + |D_p \mathcal{A}_{\tau,\varepsilon}(p)| |p| \le {\Lambda}^\ast g_\varepsilon(|p|), \\
    \langle D_p \mathcal{A}_{\tau,\varepsilon}(p)\xi, \xi \rangle \ge {\lambda}^\ast \frac{g_\varepsilon(|p|)}{|p|} |\xi|^2,
\end{cases}
\end{equation}
where ${\lambda}^\ast := \frac{\tilde{\lambda}^\ast}{2} > 0$ and ${\Lambda}^\ast := \tilde{\Lambda}^\ast > 0$ depend only on $\lambda, \Lambda, \mu, g_0,$ and $\delta$. This is a direct consequence of the explicit estimates derived in \textit{Step \#2} of the proof, which established these bounds natively in terms of $g_\varepsilon(|p|)$ prior to passing to $S_\tau(p)$.
\end{remark}

\section{Regularity $C^{1,\alpha}$ up to the boundary}\label{section5}

In this section we will obtain the regularity $C^{1,\alpha}$ up to the boundary to solutions of \eqref{1dir}. By Lemma \ref{lem:invariance}, the structural conditions in \eqref{structural conditions} are invariant under $C^{1,\alpha}$ diffeomorphisms. So it's sufficient to consider our estimates in $B_1^+$. Then, we will deal with weak solutions of \eqref{1dir} with $R=1$

Our goal is to prove this result,
\begin{theorem}[Up to the boundary regularity] \label{thm:up_to_boundary}
Let $u \in W^{1,G}( B_1^+)\cap L^\infty(B_1^+)$ be a weak solution to \eqref{1dir}. Then $u \in C^{1,\beta}(B_{1/2}^+)$ for some $\beta = \beta(n,\delta,g_0,\lambda,\Lambda,\Lambda_\alpha,\Lambda_B,\alpha_{\mathcal{A}}) \in (0,1)$. Moreover, defining the normalization constant 
$$ \kappa := \|u\|_{L^\infty(B_1^+)} + \|\psi\|_{C^{1,\alpha}(T_1)} + g^{-1}\left( \|f\|_{L^q(B_1^+)} \right), $$
there exist constants $C_\star, c_\star > 0$ depending only on $n,\delta,g_0,q,g(1),\lambda,\Lambda,\Lambda_\alpha,\alpha_{\mathcal{A}}$ such that
\begin{equation*}
    \|u\|_{C^{1,\beta}(B_{1/2}^+)} \le C_\star e^{c_\star \Lambda_B \kappa+\|\gamma\|_{L^s(B_1^+)}} \kappa.
\end{equation*}
\end{theorem}

The first step to obtain Theorem \ref{thm:up_to_boundary} is to study the regularity of 
\begin{equation} \label{eq:frozen_homo}
\begin{cases}
\operatorname{div} \mathcal{A}_0(\nabla v) = 0 & \text{in } B_{1}^{+},\\
v=\psi & \text{on } T_{1},
\end{cases}
\end{equation}
where $\mathcal{A}_0(p) := \mathcal{A}(0, u(0), p)$. By \eqref{structural conditions}, the function $A_0:\Bbb{R}^n\to\Bbb{R}^n$ satisfies

\begin{equation} \tag{SC'} \label{homogeneous structural conditions}
\begin{cases}
    & |\mathcal{A}_0(p)| + |D_p \mathcal{A}_0(p)| |p| \le \Lambda g(|p|), \\
    & \lambda \frac{g(|p|)}{|p|} |\xi|^2 \le \langle D_p \mathcal{A}_0(p)\xi, \xi \rangle,
\end{cases}
\end{equation}

For any fixed parameters $\varepsilon, \tau > 0$, the Lemma \ref{lem:shifted_unified_operator} guaranties that the $\operatorname{div} \mathcal{A}_{\tau,\varepsilon}(\nabla v_{\varepsilon, \tau}) = 0$ admits a strong solution $v_{\varepsilon, \tau} \in C^{\infty}(B_1^+)$. With this regularity we can put the operator in non divergence form in $B_1^+$

\begin{equation*}
\sum_{i,j=1}^n a^{ij}_{\varepsilon, \tau}(x) D_{ij} v_{\varepsilon, \tau} = 0,\quad \text{where }\ \ a^{ij}_{\varepsilon, \tau}(x) = \frac{\partial \mathcal{A}_{\tau,\varepsilon}^i}{\partial p_j}(\nabla v_{\varepsilon, \tau}).
\end{equation*}
Again, Lemma \ref{lem:shifted_unified_operator} guarantees that these bounded coefficients satisfy the precise uniform ellipticity bounds:
\begin{equation*}
{\tilde{\lambda}} \frac{g_{\varepsilon}(|\nabla v_{\varepsilon, \tau}|)}{|\nabla v_{\varepsilon, \tau}|}|\xi|^{2} \le \sum_{i,j=1}^n a^{ij}_{\varepsilon, \tau}(x)\xi_{i}\xi_{j} \le \tilde{\Lambda} \frac{g_{\varepsilon}(|\nabla v_{\varepsilon, \tau}|)}{|\nabla v_{\varepsilon, \tau}|}|\xi|^{2}.
\end{equation*}

Consequently, the effective ellipticity ratio of the operator is universally bounded by $\frac{4\tilde{\Lambda}}{\tilde{\lambda}}$, a constant entirely independent of the approximation parameters $\varepsilon$ and $\tau$. This independence is the cornerstone of the analysis: it ensures that the \textit{a priori} bounds do not degenerate as we approach the singular limit. Thus, following the methodology of Braga and Sousa \cite{BragaSousa2025}, we can seamlessly apply the classical techniques of Lieberman \cite{lieberman1986, lieberman1988}---namely, the geometric barrier constructions for the global Lipschitz estimates and the Krylov-Safonov iteration for the boundary Harnack inequalities. With these uniform estimates secured for the strong solutions $v_{\varepsilon, \tau}$, we employ a classical two-tier compactness argument. Passing to the limit as $\varepsilon, \tau \to 0$, a subsequence of $v_{\varepsilon, \tau}$ converges in $C^1$ to the weak solution $v$ of the original non-regularized equation, preserving the optimal $C^{1,\sigma}$ bounds intact. Therefore, the boundary regularity result follows by identical steps to the purely $g$-Laplacian framework, yielding:

\begin{lemma}[Comparison and Boundary Regularity Estimates] \label{lemma:comparison_boundary}
Let $\mathcal{A}_0(p)$ be a strictly elliptic operator satisfying the uniform structural conditions with constants $\lambda$ and $\Lambda$. Let $u \in W^{1,G}(B_R^+) \cap L^\infty(B_R^+)$ be a weak solution to the full equation \eqref{1dir} such that $u = \psi$ on $T_R$, where $\psi \in C^{1,\alpha_\psi}(T_R)$. Consider $v \in W^{1,G}(B_R^+)$ as the unique weak solution to the frozen Dirichlet problem:
\begin{equation} \label{eq:comparison_problem}
\begin{cases}
    \operatorname{div} \mathcal{A}_0(\nabla v) = 0 & \text{in } B_R^+, \\
    v = u & \text{on } \partial B_R^+.
\end{cases}
\end{equation}
Then $v \in C^{1,\sigma}(B_{3R/4}^+)$ for some $\sigma = \sigma(n,\lambda,\Lambda,\delta,g_0,\alpha_\psi) \in (0,1)$. Moreover, there exists a uniform constant $C > 0$, depending only on $n, \lambda, \Lambda, \delta, g_0, g(1),$ and $\alpha_\psi$, such that the following properties hold.

First, $v$ satisfies the Lipschitz bound $v \in C^{0,1}(B_{3R/4}^+)$ with
\begin{equation}
    \sup_{B_{3R/4}^+} |\nabla v| \le C \left[ \frac{\operatorname{osc}_{B_R^+} v}{R} + \|\psi\|_{C^{1,\alpha_\psi}(T_R)} \right].
\end{equation}

Second, for the boundary oscillation of the difference quotient, we have for any $0 < r < R/2$,
\begin{equation}
   \operatorname*{osc}_{B_r^+} \left( \frac{v - \psi}{x_n} \right) \le C \left( \frac{r}{R} \right)^\sigma \left[\operatorname*{osc}_{B_R^+} \left( \frac{v - \psi}{x_n} \right) + \|\psi\|_{C^{1,\alpha_\psi}(T_R)} R^{\alpha_\psi} \right].
\end{equation}

Third, the pointwise Hölder continuity of the gradient yields that for any $0 < r < R/2$,
\begin{equation}
   \operatorname*{osc}_{B_r^+} \nabla v \le C \left( \frac{r}{R} \right)^\sigma \left[ \sup_{B_R^+} |\nabla v| + \|\psi\|_{C^{1,\alpha_\psi}(T_R)} R^{\alpha_\psi} \right].
\end{equation}

Fourth, defining the scaled norms
\begin{equation}
    \|v\|^\star_{C^{1,\sigma}(B_{3R/4}^+)} := \|v\|_{L^\infty(B_{3R/4}^+)} + R\|\nabla v\|_{L^\infty(B_{3R/4}^+)} + R^{1+\sigma}[\nabla v]_{C^{0,\sigma}(B_{3R/4}^+)}
\end{equation}
and $\|\psi\|^\star_{C^{1,\alpha_\psi}(T_R)} := \|\psi\|_{L^\infty(T_R)} + R\|\nabla \psi\|_{L^\infty(T_R)} + R^{1+\alpha_\psi}[\nabla \psi]_{C^{0,\alpha_\psi}(T_R)}$, the following global scaled estimate holds:
\begin{equation}
    \|v\|^*_{C^{1,\sigma}(B_{3R/4}^+)} \le C \left( \|v\|_{L^\infty(B_R^+)} + \|\psi\|^\star_{C^{1,\alpha_\psi}(T_R)} \right).
\end{equation}

Fifth, the supremum of the gradient energy is bounded by its integral average:
\begin{equation} \label{eq:sup_gradient_energy}
    \sup_{B_{R/2}^+} G(|\nabla v|) \le C R^{-n} \int_{B_R^+} G(|\nabla v|) \, dx.
\end{equation}

Sixth, the gradient satisfies the Campanato-type integral decay for any $0 < r < R$:
\begin{equation} \label{eq:campanato_decay}
    \frac{1}{|B_r^+|} \int_{B_r^+} G(|\nabla v - (\nabla v)_{B_r^+}|) \, dx \le C \left( \frac{r}{R} \right)^\sigma \frac{1}{|B_R^+|} \int_{B_R^+} G(|\nabla v - (\nabla v)_{B_R^+}|) \, dx,
\end{equation}
where $(\nabla v)_{B_r^+} = \frac{1}{|B_r^+|} \int_{B_r^+} \nabla v \, dx$ denotes the integral average over the half-ball.

Seventh, the following energy comparison holds:
\begin{equation} \label{eq:energy_by comparison}
    \int_{B_R^+} G(|\nabla v|) \, dx \le C \int_{B_R^+} G(|\nabla u|) \, dx.
\end{equation}

Finally, 
\begin{equation} \label{eq:L_infty_comparison}
    \sup_{B_R^+} |u - v| \le \operatorname*{osc}_{B^+_R}u.
\end{equation}
\end{lemma}

We begin this section with the fundamental global energy estimate up to the flattened boundary for the full quasilinear equation.

\begin{lemma}\label{small norms}
Let $u \in W^{1,G}(B_1^+)\cap L^\infty(B_1^+)$ be a weak solution to \eqref{1dir}. If
$$ \max\left\{\|\psi\|_{C^{1,\alpha_\psi}(T_{1})}, \, \|u\|_{L^{\infty}(B_{1}^{+})}, \, \|f\|_{L^{q}(B_{1}^{+})}\right\} \le 1, $$
then $\nabla u \in C^{0,\beta}(B_{1/2}^+)$ for some $\beta = \beta(n, s, \delta, g_0, \Lambda, \Lambda_B, \lambda, \|\gamma\|_{L^s(B_1^+)}, \alpha_\psi, \alpha_{\mathcal{A}}, q) \in (0,1)$. Moreover, there exist $C_\star = C_\star(n, s, \delta, g_0, g(1), \alpha_\psi, \alpha_{\mathcal{A}}, \lambda, \Lambda, \Lambda_\alpha, q) > 0$ and $c_\star = c_\star(n, s, \delta, g_0, \lambda, \alpha_\psi, \alpha_{\mathcal{A}}, q) > 0$ such that
\begin{equation*}
    [\nabla u]_{C^{0,\beta}(B_{1/2}^+)} \leq C_\star e^{c_\star\left(\Lambda_B + \|\gamma\|_{L^s(B_1^+)}\right)}.
\end{equation*}
\end{lemma}

\begin{proof}
Throughout the proof, to simplify the notation, we denote by
\begin{equation*}
    \mathcal{E} := C e^{c\left(\Lambda_B + \|\gamma\|_{L^s(B_1^+)}\right)}
\end{equation*}

Fix \(R\in(0,1)\) and let \(v\in W^{1,G}(B_{R}^{+})\cap L^{\infty}(B_{R}^{+})\) be the weak solution of the homogeneous problem:
$$ \operatorname{div} \mathcal{A}_0(\nabla v) = 0 \quad \text{in } B_{R}^{+}, \qquad v = u \quad \text{on } \partial B_{R}^{+}, $$
where \(\mathcal{A}_0(p) := \mathcal{A}(0, 0, p)\). By the Maximum Principle,
$$ \sup_{B_{R}^{+}}|u-v| \le \operatorname*{osc}_{B_{R}^{+}} u =: \omega(R) \le 2\|u\|_{L^{\infty}(B_{1}^{+})} \le 2. $$
Subtracting the weak formulations of \(u\) and \(v\) against the test function \(u-v\in W_{0}^{1,G}(B_{R}^{+})\), we obtain
\begin{align*}
I &:= \int_{B_{R}^{+}} [ \mathcal{A}_0(\nabla u) - \mathcal{A}_0(\nabla v) ] \cdot (\nabla u-\nabla v) \, dx \\
&= \int_{B_{R}^{+}} [ \mathcal{A}_0(\nabla u) - \mathcal{A}(x, u, \nabla u) ] \cdot (\nabla u-\nabla v) \, dx \\
&\quad - \int_{B_{R}^{+}} \mathcal{B}(x, u, \nabla u)(u(x)-v(x)) \, dx \\
&=: III + II.
\end{align*}

For the convective term \(II\), utilizing the structural growth bound \(|\mathcal{B}(x, z, p)| \le \Lambda_B G(|p|) + |\gamma(x)| g(|p|) + |f(x)|\) and the oscillation bound \(|u-v| \le \omega(R)\), we have:
\begin{equation}\label{eq:II_split}
    |II| \le \Lambda_B \omega(R) \int_{B_{R}^{+}} G(|\nabla u|)\,dx + \int_{B_R^+} |f||u-v| \, dx + \int_{B_R^+}|\gamma(x)|g(|\nabla u|)|u-v|\,dx.
\end{equation}
For the source term in \eqref{eq:II_split}, by H\"older's inequality, the Sobolev embedding in $W_0^{1,1}(B_R^+)$, (G-3) of Lemma \ref{firstLema}, and \eqref{eq:energy_by comparison}, we obtain:
\begin{align}\label{eq:II_source}
    \int_{B_R^+} |f||u-v| \, dx &\le |B_{R}^{+}|^{\frac{1}{n}-\frac{1}{q}} \|f\|_{L^{q}(B_R^+)} \left( \int_{B_{R}^{+}} |u-v|^{\frac{n}{n-1}} \, dx \right)^{\frac{n-1}{n}} \nonumber\\
    &\le C(n) |B_{R}^{+}|^{\frac{1}{n}-\frac{1}{q}} \int_{B_{R}^{+}} |\nabla(u-v)| \,dx \nonumber\\
    &\le C |B_{R}^{+}|^{\frac{1}{n}-\frac{1}{q}} \int_{B_{R}^{+}} G(|\nabla u-\nabla v|) \, dx + C|B_{R}^{+}|^{1+\frac{1}{n}-\frac{1}{q}} \nonumber\\
    &\le C R^{1-\frac{n}{q}} \int_{B_{R}^{+}} G(|\nabla u|) \, dx + C R^{n+1-\frac{n}{q}}.
\end{align}

To estimate the drift term in \eqref{eq:II_split}, let $\mu_{\star} := 1+\delta - \frac{n(1+g_0)}{s} > 0$. Applying Young's inequality (Lemma \ref{youngorlicz}) with $\varepsilon = R^{\frac{\mu_{\star} \delta}{1+\delta}}$, we get:
\begin{equation}\label{eq:II_gamma_young}
    \int_{B_R^+} |\gamma(x)||u-v|g(|\nabla u|)\,dx \le R^{\frac{\mu_{\star} \delta}{1+\delta}} \int_{B_R^+} G(|\nabla u|)\,dx + C R^{-\frac{\mu_{\star}}{1+\delta}} \int_{B_R^+} G\big(|\gamma(x)||u-v|\big)\,dx.
\end{equation}
Setting $w := G\big(\frac{1}{2}|u-v|\big)^{\frac{1}{1+\delta}} \in W_0^{1,1+\delta}(B_R^+)$, so that $w^{1+\delta} = G\big(\frac{1}{2}|u-v|\big)$, by (G-2) of Lemma \ref{firstLema} and H\"older's inequality with $\sigma = \frac{s}{1+g_0} > 1$, we have:
\begin{align*}
    \int_{B_R^+} G\big(|\gamma(x)||u-v|\big)\,dx &\le C \int_{B_R^+} (1 + |\gamma(x|)^{1+g_0} w^{1+\delta}\,dx \\
    &\le C \|1+\gamma\|_{L^s(B_R^+)}^{1+g_0} \|w\|_{L^{(1+\delta)\sigma'}(B_R^+)}^{1+\delta}.
\end{align*}
Applying the Gagliardo--Nirenberg inequality \cite[Theorem 12.83]{leoni2017first} with $\theta = 1 - \frac{n(1+g_0)}{s(1+\delta)} = \frac{\mu_{\star}}{1+\delta} \in (0,1)$ combined with Poincar\'e's inequality $\|w\|_{L^{1+\delta}(B_R^+)}^{1+\delta} \le C R^{1+\delta} \|\nabla w\|_{L^{1+\delta}(B_R^+)}^{1+\delta}$ in $W_0^{1,1+\delta}(B_R^+)$, we deduce:
\begin{align*}
    \|w\|_{L^{(1+\delta)\sigma'}(B_R^+)}^{1+\delta} \le C \left(\|w\|_{L^{1+\delta}(B_R^+)}^{1+\delta}\right)^\theta \left(\|\nabla w\|_{L^{1+\delta}(B_R^+)}^{1+\delta}\right)^{1-\theta} &\le C R^{(1+\delta)\theta} \|\nabla w\|_{L^{1+\delta}(B_R^+)}^{1+\delta} \\
    &= C R^{\mu_{\star}} \int_{B_R^+} |\nabla w|^{1+\delta}\,dx.
\end{align*}
By Lemma \ref{Gv in W11}, $|\nabla w|^{1+\delta} \le C \big(G(|\nabla(u-v)|) + G(\frac{1}{2}|u-v|)\big)$ almost everywhere and $\frac{1}{2}|u-v| \le 1$, using \eqref{eq:energy_by comparison} we obtain:
\begin{equation*}
    \int_{B_R^+} G\big(|\gamma(x)||u-v|\big)\,dx \le C e^{c\|\gamma\|_{L^s(B_1^+)}} R^{\mu_{\star}} \left( \int_{B_R^+} G(|\nabla u|)\,dx + R^n \right).
\end{equation*}
Substituting this back into \eqref{eq:II_gamma_young} and combining with \eqref{eq:II_source} and \eqref{eq:II_split}, we conclude that:
\begin{align*}
    |II| &\le C e^{c\|\gamma\|_{L^s(B_1^+)}} \left( \Lambda_B \omega(R) + R^{1-\frac{n}{q}} + R^{\frac{\mu_{\star}\delta}{1+\delta}} \right) \int_{B_R^+} G(|\nabla u|)\,dx \\
    &\quad+ C e^{c\|\gamma\|_{L^s(B_1^+)}} \left( R^{n+1-\frac{n}{q}} + R^{n + \frac{\mu_{\star}\delta}{1+\delta}} \right).
\end{align*}

For the perturbation term \(III\), we use the H\"older continuity of the full operator \(\mathcal{A}(x, z, p)\). Let \(\sigma(R) := R^{\alpha_{\mathcal{A}}} + \omega(R)^{\alpha_{\mathcal{A}}}\). Then \(|\mathcal{A}(0, 0, p) - \mathcal{A}(x, u(x), p)| \le \Lambda_\alpha(1+g(|p|)) \sigma(R)\) in \(B_R^+\). Thus, by Young's inequality, (g-2), (G-3) of Lemma \ref{firstLema}, and \eqref{eq:energy_by comparison},
\begin{align*}
   |III| &\le \Lambda_\alpha \sigma(R) \int_{B_R^+} (1+g(|\nabla u|)) |\nabla u - \nabla v| \, dx\\
   &\leq \Lambda_\alpha \sigma(R) \int_{B_R^+} \big( |\nabla u - \nabla v| g(|\nabla u - \nabla v|) + (1 + |\nabla u|) g(1 + |\nabla u|) \big) \,dx\\
   &\leq C \sigma(R) \int_{B_R^+} G(|\nabla u|) \, dx + C \sigma(R) R^n.
\end{align*}

Combining the estimates for \(II\) and \(III\), we obtain the upper bound for \(I\):
\begin{align*}
I &\le |II| + |III| \\
&\le C e^{c\|\gamma\|_{L^s(B_1^+)}} \left( \Lambda_B \omega(R) + R^{1-\frac{n}{q}} + R^{\frac{\mu_{\star}\delta}{1+\delta}} + \sigma(R) \right) \int_{B_{R}^{+}} G(|\nabla u|)\,dx \\
&\quad + C e^{c\|\gamma\|_{L^s(B_1^+)}} \left( R^{n+1-\frac{n}{q}} + R^{n+\frac{\mu_{\star}\delta}{1+\delta}} + \sigma(R)R^n \right).
\end{align*}
By the results of Section \ref{section3}, the bounded weak solution \(u\) is locally H\"older continuous, implying its oscillation decays as \(\omega(R) \le C^\star R^\alpha\) for some \(\alpha \in (0,1)\), where the dependencies of $C^\star$ and $\alpha$ on $\Lambda_B$ and $\|\gamma\|_{L^s(B_1^+)}$ are explicit in Remark \ref{remark alpha dependence}. Recalling \(\sigma(R) = R^{\alpha_{\mathcal{A}}} + \omega(R)^{\alpha_{\mathcal{A}}}\) and our notation \(\mathcal{E} = C e^{c(\Lambda_B + \|\gamma\|_{L^s(B_1^+)})}\), we have \(\sigma(R) \le \mathcal{E} R^{\alpha \alpha_{\mathcal{A}}}\). 

To unify the decay rates, we define the geometric exponent 
$$\nu := \min\left\{\alpha \alpha_{\mathcal{A}}, \, 1 - \frac{n}{q}, \, \frac{\mu_{\star}\delta}{1+\delta}\right\} \in (0,1).$$
The coefficients in the integral bound then satisfy:
$$ C e^{c\|\gamma\|_{L^s(B_1^+)}} \left( \Lambda_B \omega(R) + R^{1-\frac{n}{q}} + R^{\frac{\mu_{\star}\delta}{1+\delta}} + \sigma(R) \right) \leq \mathcal{E} R^{\nu}. $$
Similarly, the remaining volume terms are majorized by \(\mathcal{E} R^{n+\nu}\). Thus, absorbing the lower-order volume terms into the dominant geometric scaling, we arrive at the explicit upper bound:
\begin{equation} \label{eq:I_upper_bound}
I \le \mathcal{E} R^\nu \int_{B_{R}^{+}} G(|\nabla u|)\,dx + \mathcal{E} R^{n + \nu}.
\end{equation}

As in \cite{lieberman1991natural}, by \eqref{homogeneous structural conditions} and the Fundamental Theorem of Calculus, 
\begin{equation}\label{I_lower_bounds}
 I \geq \int_{B_R^+} \int_0^1 H\big(|\nabla u+(1-t)(\nabla u-\nabla v)|\big)|\nabla u-\nabla v|^2 \, dt \, dx.
\end{equation}

Let $S_{1}=\{|\nabla u-\nabla v|\le 2|\nabla u|\}$ and $S_{2}=\{|\nabla u-\nabla v|>2|\nabla u|\}$. Then $B_{R}^{+}=S_{1}\cup S_{2}$, and we have:
\begin{enumerate}
    \item[(1)] $\frac{1}{2}|\nabla u| \le |\nabla u+(1-t)(\nabla u-\nabla v)| \le 3|\nabla u|$ in $S_{1}$ for $t\ge 3/4$,
    \item[(2)] $\frac{1}{4}|\nabla u-\nabla v| \le |\nabla u+(1-t)(\nabla u-\nabla v)| \le 3|\nabla u-\nabla v|$ in $S_{2}$ for $t\le 1/4$.
\end{enumerate}
In particular, in $S_{1}$ for $t\ge 3/4$ we have
$$ H(|\nabla u+(1-t)(\nabla u-\nabla v)|) \ge \frac{g(\frac{1}{2}|\nabla u|)}{3|\nabla u|} \ge \frac{1}{2^{g_{0}}3}H(|\nabla u|), $$
and in $S_{2}$ for $t\le 1/4$ we have
$$ H(|\nabla u+(1-t)(\nabla u-\nabla v)|)|\nabla u-\nabla v|^{2} \ge \frac{G(|\nabla u-\nabla v|)}{4^{g_{0}}3}. $$
Combining these inequalities yields
\begin{equation} \label{eq5.29}
I \ge C \left[ \int_{S_{1}} H(|\nabla u|)|\nabla u-\nabla v|^{2} \, dx + \int_{S_{2}} G(|\nabla u-\nabla v|) \, dx \right].
\end{equation}
Now, by \eqref{eq:I_upper_bound}:
$$ \int_{S_{2}} G(|\nabla u-\nabla v|) \, dx \le \mathcal{E} R^\nu \int_{B_{R}^{+}} G(|\nabla u|)\,dx + \mathcal{E} R^{n + \nu}. $$
In $S_{1}$ we have $H(|\nabla u-\nabla v|)|\nabla u-\nabla v| \le CH(|\nabla u|)|\nabla u|$. By H\"older's and Young's inequalities we conclude that
\begin{align*}
\int_{S_{1}} G(|\nabla u-\nabla v|) \, dx &\le \int_{S_{1}} [H(|\nabla u-\nabla v|)|\nabla u-\nabla v|]^{1/2} H(|\nabla u-\nabla v|)^{1/2} |\nabla u-\nabla v|^{3/2} \, dx \\
&\le C \int_{S_{1}} [H(|\nabla u|)|\nabla u|]^{1/2} H(|\nabla u-\nabla v|)^{1/2} |\nabla u-\nabla v|^{3/2} \, dx \\
&\le C \left( \int_{S_{1}} H(|\nabla u|)|\nabla u-\nabla v|^{2} \, dx \right)^{1/2} \left( \int_{S_{1}} g(|\nabla u|)|\nabla u| \, dx \right)^{1/2} \\
&\le C \cdot I^{1/2} \left( \int_{B_R^+} G(|\nabla u|) \, dx \right)^{1/2}.
\end{align*}
Again, by \eqref{eq:I_upper_bound} and the sub-additivity of the square root:
\begin{align*}
\int_{S_{1}} G(|\nabla u-\nabla v|) \, dx &\le C \left[ \mathcal{E} R^\nu \int_{B_{R}^{+}} G(|\nabla u|)\,dx + \mathcal{E} R^{n+\nu} \right]^{1/2} \left[ \int_{B_{R}^{+}} G(|\nabla u|)\,dx \right]^{1/2} \\
&\le C \mathcal{E}^{1/2} R^{\nu/2} \int_{B_{R}^{+}} G(|\nabla u|)\,dx + C \mathcal{E}^{1/2} R^{\frac{n+\nu}{2}} \left( \int_{B_R^+} G(|\nabla u|)\,dx \right)^{1/2}.
\end{align*}
Applying Young's inequality $ab \le \tau a^2 + \frac{1}{4\tau} b^2$ to the second term with $\tau = R^{\nu/2}$, we get:
\begin{align*}
     C \mathcal{E}^{1/2} R^{\frac{n+\nu}{2}} \left( \int_{B_R^+} G(|\nabla u|)\,dx \right)^{1/2} &\le R^{\nu/2} \int_{B_R^+} G(|\nabla u|)\,dx + \mathcal{E} R^{n + \nu/2}.
\end{align*}
Defining the final geometric decay exponent $\theta_\star := \frac{\nu}{2n} \in (0,1)$, we have $n\theta_\star = \nu/2$. Thus, all coefficients multiplying the energy of $u$ in $S_1$ and $S_2$ are bounded by \(\mathcal{E} R^{n\theta_\star}\), yielding:
\begin{equation} \label{eq5.30}
\int_{B_{R}^{+}} G(|\nabla u-\nabla v|) \, dx \le \mathcal{E} \left[ R^{n\theta_\star} \int_{B_{R}^{+}} G(|\nabla u|) \, dx + R^{n + n\theta_\star} \right].
\end{equation}
We know, by Lemma \ref{lemma:comparison_boundary}, that for any $0<r<R$,
$$ \int_{B_{r}^{+}} G(|\nabla v|) \, dx \le C \left[ \left(\frac{r}{R}\right)^{n} \int_{B_{R}^{+}} G(|\nabla v|) \, dx + r^{n} G(\|\psi\|_{C^{1,\alpha_\psi}(B_{1}^{\prime})}) \right]. $$
Thus, since $G(|\nabla u|) \le C(g_0)(G(|\nabla v|)+G(|\nabla u-\nabla v|))$ and $\|\psi\|_{C^{1,\alpha_\psi}(B_{1}^{\prime})} \le 1$, we obtain
\begin{align*}
\int_{B_{r}^{+}} G(|\nabla u|) \, dx &\le C \left[ \int_{B_{r}^{+}} G(|\nabla v|) \, dx + \int_{B_{r}^{+}} G(|\nabla u-\nabla v|) \, dx \right] \\
&\le C \left[ \left(\frac{r}{R}\right)^{n} \int_{B_{R}^{+}} G(|\nabla v|) \, dx + G(1)R^{n} + \mathcal{E} R^{n\theta_\star} \int_{B_{R}^{+}} G(|\nabla u|) \, dx + \mathcal{E} R^{n+n\theta_\star} \right] \\
&\le \mathcal{E} \left[ \left\{ \left(\frac{r}{R}\right)^{n} + R^{n\theta_\star} \right\} \int_{B_{R}^{+}} G(|\nabla u|) \, dx + R^{n} \right].
\end{align*}

By using Lemma \ref{leitao lemma}, Remark \ref{remark lemma leitao}, and Lemma \ref{lem:caccioppoli_boundary}, for all $\varsigma\in(0,n)$ we obtain
$$ \int_{B_{r}^{+}} G(|\nabla u|) \, dx \le \mathcal{E} r^{\varsigma}, \quad \forall r\in(0,3/4). $$
Taking $\varsigma=h-1+n$ for any $h\in(0,1)$, we get
\begin{equation} \label{eq5.31}
\int_{B_{r}^{+}} G(|\nabla u|) \, dx \le \mathcal{E} r^{h-1+n}, \quad \forall r\in(0,3/4).
\end{equation}
At this point, following the same lines as the proof of Theorem 1.1 in \cite{BragaSousa2025}, we conclude that
$$ [\nabla u]_{C^{0,\beta}(B_{1/2}^{+})} \le C_\star e^{c_\star\left(\Lambda_B + \|\gamma\|_{L^s(B_1^+)}\right)}. $$
\end{proof}

\begin{remark}[{\bf On the explicit bounds of $\beta$}]\label{remark beta dependence}
Following the Campanato-type comparison argument in \cite[Theorem 1.1]{BragaSousa2025}, the gradient H\"older exponent $\beta \in (0,1)$ in Lemma \ref{small norms} is explicitly determined by balancing the homogeneous boundary gradient decay exponent $\sigma_\star = \sigma_\star(n, \delta, g_0, \lambda, \Lambda, \alpha_\psi) \in (0,1)$ of $\nabla v$ with the geometric energy rate $n\theta_\star = \frac{\nu}{2}$, yielding
\begin{equation*}
    \beta = \frac{\sigma_\star \nu}{2(1+g_0)(2n + \nu + 2\sigma_\star)},
\end{equation*}
where $\nu = \min\left\{\alpha_\star \alpha_{\mathcal{A}}, \, 1 - \frac{n}{q}, \, \frac{\mu_\star\delta}{1+\delta}\right\}$ and $\mu_\star = 1+\delta - \frac{n(1+g_0)}{s} > 0$. Since $0 < \nu, \sigma_\star < 1$ implies $2n < 2n + \nu + 2\sigma_\star < 2n + 3$, combining this identity with the bounds for $\alpha_\star$ from Remark \ref{remark alpha dependence} gives the explicit two-sided control:
\begin{equation*}
    {\sigma_\star}c_1 \min\left\{ \frac{\alpha_{\mathcal{A}}}{\ln 2 \cdot \mathcal{E}}, \, 1 - \frac{n}{q}, \, \frac{\mu_\star\delta}{1+\delta} \right\} \le \beta \le {\sigma_\star}c_2 \min\left\{ \frac{\alpha_{\mathcal{A}}}{\ln 2 (\mathcal{E} - 1)}, \, 1 - \frac{n}{q}, \, \frac{\mu_\star\delta}{1+\delta} \right\},
\end{equation*}
where $\mathcal{E} = C e^{c(\Lambda_B + \|\gamma\|_{L^s(B_1^+)})} \ge 2$ and $c_1,c_2$ are positive constants depending only $n$ and $g_0$.
\end{remark}

\noindent Finally 

\begin{proof}[Proof of Theorem \ref{thm:up_to_boundary}]
Set
$$ \kappa := \|u\|_{L^{\infty}(B_{1}^{+})} + \|\psi\|_{C^{1,\alpha}(B_{1}^{\prime})} + g^{-1}(\|f\|_{L^{q}(B_{1}^{+})}). $$
If $\kappa$ is null, the proof is trivial. Suppose $\kappa>0$ and define the scaled solution $u_{\kappa}^{*}(x) = \frac{u(x)}{\kappa}$ and boundary data $\psi_{\kappa}^{*}(x) = \frac{\psi(x)}{\kappa}$. In order to normalize the operators, we introduce the rescaled function $g_\kappa(t) := \frac{g(\kappa t)}{g(\kappa)}$, which determines a new $N$-function $G_\kappa(t) = \int_0^t g_\kappa(s) \, ds = \frac{G(\kappa t)}{\kappa g(\kappa)}$. We define the scaled operators and source term as follows:
$$ \mathcal{A}_{\kappa}(x, z, p) := \frac{\mathcal{A}(x, \kappa z, \kappa p)}{g(\kappa)}, \quad \mathcal{B}_{\kappa}(x, z, p) := \frac{\mathcal{B}(x, \kappa z, \kappa p)}{g(\kappa)}, \quad f_{\kappa}^{*}(x) := \frac{f(x)}{g(\kappa)}. $$

For this scaled equation, we have
$$ \operatorname{div} \mathcal{A}_{\kappa}(x, u_{\kappa}^{*}, \nabla u_{\kappa}^{*}) + \mathcal{B}_{\kappa}(x, u_{\kappa}^{*}, \nabla u_{\kappa}^{*}) = 0 \quad \text{in } B_{1}^{+}, $$
$$ u_{\kappa}^{*} = \psi_{\kappa}^{*} \quad \text{on } B_{1}^{\prime}. $$

By the original structural properties given in \eqref{structural conditions}, these normalized operators satisfy the following scaled structural conditions for all $p \in \mathbb{R}^n \setminus \{0\}$:
\begin{equation*} 
\begin{cases}
    |\mathcal{A}_{\kappa}(x, z, p)| + |D_p \mathcal{A}_{\kappa}(x, z, p)| |p| \le \Lambda g_\kappa(|p|), \\
    \lambda \frac{g_\kappa(|p|)}{|p|} |\xi|^2 \le \langle D_p \mathcal{A}_{\kappa}(x, z, p)\xi, \xi \rangle, \\
    |\mathcal{A}_{\kappa}(x, z, p) - \mathcal{A}_{\kappa}(y, w, p)| \le \Lambda_\alpha g_\kappa(|p|) \left[ |x-y|^{\alpha_{\mathcal{A}}} + \kappa^{\alpha_{\mathcal{A}}} |z-w|^{\alpha_{\mathcal{A}}} \right], \\
    |\mathcal{B}_{\kappa}(x, z, p)| \le (\Lambda_B \kappa) G_\kappa(|p|) + |\gamma(x)| g_\kappa(|p|) + |f_{\kappa}^{*}(x)|.
\end{cases}
\end{equation*}

By the definitions of the scaled terms and the choice of the normalization constant $\kappa$, it follows that
$$ \max\left\{\|f_{\kappa}^{*}\|_{L^{q}(B_{1}^{+})}, \, \|u_{\kappa}^{*}\|_{L^{\infty}(B_{1}^{+})}, \, \|\psi_{\kappa}^{*}\|_{C^{1,\alpha}(T_1)}\right\} \le 1. $$
  
Applying Lemma \ref{small norms} to the normalized solution $u_\kappa^*$ with the scaled convective constant ${\Lambda}_{B_\kappa} := \Lambda_B \kappa$ and the scale-invariant drift coefficient $\gamma(x)$, this uniform bound ensures that
$$ [\nabla u_{\kappa}^{*}]_{C^{0,\beta}(B_{1/2}^{+})} \le C_\star e^{c_\star\left({\Lambda}_{B_\kappa} + \|\gamma\|_{L^s(B_1^+)}\right)} = C_\star e^{c_\star\left(\Lambda_B\kappa + \|\gamma\|_{L^s(B_1^+)}\right)}. $$
Hence, scaling back to the original function $u$:
\begin{align*}
[\nabla u]_{C^{0,\beta}(B_{1/2}^{+})} &= \kappa [\nabla u_{\kappa}^{*}]_{C^{0,\beta}(B_{1/2}^{+})} \\
&\le \kappa C_\star e^{c_\star\left(\Lambda_B\kappa + \|\gamma\|_{L^s(B_1^+)}\right)} \\
&= C_\star e^{c_\star\left(\Lambda_B\kappa + \|\gamma\|_{L^s(B_1^+)}\right)} \left( \|u\|_{L^{\infty}(B_{1}^{+})} + \|\psi\|_{C^{1,\alpha}(T_{1})} + g^{-1}(\|f\|_{L^{q}(B_{1}^{+})}) \right).
\end{align*}
\end{proof}

At this point the proof of Theorem \ref{thm:global_C1alpha} is standard.

\section{Regularity at Boundary Critical Points}\label{section7}

In this section, we extend the geometric iteration scheme to critical  points located on the boundary. We establish the uniform H\"older decay of the gradient around these critical  boundary points, fundamentally relying on the assumption that the boundary data itself provides sufficient smoothness.

\begin{equation} \label{1''}
\begin{cases}
\operatorname{div} \mathcal{A}(x, u, \nabla u) +\mathcal{B}(x,u,\nabla u) = 0 & \text{in } B_{1}^{+},\\
u=\psi & \text{on } T_{1}.
\end{cases}
\end{equation}
satisfies \eqref{structural conditions} with $\psi\in C^{1,\alpha}(T_1)$

\begin{lemma}\label{singular_boundary_nonhomo} 
Let $u\in W^{1,G}(B_1^+)\cap L^{\infty}(B_1^+)$ be a weak solution to \eqref{1''} with $\|u\|_{L^\infty(B_1^+)} \le 1$. Then, given $\eta>0$, there exists a constant $\varepsilon_\star>0$, depending only on $\eta$, $n$, $s$, $g(1)$, $\delta$, $g_0$, $\lambda$, $\Lambda$ and $\alpha_{\mathcal{A}}$, such that if 
\begin{eqnarray}
	|\nabla u(0)|+\Lambda_\alpha+\Lambda_B+\|\gamma\|_{L^s(B_1^+)}+
	\|f\|_{L^q(B_1^+)} + \|\psi\|_{C^{1,\alpha}(T_1)}<\varepsilon_\star\hbox{,}
\end{eqnarray}
then there exists a function $h\in W^{1,G}(B_{{3/4}}^+)\cap L^{\infty}(B_{3/4}^+)$, satisfying
\begin{eqnarray}
	\operatorname{div} \mathcal{A}_0(\nabla h)=0\ \ \text{in}\ \ B_{{3/4}}^+,\ \ h = 0 \ \ \text{on}\ \ T_{3/4},\ \ \text{and} \ \ \nabla h(0)=0\hbox{,}\label{4_bdry}
\end{eqnarray} 
such that 
\begin{equation}\label{singularthesis_bdry}
	 \|u-h\|_{L^G(B^+_{3/4})} < \eta.
\end{equation}
\end{lemma}

\begin{proof}
Suppose by contradiction that there exist an $\eta_0>0$ and a sequence $u_k\in W^{1,G}(B_1^+)\cap L^{\infty}(B_1^+)$ of weak solutions to the scaled equations
$$\operatorname{div} \mathcal{A}_{k}(x, u_k, \nabla u_k) + \mathcal{B}_{k}(x, u_k, \nabla u_k) = f_k\ \ \ \ \hbox{ in }\ \ \ \ B_1^+\hbox{,}$$
with boundary data $u_k=\psi_k$ on $T_1$, and uniform bound $\|u_k\|_{L^\infty(B_1^+)} \le M$. We assume that:
\begin{eqnarray}
	\max \left\{|\nabla u_k(0)|,\, \Lambda_{\mathcal{B}_{k}},\,\Lambda_{\alpha_k},\,\|\gamma_k\|_{L^s(B_1^+)},\,\|f_k\|_{L^q(B_1^+)}, \|\psi_k\|_{C^{1,\alpha}(T_1)}\right\}\to0 \hbox{,}\label{5_bdry}
\end{eqnarray}
as $k\to\infty$. Assume that the sequence of operators $\mathcal{A}_{k}(x, z, p)$ satisfies the structural conditions \eqref{structural conditions} with respect to the spatial and $z$-variables, with the oscillation bounded by:
\begin{equation} \label{oscillation_k_bdry}
    |\mathcal{A}_{k}(x, z, p)-\mathcal{A}_{k}(0,0,p)|\leq \Lambda_{\alpha_k} (1+g(|p|)) (|x|^{\alpha_{\mathcal{A}}}+|z|^{\alpha_{\mathcal{A}}}).
\end{equation}
Assume also that $\|\mathcal{B}_{k}(x, z, p)\| \le \Lambda_{\mathcal{B}_{k}} G(\|p\|) + \|\gamma_k(x)\|g(\|p\|) + \|f_k(x)\|$. However, 
\begin{equation}\label{neqthesis_bdry}
	\|u_k-h\|_{L^G(B^+_{3/4})} \geq \eta_0
\end{equation}
for any $h$ satisfying \eqref{4_bdry}.

By the structural conditions, $p \mapsto \mathcal{A}_{k}(0,0,p)$ is uniformly bounded on compact sets of $\mathbb{R}^n$ since $$\|\mathcal{A}_{k}(0,0,p)\| \le \Lambda g(\|p\|).$$
Furthermore, we claim that $\{\mathcal{A}_{k}(0,0,\cdot)\}$ is equicontinuous. Indeed, away from the origin, the bound $$\|D_p \mathcal{A}_{k}(0,0,p)\| \le \Lambda \frac{g(\|p\|)}{\|p\|}$$ ensures the operators are uniformly Lipschitz. 
On the other hand, near the origin, equicontinuity is enforced directly by the amplitude bound $$\|\mathcal{A}_{k}(0,0,p)\| \le \Lambda g(\|p\|)$$ combined with $g(0)=0$. Thus, the Arzel\`a--Ascoli theorem guarantees the existence of a continuous limit operator $\mathcal{A}_0(p)$ such that, up to a subsequence, $\mathcal{A}_{k}(0,0,\cdot)$ converges locally uniformly in $\mathbb{R}^n$ to $\mathcal{A}_0(\cdot)$.

By the boundary Caccioppoli-type inequality (Lemma \ref{lem:caccioppoli_boundary}) and the uniform bound $\|u_k\|_{L^\infty(B_1^+)} \le 1$,
\begin{align}\label{eq:caccioppoli_uniform_k_bdry}
	\int_{B_{3/4}^+}G(|\nabla u_k|)\,dx &\leq C e^{c\left(\frac{\Lambda_{\mathcal{B}_{k}}}{\lambda} + \|\gamma_k\|_{L^s(B_1^+)}\right)\left(M + \|\psi_k\|_{C^{1,\alpha}(T_1)}\right)} \nonumber\\
	&\quad \times \left[ \int_{B_1^+} G\big(4|u_k - \tilde{\psi}_k|\big)\,dx + \int_{B_1^+} G(|\nabla \tilde{\psi}_k|)\,dx + \tilde{G}\big(\|f_k\|_{L^q(B_1^+)}\big) \right] \leq C^\star
\end{align}
for all $k$ large, where $C^\star$ is independent of $k$. Since $W^{1,G}(B_1^+)$ is a reflexive Banach space and the energy is uniformly bounded, there exists a function $u_0\in W^{1,G}(B_{3/4}^+)$ such that, up to a subsequence,
\begin{eqnarray}
	u_k\rightharpoonup u_0 &\hbox{in}& W^{1,G}(B_{{3/4}}^+)\nonumber\\
	\label{strong convergence_bdry} u_k\to u_0 &\hbox{in}& L^{G}(B_{{3/4}}^+)\nonumber\\
    \nabla u_k(x)\to \nabla u_0 (x) &\hbox{for}& x\in B_{{3/4}}^+\ \  \hbox{a.e.} .\nonumber
\end{eqnarray}

Given any test function $\psi \in C^\infty_0(B_{3/4}^+)$, our goal is to show that 
$$ \int_{B_{3/4}^+} \mathcal{A}_0(\nabla u_0) \cdot \nabla \psi \,dx = 0. $$
By the structural condition, the vector fields $V_k(x) := \mathcal{A}_0(\nabla u_k(x))$ satisfy $|V_k| \le \Lambda g(|\nabla u_k|)$. By (G-4) of Lemma \ref{firstLema} and the Caccioppoli inequality \eqref{eq:caccioppoli_uniform_k_bdry},
$$ \int_{B_{3/4}^+} \tilde{G}(|V_k|) \,dx \le \int_{B_{3/4}^+} \tilde{G}(\Lambda g(|\nabla u_k|)) \,dx \le C(g_0,\delta, \Lambda)\int_{B_{3/4}^+} G(|\nabla u_k|) \,dx \le C(g_0,\delta, \Lambda)C^\star. $$

As $L^{\tilde{G}}(B_{3/4}^+)$ is reflexive, the Banach--Alaoglu theorem guarantees a weakly convergent subsequence $V_k \rightharpoonup V_0$ in $L^{\tilde{G}}(B_{3/4}^+)$. Since $\nabla u_k \to \nabla u_0$ a.e. and $\mathcal{A}_0$ is continuous, $V_k \to \mathcal{A}_0(\nabla u_0)$ pointwise a.e. Thanks to Mazur's Lemma, $V_0 = \mathcal{A}_0(\nabla u_0)$ a.e. Since $\nabla \psi \in L^\infty(B_{3/4}^+) \subset L^G(B_{3/4}^+)$, we have
$$ \int_{B_{3/4}^+} \mathcal{A}_0(\nabla u_0) \cdot \nabla \psi \,dx = \lim_{k \to \infty} \int_{B_{3/4}^+} \mathcal{A}_0(\nabla u_k) \cdot \nabla \psi \,dx. $$

Now, we evaluate the limit on the right-hand side. Utilizing the weak formulation of $u_k$, we decompose the integral algebraically for each fixed $k$:
\begin{align*}
	\int_{B_{3/4}^+} \mathcal{A}_0(\nabla u_k) \cdot \nabla \psi \,dx &= \int_{B_{3/4}^+} \left( \mathcal{A}_0(\nabla u_k) - \mathcal{A}_{k}(0,0,\nabla u_k) \right) \cdot \nabla \psi \,dx \\
	&\quad + \int_{B_{3/4}^+} \left( \mathcal{A}_{k}(0,0,\nabla u_k) - \mathcal{A}_{k}(x, u_k, \nabla u_k) \right) \cdot \nabla \psi \,dx \\
	&\quad + \int_{B_{3/4}^+} - \mathcal{B}_{k}(x, u_k, \nabla u_k) \psi \,dx \\
    &= I+II+III
\end{align*}

We estimate each residual term as $k \to \infty$:

For \textit{I}, since the Arzel\`a--Ascoli theorem only guarantees {locally} uniform convergence, we decompose the domain for any fixed $R > 0$ into the compact region $E_R = \{x \in B_{3/4}^+ : \|\nabla u_k(x)\| \le R\}$ and the tail region $E_R^c = \{x \in B_{3/4}^+ : \|\nabla u_k(x)\| > R\}$:
\begin{eqnarray*}
    |I| &\le& \|\nabla \psi\|_{L^\infty} \left( \int_{E_R} |\mathcal{A}_0(\nabla u_k) - \mathcal{A}_{k}(0,0,\nabla u_k)| \,dx + \int_{E_R^c} \left( |\mathcal{A}_0(\nabla u_k)| + |\mathcal{A}_{k}(0,0,\nabla u_k)| \right) dx \right)\\
    &\le&\|\nabla \psi\|_{L^\infty} \left(|B_{3/4}^+|\sup_{|p| \le R} |\mathcal{A}_0(p) - \mathcal{A}_{k}(0,0,p)|+  2\Lambda\int_{\{x\in B_{3/4}^+ : |\nabla u_k|>R\}} g(|\nabla u_k|)\,dx \right).
\end{eqnarray*}

Again, since $\lim_{t \to \infty} \tilde{G}(t)/t = \infty$, given $\sigma > 0$, we define $M_\sigma = 2\Lambda \tilde{C}^\star / \sigma$. By the limit definition, there exists a threshold $t_\sigma > 0$ such that for all $t > t_\sigma$, $$\tilde{G}(t)/t \ge M_\sigma \iff t \le \frac{1}{M_\sigma} \tilde{G}(t).$$ 

Choosing $R_\sigma > 0$ large enough such that $g(R_\sigma) \ge t_\sigma$, we have
$$2\Lambda \int_{\{x\in B_{3/4}^+ : \|\nabla u_k\| > R_\sigma\}} g(\|\nabla u_k\|) \,dx \le \frac{2\Lambda}{M_\sigma} \int_{\{x\in B_{3/4}^+ : \|\nabla u_k\| > R_\sigma\}} \tilde{G}(g(\|\nabla u_k\|)) \,dx \le \frac{2\Lambda}{M_\sigma} \tilde{C}^\star = \sigma.$$

Fixing $R = R_\sigma$, we get
$$\limsup_{k \to \infty} |{I}| \le \|\nabla \psi\|_{L^\infty} \left( \lim_{k \to \infty} \varepsilon_k(R_\sigma) \|B_{3/4}^+\| + \sigma \right) = \sigma \|\nabla \psi\|_{L^\infty}.$$
Since $\sigma > 0$ is arbitrary, we conclude ${I} \to 0$.

For \textit{{II}}, utilizing the structural oscillation condition \eqref{oscillation_k_bdry}:
$$|II| \le \int_{B_{3/4}^+} \Lambda_{\alpha_k} \big(1 + g(\|\nabla u_k\|)\big) \left( \|x\|^{\alpha_{\mathcal{A}}} + \|u_k(x)\|^{\alpha_{\mathcal{A}}} \right) \|\nabla \psi\|\,dx.$$

Again, using $\lim_{t\to\infty}{\tilde{G}(t)}/{t}=\infty$, there exists a threshold $t_0 > 0$ such that $t \le \tilde{G}(t)$ for all $t \ge t_0$. Splitting the domain of integration, we obtain:
$$\int_{B_{3/4}^+} g(\|\nabla u_k\|) \,dx \le \int_{\{g < t_0\}} t_0 \,dx + \int_{\{g \ge t_0\}} \tilde{G}(g(\|\nabla u_k\|)) \,dx \le t_0 \|B_{3/4}^+\| + C(g_0,\delta, \Lambda)C^\star.$$
Therefore,
$$|\textit{II}| \le \Lambda_{\alpha_k} \|\nabla \psi\|_{L^\infty} \left( 1 + M^{\alpha_{\mathcal{A}}} \right) \int_{B_{3/4}^+} \big(1 + g(\|\nabla u_k\|)\big) \,dx \leq C\Lambda_{\alpha_k} \to 0.$$

For \textit{III}, utilizing the structural bound $\|\mathcal{B}_{k}(x, z, p)\| \le \Lambda_{\mathcal{B}_{k}} G(\|p\|) + \|\gamma_k(x)\|g(\|p\|) + \|f_k(x)\|$, we split the integral into three parts:
\begin{align*}\label{eq:III_split_singular_bdry}
    |III| &\le \|\psi\|_{L^\infty} \left( \int_{B_{3/4}^+} |f_k(x)|\,dx + \Lambda_{\mathcal{B}_{k}} \int_{B_{3/4}^+} G(|\nabla u_k|)\,dx + \int_{B_{3/4}^+} |\gamma_k(x)|g(|\nabla u_k|)\,dx \right).
\end{align*}
Notice that thanks to H\"older's inequality with exponents $q$ and $q' = \frac{q}{q-1}$, we have
    \begin{equation*}
      \int_{B_{3/4}^+} |f_k(x)|\,dx \le |B_{3/4}^+|^{1-\frac{1}{q}} \|f_k\|_{L^q(B_1^+)} \to 0 \quad \text{as } k \to \infty,
    \end{equation*}
and
    \begin{equation*}
      \Lambda_{\mathcal{B}_{k}} \int_{B_{3/4}^+} G(|\nabla u_k|)\,dx \le \Lambda_{\mathcal{B}_{k}} C^\star \to 0 \quad \text{as } k \to \infty.
    \end{equation*}

  Given $\varepsilon > 0$, applying Young's inequality (Lemma \ref{youngorlicz}) with $a = \|\gamma_k(x)\|$ and $b = \|\nabla u_k\|$, and (g-2), we get almost everywhere:
    \begin{align*}
        |\gamma_k(x)| g(|\nabla u_k|) &\le \frac{\varepsilon}{(1+g_0)C^\star} |\nabla u_k| g(|\nabla u_k|) + C(\delta,C^\star) \varepsilon^{-\frac{1}{\delta}} |\gamma_k(x)| g(|\gamma_k(x)|) \\
        &\le \frac{\varepsilon}{C^\star} G(|\nabla u_k|) + C(\delta, g_0,C^\star) \varepsilon^{-\frac{1}{\delta}} G(|\gamma_k(x)|).
    \end{align*}
    Integrating over $B_{3/4}^+$ and using the uniform energy bound \eqref{eq:caccioppoli_uniform_k_bdry}, we find
    \begin{equation}\label{eq:III_3_step1_bdry}
       \int_{B_{3/4}^+} |\gamma_k(x)|g(|\nabla u_k|)\,dx \le \varepsilon  + C(\delta, g_0,C^\star) \varepsilon^{-\frac{1}{\delta}} \int_{B_{3/4}^+} G(|\gamma_k(x)|)\,dx.
    \end{equation}
    To control the last integral, we use the growth condition (G-2). Since $s >  1+g_0$, we can apply H\"older's inequality on $B_{3/4}^+$ with exponents $\frac{s}{1+\delta} > 1$ and $\frac{s}{1+g_0} > 1$ respectively:
    \begin{align*}
        \int_{B_{3/4}^+} G(|\gamma_k(x)|)\,dx &\le G(1) \int_{B_{3/4}^+} |\gamma_k(x)|^{1+\delta}\,dx + G(1) \int_{B_{3/4}^+} |\gamma_k(x)|^{1+g_0}\,dx \\
        &\le G(1) |B_{3/4}^+|^{1-\frac{1+\delta}{s}} \|\gamma_k\|_{L^s(B_1^+)}^{1+\delta} + G(1) |B_{3/4}^+|^{1-\frac{1+g_0}{s}} \|\gamma_k\|_{L^s(B_1^+)}^{1+g_0}.
    \end{align*}
    Substituting this bound back into \eqref{eq:III_3_step1_bdry}, we obtain for every $k \ge 1$:
    \begin{equation*}
          \int_{B_{3/4}^+} |\gamma_k(x)|g(|\nabla u_k|)\,dx\le \varepsilon + C(n, s, \delta, g_0, g(1), C^\star ) \varepsilon^{-\frac{1}{\delta}} \left( \|\gamma_k\|_{L^s(B_1^+)}^{1+\delta} + \|\gamma_k\|_{L^s(B_1^+)}^{1+g_0} \right).
    \end{equation*}

Therefore,
    \begin{equation*}
        \limsup_{k \to \infty} |III| \le \|\psi\|_{L^\infty} \varepsilon.
    \end{equation*}
    Since $\varepsilon > 0$ is arbitrary, letting $\varepsilon \to 0^+$ implies $\lim_{k \to \infty} |III| = 0$.

Thus, $u_0$ is a weak solution of $\operatorname{div} \mathcal{A}_0(\nabla u_0) = 0$ in $B_{3/4}^+$. Furthermore, since $u_k = \psi_k$ on $T_1$ and we assumed $\|\psi_k\|_{C^{1,\alpha}(T_1)} \to 0$, it follows that $\psi_k \to 0$ uniformly on $T_1$. By the continuity of the trace operator $\mathcal{T}:W^{1,G}(B^+_{3/4})\to L^{G}(\partial B^+_{3/4})$ \cite[Chapter 7]{kufner1977function}, we get $u_0 = 0$ on $T_{3/4}$. 

Thus, $u_0$ is a weak solution of 
\begin{equation*}
    \begin{cases}
        \operatorname{div} \mathcal{A}_0(\nabla u_0) = 0 \ \ &\text{in} \ \ B^+_{3/4},\\
        u_0=0 \ \ &\text{on} \ \ T_{3/4}.
    \end{cases}
\end{equation*}

The theory in Section \ref{section3} guarantees that $u_0 \in C^{1,\alpha}(B_{3/4}^+ \cup T_{3/4})$ for some $\alpha\in(0,1)$ depending only on $n, \delta, g_0, \lambda$, and $\Lambda$. By the $C^1$-convergence at the origin, assumption \eqref{5_bdry} ensures that $\nabla u_0(0) = 0$.

Therefore, setting $h := u_0$, the strong convergence $u_k \to u_0$ in $L^G(B_{3/4}^+)$ implies that
$$\|u_k-h\|_{L^G(B^+_{3/4})}<\frac{\eta_0}{2}$$
for $k$ large enough, which strictly contradicts \eqref{neqthesis_bdry}, completing the proof.
\end{proof}

\begin{lemma} \label{Lemma2_boundary}
Let $u \in W^{1,G}(B_1^+) \cap L^\infty(B_1^+)$ be a weak solution to \eqref{1''}. Assume $u = \psi$ on $T_1$, with $\nabla \psi(0) = 0$. Assume $\|u\|_{L^\infty(B_1^+)} \le 1$. Then, given $0<\alpha<\alpha_H$, there exists a radius $0<{r}_\star<{1/2}$ depending only on $n$, $\delta$, $\lambda$, $\Lambda$, $g_0$, $\alpha_{\mathcal{A}}$, and $\alpha$, and a constant $0<\varepsilon_\star<1$ depending only on $n$, $\delta$, $\lambda$, $\Lambda$, $g_0$, $g(1)$, $\alpha_{\mathcal{A}}$ and $\alpha$, such that if 
\begin{equation*}
	|\nabla u(0)| + \Lambda_{\alpha_{\mathcal{A}}} + \Lambda_B +\|\gamma\|_{L^s(B_1^+)}+ \|f\|_{L^q(B_1^+)} + \|\psi\|_{C^{1,\alpha}(T_1)} < \varepsilon_\star,
\end{equation*}
then we have 
\begin{equation}\label{key_boundary}
	\left(\fint_{B_{r}^+}|u(y) - \psi(0)|^{1+\delta}\,dy\right)^\frac{1}{1+\delta} < {r}^{(1+\alpha)}, \ \ \ \ \mbox{for all} \ \ \ \ 0<{r}<{r}_\star\hbox{.}
\end{equation}
\end{lemma}

\begin{proof}
Let $w(y) := u(y) - \psi(0)$. Since $\nabla w = \nabla u$, the shifted function $w$ is a weak solution to the transformed quasilinear boundary value problem
\begin{equation}\label{eq_w}
\begin{cases}
    \operatorname{div} \tilde{\mathcal{A}}(y, w, \nabla w) + \tilde{\mathcal{B}}(y, w, \nabla w) =0 & \text{in } B_1^+, \\
    w = \tilde{\psi}(y) & \text{on } T_1,
\end{cases}
\end{equation}
where $\tilde{\psi}(y) := \psi(y) - \psi(0)$ is the shifted trace, and the new operators are explicitly given by 
$$ \tilde{\mathcal{A}}(y, z, p) := \mathcal{A}(y, z + \psi(0), p) \quad \text{and} \quad \tilde{\mathcal{B}}(y, z, p) := \mathcal{B}(y, z + \psi(0), p). $$
Because $\psi(0)$ is uniformly bounded (since $\|u\|_{L^\infty(B_1^+)} \le M$), $\tilde{\mathcal{A}}$ and $\tilde{\mathcal{B}}$ strictly preserve the original structural conditions \eqref{structural conditions} with the same ellipticity constants $\lambda$ and $\Lambda$. In particular, the oscillation condition is naturally preserved around the new frozen operator $\tilde{\mathcal{A}}(0,0,p) = \mathcal{A}(0, \psi(0), p)$:
\begin{equation*}
    |\tilde{\mathcal{A}}(y, z, p) - \tilde{\mathcal{A}}(0, 0, p)| = |\mathcal{A}(y, z+\psi(0), p) - \mathcal{A}(0, \psi(0), p)| \le \Lambda_{\alpha_{\mathcal{A}}} \big(1 + g(|p|)\big) (|y|^{\alpha_{\mathcal{A}}} + |z|^{\alpha_{\mathcal{A}}}).
\end{equation*}

By construction, $\tilde{\psi}(0) = 0$ and $\nabla \tilde{\psi}(0) = \nabla \psi(0) = 0$. Moreover, its H\"older norm is controlled by the original data, yielding $$\|\tilde{\psi}\|_{C^{1,\alpha}(T_1)} \le 2 \|\psi\|_{C^{1,\alpha}(T_1)} < 2\varepsilon_\star.$$

Given $0<\eta<1$, by Lemma \ref{singular_boundary_nonhomo} applied to the shifted function $w$ and the continuous embedding of $L^G\hookrightarrow L^{1+\delta}$ \cite[Theorem 2.2]{martinez2008minimum}, there exists $h \in W^{1,G}(B_{3/4}^+) \cap L^\infty(B_{3/4}^+)$ solving the constant coefficient problem
\begin{equation*}
    \begin{cases}
        \operatorname{div} \mathcal{A}_0(\nabla h)=0 &\text{in}\ \ B_{3/4}^+,\\
        h=0 &\text{on}\ \ T_{3/4},\\
        \nabla h(0)=0,
    \end{cases}
\end{equation*}
where $\mathcal{A}_0(p) = \tilde{\mathcal{A}}(0, 0, p) = \mathcal{A}(0, \psi(0), p)$, satisfying
\begin{equation*}
	\left(\fint_{B_{3/4}^+} |w - h|^{1+\delta}\,dy\right)^\frac{1}{1+\delta} < \eta.
\end{equation*}

Since $\mathcal{A}_0(p)$ satisfies \eqref{homogeneous structural conditions}, by the boundary regularity theory in Section \ref{section5}, there exist a maximal H\"older exponent $\alpha_H\in(0,1)$, $\alpha_H=\alpha_H(n,\delta,g_0,\lambda,\Lambda)$, and a universal constant $C>0$ such that
$$|h(y)| = |h(y) - h(0)| \leq C|y|^{1+\alpha_H}\ \ \ \ \text{for all}\ \ \ \ y\in B_{1/2}^+\hbox{.}$$ 

We now estimate, for $0 < {r} \le \frac{1}{2}$,
\begin{eqnarray*}
    \fint_{B_{r}^+}|w(y)|^{1+\delta}\,dy &\leq& C\left[\fint_{B_{r}^+}|w(y)-h(y)|^{1+\delta}\,dy + \fint_{B_{r}^+}|h(y)|^{1+\delta}\,dy\right]\\
	&<& C\left[\eta^{1+\delta}{r}^{-n} + \fint_{B_{r}^+}|y|^{(1+\alpha_H)(1+\delta)}\,dy\right]\\
	&<& C\left[\eta^{1+\delta}{r}^{-n} + {r}^{(1+\alpha_H)(1+\delta)}\right].	
\end{eqnarray*}

Thus, given $0<\alpha<\alpha_H$, we need
$${r}^{(1+\alpha_H)(1+\delta)}<\frac{1}{2C} {r}^{(1+\alpha)(1+\delta)}\hbox{.}$$
Then, it is sufficient to take ${r}>0$ such that
$$0<{r}<\min\left\{\frac{1}{2},\left(\frac{1}{2C}\right)^\frac{1}{(\alpha_H-\alpha)(1+\delta)}\right\}=:{r}_\star.$$
On the other hand, taking
$$0<\eta=\left(\frac{1}{2C}\right)^{\frac{1}{1+\delta}}{r}^{\frac{n}{1+\delta}+1+\alpha}$$
and recalling that $w(y) = u(y) - \psi(0)$, we obtain \eqref{key_boundary}.
\end{proof}

\begin{theorem}[\bf{Small regime up to the boundary}]\label{small_regime_boundary}
Let $\alpha_H \in (0, 1)$ be the maximal boundary H\"older exponent from the homogeneous reference problem, and define the critical exponent
\begin{equation*}
    \alpha_\star := \min\left\{\alpha_H, \frac{q-n}{qg_0}, \frac{\alpha_{\mathcal{A}}}{g_0}\right\}.
\end{equation*}
For any $\alpha \in(0,\alpha_\star )$, there exist positive constants $\varepsilon_\star, r_\star \in (0,1)$ and $C>0$ depending only on $n$, $s$, $q$, $g_0$, $\delta$, $\lambda$, $\Lambda$, $g(1)$, $\alpha_{\mathcal{A}}$, and $\alpha$ such that if $u \in W^{1,G}(B_1^+) \cap L^\infty(B_1^+)$ is a weak solution of \eqref{1''} with $\|u\|_{L^\infty(B_1^+)} \le 1$, $\nabla \psi(0) = 0$, $0\in\bar{\Sigma}(u)$, and
\begin{equation}\label{smallness_conditions_bdry}
	|\nabla u(0)| + \Lambda_{\alpha_{\mathcal{A}}} + \Lambda_B + \|\gamma\|_{L^s(B_1^+)} + \|f\|_{L^q(B_1^+)} + \|\psi\|_{C^{1,\alpha}(T_1)} < \varepsilon_\star,
\end{equation}
then, setting the rigid constant $u(0) = \psi(0)$, there holds
$$ \sup_{y \in B_{r}^+(0)}|u(y)-\psi(0)|\le Cr^{1+\alpha}, \ \ \ \ \text{for all } 0<r< r_\star.$$

\noindent Furthermore, if the origin is additionally a point where the Hardy-Littlewood maximal function of $|f|^q$ is finite, i.e.,
\begin{equation*}
    \mathcal{M}|f|^q(0) := \sup_{s > 0} \fint_{B_s^+(0)} |f(y)|^q \, dy < \infty,
\end{equation*}
the $L^q$-integrability restriction on the H\"older exponent is removed. In this case, the pointwise estimate holds for any optimal exponent $\alpha \in \left(0, \min\left\{\alpha_H, \frac{1}{g_0}, \frac{\alpha_{\mathcal{A}}}{g_0}\right\}\right)$.
\end{theorem}

\begin{proof} 
Given $0<\alpha<\alpha_\star$, by Lemma \ref{Lemma2_boundary} there exist small universal constants $\varepsilon_\star>0$ and ${r}_\star>0$ such that if $\|f\|_{L^q(B_1^+)}<\varepsilon_\star$, $\|\gamma\|_{L^s(B_1^+)} < \varepsilon_\star$, $\Lambda_{\alpha_{\mathcal{A}}} < \varepsilon_\star$, $\Lambda_B < \varepsilon_\star$, and $\|\psi\|_{C^{1,\alpha}(T_1)} < \varepsilon_\star$, then
\begin{equation}\label{case1_bdry}
	\fint_{B_{{r}}^+}|u(y)-\psi(0)|^{1+\delta}\,dy\leq{r}^{(1+\delta)(1+\alpha)}.
\end{equation}
   
We will show by finite induction that for every $k \ge 1$, the rigid anchoring to the boundary data persists:
\begin{equation}\label{6_bdry}
	\fint_{B_{{r}^k}^+}|u(y)-\psi(0)|^{1+\delta}\,dy\leq{r}^{k(1+\delta)(1+\alpha)},
\end{equation}
where $r$ is the small universal radius from Lemma \ref{Lemma2_boundary}. 

We argue by finite induction. The case $k=1$ is \eqref{case1_bdry}. Suppose that we have verified \eqref{6_bdry} for $k$. Define the scaled function $u_k: B_1^+ \to\mathbb{R}$ by expanding the half-ball around the boundary point $0$:
$$u_k(y):=\frac{u({r}^ky)-\psi(0)}{{r}^{k (1+\alpha)}}\hbox{.}$$
Since the domain is a half-ball and the origin is on the flat boundary, the scaling maps $B_1^+$ into $B_1^+$ and preserves the flat boundary portion $T_1$.

Then $u_k$ is a weak solution of 
$$\operatorname{div} \mathcal{A}_{k}(y, u_k, \nabla u_k) + \mathcal{B}_{k}(y, u_k, \nabla u_k) = f_k(y) \quad\hbox{in}\quad B_1^+,$$
where the scaled operators and functions are:
$$\mathcal{A}_{k}(y, z, p) = \frac{\mathcal{A}(r^k y, \psi(0) + r^{k(1+\alpha)}z, r^{k\alpha} p)}{g(r^{k\alpha})},$$
$$\mathcal{B}_{k}(y, z, p) = \frac{r^k \mathcal{B}(r^k y, \psi(0) + r^{k(1+\alpha)}z, r^{k\alpha} p)}{g(r^{k\alpha})},$$
with the associated $N$-function $G_k(t) = \int_0^t \frac{g(r^{k\alpha}s)}{g(r^{k\alpha})} ds$.

Notice that $G(r^{k\alpha} |p|) = r^{k\alpha} g(r^{k\alpha}) G_k(|p|)$ and $g(r^{k\alpha} |p|) = g(r^{k\alpha}) g_k(|p|)$. Therefore, the lower-order term $\mathcal{B}_{k}$ and the oscillation of the principal part $\mathcal{A}_{k}$ satisfy the structural bounds:
\begin{align*}
     |\mathcal{B}_{k}(y, z, p)| &\le \frac{r^k}{g(r^{k\alpha})} \left[ \Lambda_B G(r^{k\alpha}|p|) + |\gamma(r^k y)| g(r^{k\alpha}|p|) + |f(r^k y)| \right] \\
     &= \Lambda_{\mathcal{B}_{k}} G_k(|p|) + |\gamma_k(y)| g_k(|p|) + |f_k(y)|,
\end{align*}
and
\begin{align*}
    |\mathcal{A}_{k}(y, z, p) - \mathcal{A}_{k}(0,0,p)| &\le \frac{\Lambda_{\alpha_{\mathcal{A}}} \big(1 + g(r^{k\alpha}|p|)\big)}{g(r^{k\alpha})} \left( r^{k\alpha_{\mathcal{A}}}|y|^{\alpha_{\mathcal{A}}} + r^{k(1+\alpha)\alpha_{\mathcal{A}}}|z|^{\alpha_{\mathcal{A}}} \right) \\
    &\le \Lambda_{\alpha_{\mathcal{A}}} \left( \frac{r^{k\alpha_{\mathcal{A}}}}{g(r^{k\alpha})} + r^{k\alpha_{\mathcal{A}}} g_k(|p|) \right) \left( |y|^{\alpha_{\mathcal{A}}} + r^{k\alpha\alpha_{\mathcal{A}}}|z|^{\alpha_{\mathcal{A}}} \right) \\
    &\le \Lambda_{\alpha_k} \big( 1 + g_k(|p|) \big) \left( |y|^{\alpha_{\mathcal{A}}} + |z|^{\alpha_{\mathcal{A}}} \right),
\end{align*}
where the new structural constants, drift coefficient, and source term are given by:
$$ \Lambda_{\mathcal{B}_{k}} = \Lambda_B r^{k(1+\alpha)}, \ \Lambda_{\alpha_k} = \Lambda_{\alpha_{\mathcal{A}}} \max\left\{ \frac{r^{k\alpha_{\mathcal{A}}}}{g(r^{k\alpha})}, r^{k\alpha_{\mathcal{A}}} \right\}, \  \gamma_k(y) = r^k \gamma(r^k y), \  \hbox{and}\ f_k(y)=\frac{{r}^{k} f({r}^k y)}{g({r}^{k\alpha})}. $$

For the scaled drift term $\gamma_k$, performing the change of variables $z = r^k y$ and noting that $s > \frac{n(1+g_0)}{1+\delta} \ge n$ implies $1 - \frac{n}{s} > 0$, we immediately obtain:
$$ \|\gamma_k\|_{L^s(B_1^+)} = r^{k\left(1 - \frac{n}{s}\right)} \|\gamma\|_{L^s(B_{r^k}^+)} \le \|\gamma\|_{L^s(B_1^+)} < \varepsilon_\star. $$

Using \textit{(g-1)} of Lemma \ref{firstLema}, we also have
$$\frac{r^{k\alpha_{\mathcal{A}}}}{g(r^{k\alpha})} \le \frac{1}{g(1)} r^{k(\alpha_{\mathcal{A}} - \alpha g_0)}.$$
Thus, we must require $\alpha_{\mathcal{A}} - \alpha g_0 \ge 0$, which yields $\alpha \le \frac{\alpha_{\mathcal{A}}}{g_0}$. 
Under this restriction, we have 
$$\Lambda_{\mathcal{B}_{k}} \le \Lambda_B < \varepsilon_\star \quad \text{and} \quad \Lambda_{\alpha_k} \le \tilde{C} \Lambda_{\alpha_{\mathcal{A}}} < \varepsilon_\star.$$ 
Making these choices and using \textit{(g-1)} again, we estimate $f_k$ by
\begin{eqnarray*}
    \|f_k\|_{L^q(B_1^+)}&\leq& \frac{{r}^{k-k\frac{n}{q}}}{g({r}^{k\alpha})}\|f\|_{L^q(B_1^+)}\leq \frac{{r}^{-k\alpha g_0}{r}^{k\left(1-\frac{n}{q}\right)}}{g(1)}\|f\|_{L^q(B_1^+)}=\frac{{r}^{k\left[\left(1-\frac{n}{q}\right)-\alpha g_0\right]}}{g(1)}\|f\|_{L^q(B_1^+)}\hbox{.}
\end{eqnarray*}
We require that 
$$1-\frac{n}{q}- \alpha g_0\geq0\iff \alpha\leq \frac{q-n}{q g_0}\hbox{.}$$
Thus we need to consider 
$$0<\alpha<\min\left\{\alpha_H, \frac{q-n}{q g_0}, \frac{\alpha_{\mathcal{A}}}{g_0}\right\}.$$
Therefore, for $k$ large we obtain
$$\|f_k\|_{L^q(B_1^+)}\leq \varepsilon_\star\hbox{.}$$

Furthermore, if $\mathcal{M}|f|^q(0)<\infty$, the change of variables $z = r^k y$ yields:
\begin{align*}
    \|f_k\|_{L^q(B_1^+)} &= \frac{r^k}{g(r^{k\alpha})} \left( \int_{B_1^+} |f(r^k y)|^q dy \right)^{\frac{1}{q}} \\
    &= \frac{r^k \cdot r^{-k\frac{n}{q}}}{g(r^{k\alpha})} \left( \int_{B_{r^k}^+(0)} |f(z)|^q dz \right)^{\frac{1}{q}} \\
    &= \frac{r^k}{g(r^{k\alpha})} |B_1^+|^{\frac{1}{q}} \left( \fint_{B_{r^k}^+(0)} |f(z)|^q dz \right)^{\frac{1}{q}} \\
    &\le r^{k(1-\alpha g_0)} \left[ \frac{|B_1^+|^{\frac{1}{q}}}{g(1)} \right] \Big[\mathcal{M}|f|^q(0)\Big]^{\frac{1}{q}}.
\end{align*}
Since $r < 1$, ensuring that this term vanishes as $k \to \infty$ requires strictly $1 - \alpha g_0 > 0$. Thus, it is sufficient to have $0<\alpha < \min\left\{\alpha_H,\frac{1}{g_0}, \frac{\alpha_{\mathcal{A}}}{g_0}\right\}$.

Furthermore, the scaled boundary condition for $u_k$ on $T_1$ is:
$$ \psi_k(y) := \frac{\psi(r^k y) - \psi(0)}{r^{k(1+\alpha)}}. $$
Because $\psi \in C^{1,\alpha}(T_1)$ and we assumed $\nabla \psi(0) = 0$, Taylor's expansion gives 
$$|\psi(r^k y) - \psi(0)| \le \|\psi\|_{C^{1,\alpha}(T_1)} r^{k(1+\alpha)} |y|^{1+\alpha}.$$ 
Therefore, $$\|\psi_k\|_{C^{1,\alpha}(T_1)} \le \|\psi\|_{C^{1,\alpha}(T_1)} < \varepsilon_\star,\ \ \text{and}\ \ \psi_k(0) = 0.$$

We have shown that $u_k$ is strictly under the hypotheses of Lemma \ref{Lemma2_boundary}, which assures that
$$ \fint_{B_{{r}}^+}\left|\frac{u({r}^{k+1}y)-\psi(0)}{{r}^{k (1+\alpha)}}\right|^{1+\delta}\,dy = \fint_{B_{{r}}^+}|u_k(ry)|^{1+\delta}\,dy \leq {r}^{(1+\delta)(1+\alpha)}, $$
and we get the inductive step:
$$ \fint_{B_{{r}^{k+1}}^+}\left|u(y) - \psi(0)\right|^{1+\delta}\,dy \leq {r}^{(k+1) (1+\alpha)(1+\delta)}\hbox{.} $$
This completes the induction. The rest of the proof follows the same lines of Section 7 of \cite{Teixeira2014}.
\end{proof}

\begin{proof}[\bf{Proof of Theorem \ref{main}}]
Set $\alpha_0=\min\{\alpha_\psi,\alpha_\Omega\}$. Let $x_0 \in \partial\Omega \cap \Sigma(u)$. By regularity of $\partial\Omega$, there exists a local $C^{1,\alpha_0}$ diffeomorphism $\Phi$ that flattens a neighborhood of $x_0$ into a half-ball $B_R^+$, mapping the critical  point $x_0$ to the origin $0 \in T_R$. The transformed solution $v(y) = u(\Phi^{-1}(y))$ is a weak solution to the corresponding quasilinear equation
$$ \operatorname{div} \tilde{\mathcal{A}}(y, v, \nabla v) + \tilde{\mathcal{B}}(y, v, \nabla v) = 0 \quad \text{in } B_R^+, $$
where, setting $J_\Phi(y) := |\det D\Phi(\Phi^{-1}(y))|^{-1}$, the transformed operators are explicitly defined for $(y, z, p) \in B_R^+ \times \mathbb{R} \times \mathbb{R}^n$ by
\begin{align*}
    \tilde{\mathcal{A}}(y, z, p) &:= J_\Phi(y) \, D\Phi(\Phi^{-1}(y)) \, \mathcal{A}\!\left(\Phi^{-1}(y), \, z, \, D\Phi(\Phi^{-1}(y))^T p\right), \\
    \tilde{\mathcal{B}}(y, z, p) &:= J_\Phi(y) \, \mathcal{B}\!\left(\Phi^{-1}(y), \, z, \, D\Phi(\Phi^{-1}(y))^T p\right).
\end{align*}

Lemma \ref{lem:invariance} ensures that the structural conditions \eqref{structural conditions} are invariant under this $C^{1,\alpha_0}$ change of variables: $\tilde{\mathcal{A}}$ and $\tilde{\mathcal{B}}$ satisfy \eqref{structural conditions} in $B_R^+$ with updated structural constants $\tilde{\lambda}, \tilde{\Lambda}, \Lambda_{\tilde{B}},$ and $\Lambda_{\alpha_{\tilde{\mathcal{A}}}}$, H\"older exponent $\alpha_0$, and transformed lower-order weights
$$ \tilde{\gamma}(y) := C_\Phi \, \gamma(\Phi^{-1}(y)) \in L^s(B_R^+) \quad \text{and} \quad \tilde{f}(y) := J_\Phi(y) \, f(\Phi^{-1}(y)) \in L^q(B_R^+), $$
for a structural constant $C_\Phi > 0$ depending only on $\|D\Phi\|_{L^\infty}$, $\|D\Phi^{-1}\|_{L^\infty}$, and $g_0$. Because $\Phi$ is bi-Lipschitz, the $L^s$ and $L^q$ norms as well as local averages are preserved up to a universal geometric constant, ensuring that $\mathcal{M}|\tilde{f}|^q(0) < \infty$ whenever the maximal bound holds for $f$ at $x_0$. The boundary datum on the flat portion $T_R$ is naturally given by the composition $\tilde{\psi} = \psi \circ \Phi^{-1} \in C^{1,\alpha_0}(T_R)$.

Since $x_0$ is a critical  point of $u$ and the diffeomorphism $\Phi$ is regular, the chain rule $\nabla v(0) = \big(D\Phi(x_0)\big)^{-T} \nabla u(x_0)$ guarantees that $\nabla v(0) = 0$. Because $\tilde{\psi}$ is the trace of $v$ on $T_R$, we strictly have $v(0) = \tilde{\psi}(0)$ and $\nabla_T \tilde{\psi}(0) = 0$.

To invoke Theorem \ref{small_regime_boundary}, we first fix the amplitude divisor $\sigma$. Let
$$ \sigma_0 := \min\left\{ 1, \, \frac{\varepsilon_\star}{5\Lambda_{\tilde{B}}}, \, \left(\frac{\varepsilon_\star}{5 {\Lambda}_{{\alpha}_{\tilde{\mathcal{A}}}} (1/g(1) + 1)}\right)^{\frac{1}{\alpha_0}} \right\}. $$
Since $v$ is continuous at the origin, we may decrease $R$ (this does not affect $\Phi$ nor the structural constants $\Lambda_{\tilde{B}}$, $\Lambda_{\alpha_{\tilde{\mathcal{A}}}}$) so that
$$ m := \sup_{B_R^+}|v - v(0)| \le \frac{\sigma_0}{2}. $$
If $m = 0$, then $v \equiv v(0)$ in $B_R^+$ and the conclusion is trivial, so we assume $m>0$ and set
$$ \sigma := 2m = 2\sup_{B_R^+}|v - v(0)| \in (0, \sigma_0]. $$

For scaling parameters $\varrho \in (0,1]$ (with $\varrho < R$), we define the normalized function:
$$ w(y) := \frac{v(\varrho y) - v(0)}{\sigma}, \quad y \in B_1^+. $$

Then $w$ is a weak solution of the rescaled boundary problem
$$ \begin{cases}\operatorname{div}\mathcal{A}_\star(y, w, \nabla w) +\mathcal{B}_\star(y, w, \nabla w) = 0 \quad &\text{in} \quad B_1^+, \\
w=\psi_\star \quad&\text{on}\quad T_1,
\end{cases}
$$
where the scaled operators, drift coefficient, and source term are given by:
$$\mathcal{A}_\star(y, z, p) = \frac{\tilde{\mathcal{A}}\left(\varrho y, v(0) + \sigma z, \frac{\sigma}{\varrho} p\right)}{g\left(\frac{\sigma}{\varrho}\right)}, \quad\mathcal{B}_\star(y, z, p) = \frac{\varrho \tilde{\mathcal{B}}\left(\varrho y, v(0) + \sigma z, \frac{\sigma}{\varrho} p\right)}{g\left(\frac{\sigma}{\varrho}\right)}, $$
$$ \gamma_\star(y) = \varrho \tilde{\gamma}(\varrho y), \qquad f_\star(y) = \frac{\varrho \tilde{f}(\varrho y)}{g(\sigma/\varrho)}, $$
and the transformed boundary datum on $T_1$ is given by ${\psi}_\star(y) = \frac{\tilde{\psi}(\varrho y) - \tilde{\psi}(0)}{\sigma}$. Here, $\mathcal{A}_\star$ and $\mathcal{B}_\star$ are governed by the rescaled growth density $g_\star$ and its associated $N$-function $G_\star$, explicitly defined for $t \ge 0$ by
$$ g_\star(t) := \frac{g\left(\frac{\sigma}{\varrho}t\right)}{g\left(\frac{\sigma}{\varrho}\right)} \quad \text{and} \quad G_\star(t) := \int_0^t g_\star(s)\,ds = \frac{\varrho}{\sigma g\left(\frac{\sigma}{\varrho}\right)} G\left(\frac{\sigma}{\varrho}t\right), $$
which preserves condition \eqref{01} with the exact same structural exponents $\delta$ and $g_0$.

Consequently, the lower-order term satisfies 
$$|\mathcal{B}_\star(y, z, p)| \le \Lambda_{B_\star} G_\star(|p|) + |\gamma_\star(y)| g_\star(|p|) + |f_\star(y)|,$$
where the new structural constants (for $\varrho \le \sigma \le 1$) are bounded by:
$$ \Lambda_{B_\star} \le \sigma \Lambda_{\tilde{B}} \quad \text{and} \quad \Lambda_{{\alpha}_{\mathcal{A_\star}}} \le {\Lambda}_{{\alpha}_{\tilde{\mathcal{A}}}} \left(\frac{1}{g(1)} + 1\right) \max\{\varrho^{\alpha_0}, \sigma^{\alpha_0}\}. $$
Since $\sigma \le \sigma_0$, the choice of $\sigma_0$ immediately ensures $\Lambda_{B_\star} \le \varepsilon_\star/5$ and $\Lambda_{\alpha_\star} \le \varepsilon_\star/5$.

Furthermore, the uniform $L^\infty$ bound required by Theorem \ref{small_regime_boundary} is satisfied with $M_w = 1$: indeed, since $\sigma = 2\sup_{B_R^+}|v - v(0)|$,
$$ \|w\|_{L^\infty(B_1^+)} \le \frac{\sup_{B_R^+}|v - v(0)|}{\sigma} = \frac{1}{2} \le 1 =: M_w. $$

Simultaneously, since $\tilde{\psi} \in C^{1,\alpha_0}(T_R)$ with $\nabla_T \tilde{\psi}(0) = 0$, Taylor's expansion implies that the scaled boundary datum satisfies 
$$\|\psi_\star\|_{C^{1,\alpha}(T_1)} \le \frac{C \varrho^{1+\alpha_0 - \alpha}}{\sigma} \|\tilde{\psi}\|_{C^{1,\alpha_0}(T_R)}.$$ 
For the scaled drift term $\gamma_\star$ and source term $f_\star$, since $\varrho \le \sigma$ implies $g(\sigma/\varrho) \ge g(1)$ and $s > \frac{n(1+g_0)}{1+\delta} \ge n$, we apply the bounds:
$$\|\gamma_\star\|_{L^s(B_1^+)} \le \varrho^{1-n/s} \|\tilde{\gamma}\|_{L^s(B_R^+)} \quad \text{and} \quad \|f_\star\|_{L^q(B_1^+)} \le \frac{\varrho^{1-n/q}}{g(1)} \|\tilde{f}\|_{L^q(B_R^+)}.$$

We explicitly define the geometric radius $\varrho$ by embedding their respective norms and powers into a single strict bound:
$$ \varrho := \min\left\{ 1, \, \frac{R}{2}, \, \sigma, \, \left( \frac{\sigma \varepsilon_\star}{5 C \|\tilde{\psi}\|_{C^{1,\alpha_0}(T_R)}} \right)^{\frac{1}{1+\alpha_0-\alpha}}, \, \left( \frac{\varepsilon_\star}{5 \|\tilde{\gamma}\|_{L^s(B_R^+)}} \right)^{\frac{s}{s-n}}, \, \left( \frac{\varepsilon_\star g(1)}{5 \|\tilde{f}\|_{L^q(B_R^+)}} \right)^{\frac{q}{q-n}} \right\}. $$

Alternatively, if $\mathcal{M}|\tilde{f}|^q(0) < \infty$, the dimensional restriction $q > n$ is bypassed. Evaluating the scaled source term via the local average directly yields 
$$\|f_\star\|_{L^q(B_1^+)} \le \frac{\varrho}{g(1)} |B_1^+|^{\frac{1}{q}} \left[\mathcal{M}|\tilde{f}|^q(0)\right]^{\frac{1}{q}}.$$
The last dimensional term in the definition of $\varrho$ is then dynamically replaced by the sufficiently small radius 
$$\varrho \le \frac{\varepsilon_\star g(1)}{5 |B_1^+|^{1/q} [\mathcal{M}|\tilde{f}|^q(0)]^{1/q}},$$
which relies purely on the linear scaling $\varrho \to 0$.

Under these explicit parameter selections, we have 
$$\Lambda_{B_\star} + \Lambda_{\alpha_\star} + \|\psi_\star\|_{C^{1,\alpha}(T_1)} + \|\gamma_\star\|_{L^s(B_1^+)} + \|f_\star\|_{L^q(B_1^+)} \le \frac{4}{5}\varepsilon_\star < \varepsilon_\star,$$
and since $\nabla w(0) = 0$, $w$ completely satisfies all hypotheses of Theorem \ref{small_regime_boundary}. Consequently, there exists a universal radius $r_\star$ such that for $0 < \tau < r_\star$,
$$ \sup_{B_\tau^+(0)}|w(y)-w(0)|\leq C \tau^{1+\alpha}. $$

We now unravel the scaling $w(y) = \frac{v(\varrho y) - v(0)}{\sigma}$ to return to the flattened physical variables. Setting $z = \varrho y$ and $r = \varrho \tau$, for all $0 < r < \varrho r_\star$, we obtain:
$$ \sup_{B_r^+(0)} |v(z) - v(0)| \le \left( \frac{\sigma C}{\varrho^{1+\alpha}} \right) r^{1+\alpha} = C r^{1+\alpha}. $$
Since $\nabla v(0) = 0$, this pointwise decay estimate proves that $v$ is pointwise $C^{1,\alpha}$ at the origin in the half-ball. 
    
Finally, we return to the original physical domain by recalling that $u(x) = v(\Phi(x))$ and $\Phi(x_0) = 0$. Since the local boundary flattening transformation $\Phi$ is a $C^{1,\alpha_0}$ diffeomorphism, it is bi-Lipschitz; in particular, setting $L_\Phi := \|D\Phi\|_{L^\infty}$, we have $|\Phi(x)| = |\Phi(x) - \Phi(x_0)| \le L_\Phi |x - x_0|$. Consequently, for every radius $r \in (0, \tilde{r}_\star)$, where $\tilde{r}_\star := \frac{\varrho r_\star}{L_\Phi}$, the image of $B_r(x_0) \cap \Omega$ under $\Phi$ is contained in $B_{L_\Phi r}^+(0)$ with $L_\Phi r < \varrho r_\star$. Thus,
$$ \sup_{x \in B_r(x_0) \cap \Omega} |u(x) - u(x_0)| \le \sup_{z \in B_{L_\Phi r}^+(0)} |v(z) - v(0)| \le C (L_\Phi r)^{1+\alpha} = C_0 r^{1+\alpha}, \quad \text{for all } 0 < r < \tilde{r}_\star. $$

\end{proof}
\appendix

\section{On the $n < s \le \max\left\{1+g_0, \frac{n(1+g_0)}{1+\delta}\right\}$}\label{app:subcritical_regimes}

In our variational perturbation framework, the structural restriction 
$$s > s_\star := \max\left\{1+g_0, \, \frac{n(1+g_0)}{1+\delta}\right\}$$ 
arises directly from the simultaneous need to apply H\"older's inequality and the Gagliardo--Nirenberg interpolation inequality when estimating the drift integral. Notice that when $g_0=\delta=p-1$ we return the same conditions in $p$-Laplacian case i.e. $s_\star=\max\{n,p\}$. 

In the subcritical range $n < s \le s_\star$ (when $g_0 > \delta$), at least one of these inequalities fails; Table \ref{tab:subcritical_regimes} summarizes the exact analytical obstruction in each regime and the alternative structural or integrability hypotheses on the data that suffice to recover the proof.

\begin{table}[htbp]
\centering
\caption{Diagnostic of the variational method and sufficient hypotheses in the regimes $s > n$.}
\label{tab:subcritical_regimes}
\renewcommand{\arraystretch}{1.35}
\setlength{\tabcolsep}{3.5pt}
\scriptsize
\begin{tabular}{|p{2.4cm}|p{1.5cm}|p{2.6cm}|p{3.2cm}|p{4.2cm}|}
\hline
\textbf{Regime of $n,\delta,g_0$} & \textbf{Value of $s_\star$} & \textbf{Range of $s$} & \textbf{Analytical Obstruction} & \textbf{Sufficient Hypotheses to Recover the Proof (any option suffices)} \\ \hline\hline

\textbf{Case 1:} \newline $n \le 1+\delta \le 1+g_0$ 
& $1+g_0$ 
& $(n, \, 1+g_0]$ 
& Sobolev holds ($1+\delta \ge n$), but Young/H\"older fails ($\sigma = \frac{s}{1+g_0} \le 1 \Rightarrow G(|\gamma|) \notin L^1(\Omega)$). 
& $\bullet$ \textbf{In $\mathcal{B}$ structure:} replace $|\gamma|g(|p|)$ by $|\gamma|G(|p|)^{1-\frac{1}{s}}$ or $|\gamma|G(|p|)^{\frac{\delta}{1+\delta}}$. \\ \hline

\textbf{Case 2:} \newline $1+\delta < 1+g_0 \le n$ 
& $\dfrac{n(1+g_0)}{1+\delta}$ 
& $\left(n, \, \dfrac{n(1+g_0)}{1+\delta}\right]$ 
& Young/H\"older holds ($\sigma > 1$), but Sobolev/GN fails ($\theta \le 0$, $\mu_\star \le 0$ in $W^{1,1+\delta}$). 
& $\bullet$ \textbf{In $\mathcal{B}$ structure:} replace $|\gamma|g(|p|)$ by $|\gamma|G(|p|)^{\frac{\delta}{1+\delta}}$. \newline
  $\bullet$ \textbf{In $\gamma$ space:} $\gamma \in L^{s,\lambda}(\Omega)$ with $\lambda > n - \frac{s(1+\delta)}{1+g_0}$. \\ \hline

\textbf{Case 3:} \newline $1+\delta < n < 1+g_0$ 
& $\dfrac{n(1+g_0)}{1+\delta}$ 
& \textbf{3A:} $(n, \, 1+g_0]$ \newline\newline 
  \textbf{3B:} $\left(1+g_0, \, \dfrac{n(1+g_0)}{1+\delta}\right]$ 
& \textbf{In 3A:} Double failure ($\sigma \le 1 \Rightarrow G(|\gamma|) \notin L^1(\Omega)$ and $\mu_\star < 0$). \newline\newline
  \textbf{In 3B:} Young/H\"older holds ($\sigma > 1$), but Sobolev/GN fails ($\mu_\star \le 0$). 
& \textbf{For 3A:} \newline
  $\bullet$ \textbf{In $\mathcal{B}$ structure:} replace $|\gamma|g(|p|)$ by $|\gamma|G(|p|)^{\frac{\delta}{1+\delta}}$. \newline
  \textbf{For 3B:} \newline
  $\bullet$ \textbf{In $\mathcal{B}$ structure:} replace $|\gamma|g(|p|)$ by $|\gamma|G(|p|)^{\frac{\delta}{1+\delta}}$. \newline
  $\bullet$ \textbf{In $\gamma$ space:} $\gamma \in L^{s,\lambda}(\Omega)$, $\lambda > n - \frac{s(1+\delta)}{1+g_0}$. \\ \hline
\end{tabular}
\end{table}

In Table \ref{tab:subcritical_regimes}, ``{GN}" denotes the Gagliardo--Nirenberg interpolation inequality applied to $w = G(v)^{\frac{1}{1+\delta}} \in W^{1,1+\delta}$ with weight $\theta := 1 - \frac{n(1+g_0)}{s(1+\delta)}$, $\sigma := \frac{s}{1+g_0}$ is the H\"older exponent used to control the upper Orlicz growth of $G(|\gamma|v)$, $\mu_\star := (1+\delta)\theta = 1+\delta - \frac{n(1+g_0)}{s}$ is the net scaling exponent of the drift energy, and $L^{s,\lambda}(\Omega)$ denotes the classical Morrey space endowed with the norm $\|\gamma\|_{L^{s,\lambda}(\Omega)} := \sup_{x_0 \in \overline{\Omega}, \, \rho > 0} \left( \rho^{-\lambda} \int_{B_\rho(x_0) \cap \Omega} |\gamma(x)|^s \, dx \right)^{1/s}$.


\end{document}